\documentclass[a4paper, 12pt]{amsart}

\usepackage{graphicx} 
\usepackage{amssymb}
\usepackage{amsmath}
\usepackage{amsthm}
\usepackage{amscd}

\usepackage{color}

\title
[Collapsing three-dimensional Alexandrov spaces]
{Collapsing three-dimensional closed Alexandrov spaces with a lower curvature bound}
\author{Ayato Mitsuishi}
\author{Takao Yamaguchi}
\address
{Mathematical Institute, Tohoku University, Sendai 980-8578, JAPAN}
\address
{Institute of Mathematics, University of Tsukuba, Tsukuba 305-8571, JAPAN}
\email[A.~Mitsuishi]{mitsuishi@math.tohoku.ac.jp}
\email[T.~Yamaguchi]{takao@math.tsukuba.ac.jp}

\theoremstyle{plain}
\newtheorem{theorem}{Theorem}[section]
\newtheorem{lemma}[theorem]{Lemma}
\newtheorem{corollary}[theorem]{Corollary}
\newtheorem{proposition}[theorem]{Proposition}
\newtheorem{definition}[theorem]{Definition}
\newtheorem{remark}[theorem]{Remark}
\newtheorem{problem}[theorem]{Problem}
\newtheorem{assertion}[theorem]{Assertion}
\newtheorem{example}[theorem]{Example}

\newtheorem{assumption}[theorem]{Assumption}
\newtheorem{conjecture}[theorem]{Conjecture}

\makeatletter

 \@addtoreset{equation}{section}
\makeatother

\newcommand{\wangle}[0]{\tilde{\angle}}
\newcommand{\Mo}[0]{\mathrm{M\ddot{o}}}
\newcommand{\diam}[0]{\mathrm{diam}\,}
\newcommand{\rad}[0]{\mathrm{rad}\,}
\newcommand{\e}[0]{\varepsilon}
\newcommand{\de}[0]{\delta}

\newcommand{\pt}[0]{\mathrm{pt}}
\newcommand{\dist}[0]{\mathrm{dist}}
\newcommand{\reg}[0]{\mathrm{reg}}

\newcommand{\strrad}[0]{\text{-}\mathrm{str.\, rad}\,}
\newcommand{\ess}[0]{\mathrm{Ess}}

\begin{document}
\begin{abstract}
In the present paper, we determine 
the topologies of three-dimensional closed Alexandrov spaces 
which converge to lower dimensional spaces in the Gromov-Hausdorff topology. 
\end{abstract}
\maketitle
\tableofcontents

\section{Introduction}

The purpose of the present paper is to determine the topologies of 
collapsing three-dimensional Alexandrov spaces.

Alexandrov  spaces are complete length spaces with 
the notion of curvature bounds. 
In this paper, we deal with finite dimensional Alexandrov spaces 
with a lower curvature bound (see Definition \ref{def of Alex sp}). 
Alexandrov spaces naturally appear in 
convergence and collapsing phenomena of Riemannian manifolds 
with a lower curvature bound (\cite{SY}, \cite{Y 4-dim}), 
and have played important roles 
in the study of collapsing Riemannian manifolds with a lower curvature bound.

For a positive integer $n$, $D > 0$, $\kappa \in \mathbb{R}$, 
let us consider the following two families: 
$\mathcal{M}^n(D,\kappa)$ is the family of all isometry classes 
of complete $n$-dimensional Riemannian manifolds $M$  whose diameters and 
sectional curvatures satisfy $\diam(M) \leq D$ and  $sec(M)\geq \kappa$. 
$\mathcal{A}^n(D,\kappa)$ is the family of all isometry classes 
of $n$-dimensional Alexandrov spaces with $\diam \leq D$ and curvature $\geq \kappa$. 
It follows from the definition of Alexandrov spaces
that $\mathcal{M}^n(D,\kappa) \subset \mathcal{A}^n(D,\kappa)$. 
By Gromov's precompactness theorem, $\mathcal{A}^n(D,\kappa)$ has a nice property that 
$\bigcup_{k \leq n} \mathcal{A}^k(D,\kappa)$ is compact  in the Gromov-Hausdorff
topology, while  $\bigcup_{k \leq n}\mathcal{M}^k(D,\kappa)$ is precompact.
Therefore, it is quite natural to study the convergence and collapsing
phenomena in 
$\mathcal{A}^n (D, \kappa)$.
Thus, the following problem naturally appears:
\begin{problem} \upshape \label{convergence theory}
Let $\{M_i^n\}_{i = 1}^{\infty}$ be a sequence in $\mathcal{A}^n( D,
\kappa)$  converging to  an Alexandrov  space $X$.
Can one describe the topological structure of $M_i$ by using the
geometry and topology of $X$ for large $i$\,?
\end{problem}

In this paper, we consider Problem \ref{convergence theory} for $n =
3$ when $M_i$ has no boundary.
We exhibit previously known results related to Problem \ref{convergence theory}.
Let us fix the following setting:
$M_i := M_i^n \in \mathcal{A}^n(D, \kappa)$ converges to $X$ as $i \to
\infty$, and fix a sufficiently large integer $i$.

If the non-collapsing case arises, i.e. $\dim X = n$,
Perelman's stability theorem \cite{Per Alex II}
(cf. \cite{Kap stab}) shows that $M_i$ is homeomorphic to $X$.

In the collapsing case, we know the following results in 
the general dimension:
If $M_i$ and $X$ are Riemannian manifolds, then Yamaguchi
proved that there is a locally trivial fiber bundle (smooth
submersion) $f_i : M_i \to X$ whose fiber is a quotient of torus by
some finite group action (\cite{Yam collapsing and pinching}, 
\cite{Y  convergence}).
Fukaya and Yamaguchi proved that if $M_i$ are Riemannian manifolds and
$X$ is a single-point set, then $\pi_1(M_i)$  has a  nilpotent
subgroup of finite index \cite{FY}.
This statement also goes through even if $M_i$ is an Alexandrov space
(\cite{Y  convergence}).

In the lower dimensional cases, we know the following conclusive results:
In dimension three,  Shioya and Yamaguchi \cite{SY} gave a complete classification of
three-dimensional closed (orientable) Riemannian manifolds $M_i$ 
collapsing in $\mathcal M^3(D, \kappa)$.
It is also proved that volume collapsed closed orientable
Riemannian three-manifolds $M_i$ with no diameter bound
are graph-manifolds, or have small diameters and finite fundamental
groups (\cite{SY vol collapse}, \cite{Pr:entropy}).
For more recent works, see Morgan and Tian \cite{MT}, Cao and Ge
\cite{CaGe}, Kleiner and Lott \cite{KL}.
In dimension four, 
Yamaguchi \cite{Y 4-dim} gave a classification of four-dimensional
orientable closed Riemannian manifolds $M_i$ collapsing in $\mathcal M^4(D, \kappa)$.

\subsection{Main results}
To state our results, we fix notations in this paper.
$D^n$ is a closed $n$-disk.
$D^1$ is written as $I$, called an (bounded closed) interval.
$P^n$ is an $n$-dimensional real projective space.
$T^n$ is an $n$-dimensional torus.
$K^2$ is a Klein bottle, $\Mo$ is a Mobius band.
$K^2 \tilde{\times} I$ is an orientable (non-trivial) $I$-bundle over $K^2$. 
$K^2 \hat{\times} I$ is a non-orientable non-trivial $I$-bundle over $K^2$. 
A solid Klein bottle $S^1 \tilde{\times} D^2$ is obtained by
$\mathbb{R} \times D^2$ with identification $(t, x) = (t + 1,
\bar{x})$.
Here, we consider $D^2$ as the unit disk on the complex plane and $\bar{x}$ is the complex conjugate of $x$. 
Note that a solid Klein bottle is homeomorphic to $\Mo \times I$.

Let us first provide an important example of a collapsing Alexandrov space $M_{\mathrm{pt}}$ which is not a manifold.
We observe that this space $M_\pt$ can be regard as a ``circle
fibration'' over a cone with a singular interval fiber.

\begin{example} \upshape \label{M_pt}
Let $S^1 \times \mathbb{R}^2$ be a flat manifold with product metric.
For the isometric involution $\alpha$ defined by 
\[
\alpha (e^{i\theta}, x) = (e^{-i\theta}, -x),
\]
we consider the quotient space $M_{\mathrm{pt}} := S^1 \times
\mathbb{R}^2 / \langle \alpha \rangle$ which is an Alexandrov space
with nonnegative curvature.
This space $M_{\pt}$ has the two topologically singular points,
i.e. non-manifold points, $p_+ := [(1, 0)]$ and $p_- := [(-1, 0)]$
which correspond to fixed points $(1, 0)$ and $(-1, 0)$ of $\alpha$. 
We consider a standard projection $p : M_\pt \to \mathbb{R}^2 / x \sim
-x = K(S^1_\pi)$ from $M_\pt$ to the cone $K(S_\pi^1)$ over the circle
$S_\pi^1$ of length $\pi$.
This is an $S^1$-fiber bundle over $K(S^1_\pi)$ except the vertex $o \in K(S^1_\pi)$.
Remark that the fiber $p^{-1}(\partial B(o, r))$ over a metric circle at $o$ is topologically a Klein bottle. 
The fiber $p^{-1}(o)$ over the origin is an interval joining the
topologically singular points $p_+$ and $p_-$.
Thus, we may regard $M_\pt$ as a circle fibration, with the singular
fiber $p^{-1}(o)$, over the cone $K(S^1_\pi)$.
We rescale the ``circle orbits'' of $M_\pt$ as $M_\pt(\e) := (\e S^1)
\times \mathbb{R}^2 / \langle \alpha \rangle$. 
Then, as $\e \to 0$, $M_\pt(\e)$ collapse to the cone $K(S^1_\pi)$.
\end{example}

We obtain the following results.

A compact Alexandrov space is called {\it closed} if it has no  boundary.
An {\it essential singular point} of an Alexandrov space is a point at which the space of directions 
has radius not greater than $\pi /2$.
\begin{theorem} \label{2-dim interior}
Let $M^3_i$ be a sequence of three-dimensional closed Alexandrov
spaces with curvature $\geq -1$ and $\diam M_i \leq D$.
Suppose that $M_i$ converges to an Alexandrov surface $X$ without boundary.
Then, for sufficiently large $i$, $M_i$ is homeomorphic to a generalized Seifert fiber space over $X$.
Further, singular orbits may occur over essential singular points in $X$.%
\end{theorem}

Here, a {\it generalized} Seifert fiber space is a Seifert fiber space in a generalized 
sense, which possibly has singular interval fibers just as in Example \ref{M_pt}. 
For the precise definition, see Definition \ref{generalized Seifert fiber space}.

To describe the topologies of $M_i^3$ converging to an Alexandrov surface with nonempty boundary, we define the notion of generalized solid tori and generalized solid Klein bottles.
Let $K(A)$ be the cone over a topological space $A$, obtained from $A \times [0, +\infty)$ smashing $A \times \{0\}$ to a point.
Let $K_1(A)$ be the closed cone over $A$, obtained from $A \times [0, 1]$ smashing $A \times \{0\}$ to a point.
We put $\partial K_1(A) := A \times \{1\}$.

\begin{definition} \label{generalized solid} \upshape
We will construct a certain three-dimensional topological orbifold
whose boundary is homeomorphic to a torus or a Klein bottle.

We first observe that the closed cone $K_1(P^2)$ over $P^2$ 
can be regarded as a ``fibration''\footnote{In fact, it is NOT a Serre fibration, 
because the fibers $D^2$ and $\Mo$ are not weak homotopy equivalence.} 
over $I$ as follows.
Let $\Gamma \cong \mathbb{Z}_2$ be the group generated by the involution $\gamma$ on $\mathbb{R}^3$ defined by $\gamma(v) = -v$.
Then $\mathbb{R}^3 / \Gamma = K(P^2)$.

We consider the following families of surfaces in $\mathbb{R}^3$,
\begin{align*}
A(t) &:= \{ v = (x,y,z) \,|\, 
x^2 + y^2 - z^2 = t^2, |z| \leq 1 \}, \\
B(t) &:= \{ v = (x,y,z) \,|\, 
x^2 + y^2 - z^2 = -t^2, x^2 + y^2 \leq 1 \}.
\end{align*}
and set 
\begin{align*}
D(t) := \left\{
\begin{aligned}
& A(t) / \Gamma \text{ if } t > 0 \\
& B(t) / \Gamma \text{ if } t \leq 0
\end{aligned}
\right.
\end{align*}
Then $D(t)$ is homeomorphic to a Mobius band for $t > 0$, and is homeomorphic to a disk for $t \leq 0$. 
Remark that $\bigcup_{t \in [-1,1]} \partial D(t)$ is homeomorphic to $S^1 \times I$.
The union $D(1) \cup \bigcup_{t \in [-1, 1]} \partial D(t) \cup D(-1)$
corresponds to $P^2 \times \{1\} = \partial K_1(P^2) \subset
K_1(P^2)$. 
Define a projection 
\begin{equation} \label{projection of K_1}
\pi : K_1 (P^2) \approx \bigcup_{t \in [-1, 1]} D(t) \to [-1, 1] \text{ as } \pi(D(t)) = t.
\end{equation}
This is a ``fibration'' stated as above.

For a positive integer $N \geq 1$, let us consider a circle $S^1 = [0, 2 N] / \{0\} \sim \{2 N\}$.
Let $I_j$ be a sub-arc in $S^1$ corresponding to $[j-1, j] \subset [0, 2 N]$ for $j = 1, \dots, 2 N$.
We consider a sequence $B_j$ of topological spaces such that each $B_j$ is homeomorphic to $K_1(P^2)$. 
We take a sequence of projections $\pi_j : B_j \to I_j$ obtained as
above such that there are homeomorphisms $\phi_j : \pi_j^{-1} (j)
\approx \pi_{j+1}^{-1} (j)$ for all $j = 1, \dots, 2 N$.
Then we obtain a topological space $Y := \bigcup_{j = 1}^{2N} B_j$ glued by $\phi_j$'s.
Define a projection 
\begin{equation} \label{projection of gen solid}
\pi : Y = \bigcup_{j = 1}^{2 N} B_j \to S^1 \text{ by } \pi (\pi_j^{-1}(t)) = t
\end{equation}
for any $t \in S^1$.
By the construction, $Y$ has $2 N$ topologically singular points.
Remark that the restriction $\pi|_{\partial Y} : \partial Y \to S^1$ is a usual $S^1$-fiber bundle.
Then we obtain an topological orbifold $Y$ whose boundary $\partial Y$ is
homeomorphic to a torus or a Klein bottle. 
If $\partial Y$ is a torus, then $Y$ is called a {\it generalized solid torus of type} $N$.
If $\partial Y$ is a Klein bottle, then $Y$ is called a {\it generalized solid Klein bottle of type} $N$.
We regard a solid torus $S^1 \times D^2$ and the product $S^1 \times
\Mo$ as generalized solid tori of type $0$. 
We also regard a solid Klein bottle $S^1 \tilde{\times} D^2$ and
non-trivial $\Mo$-bundle $S^1 \tilde{\times} \Mo$ over $S^1$ as
generalized solid Klein bottles of type $0$.
Note that $S^1 \tilde{\times} \Mo$ is homeomorphic to a non-orientable
$I$-bundle $K^2 \hat{\times} I$ over $K^2$.
\end{definition}


For a two-dimensional Alexandrov space $X$, a boundary point $x \in
\partial X$ is called a {\it corner} point if $\diam \Sigma_x \le
\pi$, in other word, if it is an essential singular point.

\begin{theorem} \label{2-dim boundary}
Let $\{ M_i \}_{i =1}^{\infty}$ be a sequence of three-dimensional
closed Alexandrov spaces with curvature $\geq -1$ and $\diam M_i \leq
D$.
Suppose that $M_i$ converges to an Alexandrov surface $X$ with
non-empty boundary.
Then, for large $i$, there exist a generalized Seifert fiber space
$\mathrm{Seif}_i\,(X)$ over $X$ and generalized solid tori or
generalized solid Klein bottles $\pi_{i,k} : Y_{i,k} \to (\partial
X)_k$ over each component $(\partial X)_k$ of $\partial X$ such that
$M_i$ is homeomorphic to a union of $\mathrm{Seif}_i\,(X)$ and
$Y_{i,k}$'s glued along their  boundaries, where the fibers of
$\mathrm{Seif}_i\,(X)$ over a boundary points $x \in (\partial X)_k$
are identified with $\partial \pi_{i,k}^{-1}(x) \approx S^1$.
\end{theorem}

It should be remarked that in Theorem \ref{2-dim boundary},
the fiber of   $\pi_{i,k} : Y_{i,k} \to (\partial
X)_k$  may change at  a corner point of $(\partial X)_k$, and that
The type of $Y_{i, k}$ is less than or equals to the half of the 
number of corner points in $(\partial X)_k$.

\begin{corollary}\label{2-dim boundary corollary}
Under the same assumption and notation of Theorem \ref{2-dim boundary}, 
for large $i$, there exists a continuous surjection $f_i : M_i \to X$
which is a $\theta(i)$-approximation satisfying the following.
\begin{itemize}
\item[(1)] 
$f_i : f_i^{-1}(\mathrm{int}\, X) \to \mathrm{int} X$ is a generalized Siefert fibration.
\item[(2)] 
For $x \in \partial X$, $f_i^{-1} (x)$ is homeomorphic to a one-point set or a circle. 
The fiber of $f_i$ may change over a corner point in $\partial X$.
\item[(3)]
For any collar neighborhood $\varphi : (\partial X)_k \times [0,1] \to
X$ of a component $(\partial X)_k$ of $\partial X$, which contains no
interior essential singular points, 
$f_i^{-1} (\mathrm{image}\, \varphi)$ is a generalized solid torus or a generalized solid Klein bottle.
\end{itemize}
\end{corollary}

Under the same notation of Corollary \ref{2-dim boundary corollary},
we remark that, for $x \in (\partial X)_k$, 
\begin{align*}
f_i^{-1} (\varphi (\{x\} \times [0,1])) &\approx D^2 \text{ if } f_i^{-1}(x) \approx \{\pt\} \text{; and } \\ 
f_i^{-1} (\varphi (\{x\} \times [0,1])) &\approx \Mo \text{ if } f_i^{-1}(x) \approx S^1.
\end{align*}

The structure of $M_i$ collapsing to one-dimensional space is determined as follows.

\begin{theorem} \label{1-dim circle}
Let $M^3_i$ be a sequence of three-dimensional closed Alexandrov spaces with curvature $\geq -1$ and $\diam M_i \leq D$.
Suppose that $M_i^3$ converges to a circle.
Then, for large $i$, $M_i$ is homeomorphic to a total space of an $F_i$-fiber bundle over $S^1$, 
where the fiber $F_i$ is homeomorphic to 
one of $S^2$, $P^2$, $T^2$ and $K^2$.
\end{theorem}

To describe the structures of $M_i$ converging to an interval $I$, we prepare certain topological orbifolds.
First, we provide
\[
B(\pt) := S^1 \times D^2/\langle \alpha \rangle
\]
Here, the involution $\alpha$ is the restriction of one provided in Example \ref{M_pt}.
Remark that $\partial B(\pt) \approx S^2$.
We also need  to consider three-dimensional open Alexandrov spaces $L_2$ and  $L_4$
with two-dimensional souls $S_2$ and $S_4$ respectively,
where $S_2$ (resp. $S_4$) is homeomorpshic to $S^2$ or $P^2$ (resp. to $S^2$).
For their definition, see Example \ref{L_i}.
The space $L_i$ $(i=2,4)$ has $i$ topologically singular points, which are contained 
in $S_i$. We denote by $B(S_i)$  a metric ball around $S_i$ in $L_i$. Here we point out that
$\partial B(S_2) \approx S^2$  (resp.$\approx K^2$)  if $S_2 \approx S^2$ (resp.
if $S_2 \approx P^2$), and $\partial B(S_4) \approx T^2$.

\begin{theorem} \label{1-dim interval}
Let $M^3_i$ be a sequence of three-dimensional closed Alexandrov spaces with curvature $\geq -1$ and $\diam M_i \leq D$.
Suppose that $M_i^3$ converges to an interval.
Then, for large $i$, $M_i$ is the union of $B_i \cup B_i'$ glued along
their boundaries. $\partial B_i$ is homeomorphic to one of $S^2$,
$P^2$, $T^2$ and $K^2$.
The topologies of $B_i$ (and $B_i'$) are determined as follows:
\begin{itemize}
\item[(1)]
If $\partial B_i \approx S^2$, then $B_i$ is homeomorphic to one of
$D^3$, $P^3 - \mathrm{int}\, D^3$, $B(S_2)$ with $S_2 \approx S^2$.
\item[(2)]
If $\partial B_i \approx P^2$, then $B_i$ is homeomorphic to $K_1(P^2)$. 
\item[(3)]
If $\partial B_i \approx T^2$, then $B_i$ is homeomorphic to one of 
$S^1 \times D^2$, 
$S^1 \times \Mo$, 
$K^2 \tilde{\times} I$, and 
$B(S_4)$.
\item[(4)]
If $\partial B_i \approx K^2$, then $B_i$ is homeomorphic to one of 
$S^1 \tilde{\times} D^2$, 
$K^2 \hat{\times} I$, 
$B(\mathrm{pt})$, and 
$B(S_2)$ with $S_2 \approx P^2$.
\end{itemize}
\end{theorem}

\begin{corollary} \label{0-dim}
Let $M_i$ be a sequence of three-dimensional closed Alexandrov space with curvature $\geq -1$ and diameter $\leq D$.
Suppose $M_i$ converges to a point.
Then, for large $i$, $M_i$ is homeomorphic to one of 
\begin{itemize}
\item
generalized Seifert fiber spaces in the conclusion of Theorem \ref{2-dim interior} with a base Alexandrov surface having nonnegative curvature,
\item
spaces in the conclusion of Theorem \ref{2-dim boundary} with a base Alexandrov surface having nonnegative curvature,
\item
spaces in the conclusion of Theorem \ref{1-dim circle} and \ref{1-dim interval}, and
\item 
closed Alexandrov spaces with nonnegative curvature having finite fundamental groups.
\end{itemize}
\end{corollary}

We remark that all spaces appeared in the conclusion of 
Theorems \ref{2-dim interior}, \ref{2-dim boundary}, \ref{1-dim
  circle} and \ref{1-dim interval} 
and Corollary \ref{0-dim} actually have sequences of  metrics as
Alexandrov spaces collapsing
to such respective limit spaces descrived there.

By Corollary \ref{0-dim}, to achieve a complete classification of the
topologies of collapsing three-dimensional closed Alexandrov spaces, 
we provide a version of ``Poincare conjecture'' for three-dimensional
closed Alexandrov spaces
 with nonnegative curvature.

 For Alexandrov spaces $A$ and $A'$ having  boundaries isometric to each other,
$A \cup_\partial A'$ denotes the gluing of $A \cup A'$ via an
isometry 
$\phi : \partial A \to \partial A'$. Note that $A \cup_\partial A'$ is
an  Alexandrov space (see \cite{Pet Appl})


\begin{conjecture} \label{Poincare conjecture}
A simply connected three-dimensional closed Alexandrov space with nonnegative curvature is homeomorphic to 
an isometric gluing $A \cup_\partial A'$ for $A$ and $A'$ chosen in
the following list \eqref{Poincare conjecture list} of nonnegatively
curved Alexandrov spaces:
\begin{equation} \label{Poincare conjecture list}
\begin{aligned}
& D^3, K_1(P^2), B(\pt ), B(S_2), B(S_4). 
\end{aligned}
\end{equation}
\end{conjecture}

We also remark that any connected sum of those spaces admits a metric of Alexandrov space 
having a lower curvature bound by some constant. 

\begin{conjecture} \label{positive Poincare conjecture}
A simply connected three-dimensional closed Alexandrov space with curvature $\geq 1$ is homeomorphic to 
a three-sphere $S^3$ or a suspension $\Sigma(P^2)$ over $P^2$. 
\end{conjecture}


The organization of this paper and 
basic ideas of the proofs of our results are as follows: 

In Section \ref{Preliminaries}, we review 
some basic notations and results on Alexandrov spaces. 
We provide a three-dimensional topological orbifold having a circle fiber
structure with singular arc fibers, and call it a generalized Seifert
fiber space.
At the end of this section, we prove fundamental properties
on the topologically  singular point set.

In Section \ref{proof of flow theorem}, 
for any $n \in \mathbb{N}$, 
we consider $n$-dimensional closed Alexandrov spaces $M_i^n$
collapsing to a space $X^{n-1}$ of co-dimension one.
Assume that all points in $X$ are almost regular, 
except finite points $x_1, \dots, x_m$.  
For any fixed $p \in \{x_{\alpha}\}$, 
we take a sequence $p_i \in M_i$ converging to $p$.
By Yamaguchi's Fibration Theorem \ref{fibration theorem}, 
for large $i$, there is a fiber bundle $\pi_i : A_i \to A$, 
where $A$ is a small metric annulus $A= A(p; r, R)$ around $p$ 
and $A_i$ is a some corresponding domain.
Here, $r$ and $R$ are small positive numbers so that $r \ll R$.

Although $A_i$ is not a metric annulus in general,  
it is expected that $A_i$ is homeomorphic to an standard annulus $A(p_i; r, R)$.
Moreover, we may expect that 
there exist an isotopy $\phi: M_i \times [0,1] \to M_i$ such that, 
putting $\phi_t := \phi(\cdot, t)$, 
\begin{equation}\label{isotopy intro}
\left\{
\begin{aligned}
&\phi_0 = id_{M_i},  \\
&\phi_1 \bigg( B \bigg(p_i, \frac{r + R}{2}\bigg) \cup A_i \bigg) 
= B(p_i, R), \text{ and }\\
&\phi_1 (x) = x \text{ if } x \not\in B(p_i, R+\de)
\end{aligned}
\right.
\end{equation}
for any fixed $\de > 0$.

If we consider the case that 
all $M_i$ are Riemannian manifolds, 
then we can obtain a smooth flow $\Phi_t$
of a gradient-like vector field $V$ of the distance function $\dist_{p_i}$ from $p_i$. 
Then, by using integral curves of $V$, 
we can obtain such an isotopy $\phi$ from $id_{M_i}$
satisfying the property \eqref{isotopy intro}. 

We will prove that 
such an argument of flow goes through 
on Alexandrov spaces $M_i$ as well.
To do this, we first prove a main result, Flow Theorem \ref{flow theorem},  in this section. 
Theorem \ref{flow theorem} implies the existence of an integral flow
$\Phi_t$ of a gradient-like vector field of a distance function
$\dist_{p_i}$ on $A(p_i; r, R)$ in a suitable sense. 
This flow leads an isotopy 
$\phi$ satisfying the property \eqref{isotopy intro}.
Theorem \ref{flow theorem} is important 
throughout the paper.

\vspace{1em}
In Sections \ref{proof of 2-dim interior} -- \ref{proof of 0-dim}, 
we prove Theorems \ref{2-dim interior} -- \ref{1-dim interval} and
Corollary \ref{0-dim}.
To explain the arguments used in those proofs, 
let us fix a sequence $M_i = M_i^3$ of 
three-dimensional closed Alexandrov spaces  
in $\mathcal{A}^3(-1, D)$ converging to $X$ of dimension $\leq 2$.

In Section \ref{proof of 2-dim interior}, 
we consider the case that $\dim X = 2$ and $\partial X = \emptyset$.
Let $p_1, \dots, p_m$ be all $\de$-singular points in $X$ for a fixed small $\de > 0$.
Let us take a converging sequence $p_{i,\alpha} \to p_\alpha$ ($i \to \infty$) 
for each $\alpha = 1, \dots, m$.
Let us fix any $\alpha$ and set $p := p_{\alpha}$, $p_i := p_{i,
  \alpha}$. We take $r = r_p >0$ such that all points in $B(p, 2 r) -
\{p\}$ are $(2, \e)$-strained. Then, all points in an annulus $A(p_i;
\e_i, 2 r-\e_i)$ are $(3, \theta(i, \e))$-strained.
Here, $\e_i$ is a sequence of positive numbers converging to zero.
Then, by Fibration Theorem \ref{fibration theorem}, 
we have an $S^1$-fiber bundle $\pi_i : A_i \to A(p; r, 2r)$.
On the other hand, by the rescaling argument \ref{rescaling argument}, 
we obtain the conclusion that 
$B_i := B(p_i, r)$ is homeomorphic to a solid torus or $B(\pt)$.
Here, we can exclude the possibility that 
$B_i$ is topologically a solid Klein bottle. 
Theorem \ref{flow theorem} implies that 
there exists an isotopy carrying 
the fiber $\pi_i^{-1}(\partial B(p, r))$ to $\partial B_i$.
If $B_i \approx S^1 \times D^2$ then we can prove an  argument similar to  \cite{SY} that $B_i$ has 
the structure of a Seifert fibered torus in the usual sense, extending $\pi$. 
If $B_i \approx B(\pt)$, then by some new observation on the topological structure of $B(\pt)$,  
we can prove that $B_i$ has the standard ``circle fibration'' structure provided in Example \ref{M_pt},
compatible with $\pi$.
In this way, we obtain the structure of a generalized Seifert fiber space on $M_i$.

In Section \ref{proof of 2-dim boundary}, 
we consider the case that $\dim X = 2$ and $\partial X \neq \emptyset$.
Take a decomposition of $\partial X$ to 
connected components $\bigcup_{\beta} (\partial X)_{\beta}$.
Put $X_0 := X - U(\partial X, r)$ for some small $r >0$.
By Theorem \ref{2-dim interior}, 
we have a generalized Seifert fibration $\pi_i : M_{i, 0} \to X_0$ 
for some closed domain $M_{i, 0} \subset M_i$.
For any fixed $\beta$, we take points $p_{\alpha}$ in $(\partial X)_{\beta}$
so fine that $\{p_\alpha\}$ contains all $\e$-singular points in $(\partial X)_{\beta}$.
Let $p_{i, \alpha}\in M_i$ be a sequence converging to $p_\alpha$.
Deform a metric ball $B(p_{i, \alpha}, r)$ to 
a neighborhood $B_{i, \alpha}$ of $p_{i, \alpha}$ 
by an isotopy obtained in Theorem \ref{flow theorem}.
Because of the existence of $\partial X$, 
we need a bit complicated construction 
of flows of gradient-like vector fields of distance functions.

In Section \ref{proof of 1-dim circle}, 
we consider the case that $X$ is isometric to a circle $S^1(\ell)$ of length $\ell$. 
If $M_i$ has no $\e$-singular points, 
by Fibration Theorem \ref{fibration theorem}, 
we obviously obtain the conclusion of Theorem \ref{1-dim circle}.
But, in general, $M_i$ has $\e$-singular points.
Therefore, we use Perelman's Morse theory to construct a fibration over $S^1$.

In Section \ref{proof of 1-dim interval},
we consider the case that $X$ is isometric to 
an interval $[0, \ell]$ of some length $\ell$.
We use rescaling arguments around the end points of interval $X$ and an argument  similar to 
Theorem \ref{1-dim circle} to prove Theorem \ref{1-dim interval}.

In Section \ref{proof of 0-dim}, 
we consider the case of  $\dim X = 0$ 
and prove Corollary \ref{0-dim}.

\par
\medskip
For three-dimensional Alexandrov spaces with non-empty boundary
collapsing to lower dimensional spaces,
considering their doubles, one could make use of the results in the
present paper to obtain the structure of collapsing in that case. 
This will appear in a forthcoming paper.


\section{Preliminaries} \label{Preliminaries}

\subsection{Definitions, Conventions and Notations}

In the present paper, we use the following notations. 
\begin{itemize}
\item $\theta(\de)$ is a function depending on 
$\de = (\de_1, \dots, \de_k)$
such that $\lim_{\de \to 0} \theta(\de) = 0$.
$\theta(i, \de)$ is a function depending on 
$\de \in \mathbb{R}^k$ and $i \in \mathbb{N}$ 
such that $\lim_{i \to \infty, \de \to 0} \theta(i, \de) = 0$.
When we write $A < \theta(\de)$ for a nonnegative number $A$, 
we always assume that $\theta(\de)$ is taken to be nonnegative. 

\item $X \approx Y$ means that $X$ is homeomorphic to $Y$.
For metric spaces $X$ and $Y$, $X \equiv Y$ means that $X$ is isometric to $Y$.

\item For metric spaces $X$ and $Y$, the direct product $X \times Y$ has the product metric if nothing stated.

\item For continuous mappings $f_1 : X_1 \to Y$, $f_2 : X_2 \to Y$ and $g : X_1 \to X_2$, we say that {\it $g$ represents $f_1$ and $f_2$} if $f_1 = f_2 \circ g$ holds.

\item Denote by $d(x,y)$, $|x,y|$, and $|x y|$ 
the distance between $x$ and $y$ in a metric space $X$.
Sometimes we mark $X$ as lower index $|x,y|_X$.

\item For a subset $S$ of a topological space, 
$\overline{S}$ is the closure of $S$ in the whole space. 

\item For a metric space $X = (X, d)$ and $r > 0$, 
denote the rescaling metric space $r X = (X, r d)$. 

\item For a subset $Y$ of a metric space, 
denote by $\dist_{Y}$ the distance function from $Y$.
When $Y = \{x\}$ we denote $\dist_x := \dist_{\{x\}}$.
For a subset $Y$ of a metric space $X$ and a subset $I$ of $\mathbb{R}_+$, 
define a subset $B(Y; I) := B_X(Y; I) := \dist_Y^{-1} (I) \subset X$.
For special cases, we denote and call those sets in the following way: 
$B(Y, r) := B(Y; [0,r])$ the closed ball, $U(Y, r) := B(Y; [0,r))$ the open ball, 
$A(Y; r', r) := B(Y; [r', r])$ the annulus, 
and $\partial B(Y, r) := B(Y; \{r\})$ the metric sphere.
For $Y = \{x\}$, we set $B(x, r) := B(\{x\}, r)$, 
$U(x,r) := U(\{x\}, r)$ and $A(x; r', r) := A(\{x\}; r', r)$.

\item For a topological space $X$, 
the {\it cone} $K(X)$ over $X$ is obtained from 
$X \times [0, \infty)$ by smashing $X \times \{0\}$ to a point.
An equivalent class $[(x,a)] \in K(X)$ of 
$(x,a) \in X \times [0, +\infty)$ is denoted by $ax$, 
or often simply written by $(x,a)$.
A special point $(x, 0) = 0 x \in K(X)$ is denoted by $o$ or $o_X$, 
called the {\it origin} of $K(X)$. 
A point $v \in K(X)$ is often called a vector.
$K_1(X)$ denotes the (unit) closed cone over $X$, i.e.
\[
K_1(X) := \{ ax \in K(X)  \,|\, x \in X, 0 \leq a \leq 1\}.
\]
$K_1(X)$ is homeomorphic to the join between $X$ and a single-point. 

\item For a metric space $X$, 
$K(X)$ often denotes the {\it Euclidean metric cone}, 
which is equipped the metric as follows:
for two points $(x_1,r_1), (x_2, r_2) \in X \times [0, \infty)$
the distance between them is defined by 
\[
d((x_1,r_1), (x_2, r_2))^2 := r_1^2 + r_2^2 -2 r_1 r_2 \cos \min \{d(x_1, x_2), \pi\}.
\]

And for $v \in K(X)$,
we put $|v| := d(x, o)$ and call it the {\it norm} of $v$.
Define an {\it inner product} $\langle v, w \rangle$ of $v, w \in K(X)$ by
$\langle v, w \rangle := |v| |w| \cos \angle vow$.


\item When we write $M^n$ marked upper index $n$, 
this means that $M$ is an $n$-dimensional Alexandrov space.
\end{itemize}

For a curve $\gamma: [0, 1] \to X$ in a metric space $X$, 
the {\it length} $L(\gamma)$ of $\gamma$ is defined by 
\[
L(\gamma) := 
\sup_{ 0 = t_0 < t_1 < \cdots < t_m = 1} 
\sum_{i = 1}^{m} d(\gamma(t_{i-1}), \gamma(t_i)) \in [0, +\infty].
\]
A metric space $X$ is called a {\it length space} if 
for any $x$, $y \in X$ and $\e > 0$, there exists 
a curve $\gamma : [0,1] \to X$ such that 
$\gamma(0) = x$, $\gamma(1) = y$ and $0 \leq L(\gamma) - d(x,y) \leq \e$.
A curve is called a {\it geodesic} if it is an isometric embedding from some interval 
Sometime a geodesic $\gamma$ defined on 
a bounded closed interval $[0, \ell]$ is called a geodesic {\it segment}. 
A geodesic defined on $\mathbb{R}$ is called a {\it line}, 
a geodesic defined on $[0, +\infty)$ is called a {\it ray}.
For a geodesic $\gamma : I \to X$ in a metric space $X$, 
we often regard $\gamma$ itself as 
the subset $\gamma(I) \subset X$.

\subsection{Alexandrov spaces}

From now on, throughout this paper, 
we always assume that a metric space is {\it proper}, 
namely, any closed bounded subset is compact.
A proper length space is a {\it geodesic space}, namely 
any two points are jointed by a geodesic. 

For three points $x_0$, $x_1$, $x_2$ in a metric space, the {\it size} of $(x_0, x_1, x_2)$ is $\mathrm{size}\,(x_0, x_1, x_2) := |x_0 x_1| + |x_1 x_2| + |x_2x_0|$.
The {\it size} of four points $(x_0; x_1, x_2, x_3)$ (centered at $x_0$) is defined by the maximum of $\mathrm{size}\, (x_0, x_i, x_j)$ for $1 \leq i \neq j \leq 3$, denoted by $\mathrm{size}\, (x_0; x_1, x_2, x_3)$. 

\begin{definition}\upshape \label{def of comparison angle}
For three points $x_0, x_1, x_2$ in a metric space $X$ with 
$\mathrm{size}\, (x_0, x_1,x_2) < 2 \pi/ \sqrt{\kappa}$, 
the $\kappa$-{\it comparison angle} of $(x_0; x_1, x_2)$, 
written by $\wangle_{\kappa} x_1 x_0 x_2$ or $\wangle_{\kappa} (x_0; x_1, x_2)$, 
is defined as following:
Take three points $\tilde{x}_i$ ($i = 0, 1, 2$) in $\kappa$-plane $\mathbb M^2_{\kappa}$, 
which is a simply connected complete surface with constant curvature $= \kappa$,
such that $d(x_i, x_j) = d(\tilde{x}_i, \tilde{x}_j)$ for $0 \leq i, j \leq 2$
and put $\wangle_{\kappa} x_1 x_0 x_2 := 
\angle \tilde{x}_1\tilde{x}_0\tilde{x}_2$.
Sometime we write $\wangle$ omitting $\kappa$ in the notation $\wangle_{\kappa}$. 
\end{definition}

\begin{definition}\upshape \label{def of Alex sp}
For $\kappa \in \mathbb{R}$, 
a complete metric space $X$ is called 
an {\it Alexandrov space with curvature} $\geq \kappa$ 
if $X$ is a length space and, 
for every four points $x_0, x_1, x_2, x_3 \in X$ 
(with $\mathrm{size}\, (x_0; x_1, x_2, x_3) < 2 \pi/\sqrt{\kappa}$ if $\kappa > 0$), 
we have the next inequality:
\[
\wangle_{\kappa} x_1x_0x_2 +
\wangle_{\kappa} x_2x_0x_3 +
\wangle_{\kappa} x_3x_0x_1 \leq 2 \pi.
\]
\end{definition}

The {\it dimension} of an Alexandrov space means its Hausdorff dimension. 
The Hausdorff dimension and the topological dimension 
are equal to each other (\cite{BGP}, \cite{PP QG}, \cite{Plaut}).
Throughout this paper, we always assume that an Alexandrov space is finite dimensional.
\begin{remark}
If $X$ is an Alexandrov space with curvature $\geq \kappa$, 
then the rescaling space $r X$ is an Alexandrov space with curvature $\geq \kappa/r^2$.
\end{remark}

For two geodesics $\alpha, \beta : [0, \e] \to X$ 
emanating at $\alpha(0) = \beta(0) =p \in X$ in an Alexandrov space $X$, 
the {\it angle} $\angle(\alpha, \beta)$ at $p$ is defined by 
\[
\angle(\alpha, \beta) := 
\angle_p(\alpha, \beta) :=
\lim_{s, t \to 0} \wangle(p; \alpha(t), \beta(s)).
\]
The set of all non-trivial geodesics emanating at $p$ 
in an Alexandrov space $X$ is denoted by $\Sigma_p' X$. 
The angle $\angle_p$ at $p$ satisfies the triangle inequality on this set.
Its metric completion is denoted by $\Sigma_p = \Sigma_p X$, 
called the {\it space of directions} at $p$.
For a geodesic $\gamma : [0, \ell] \to X$ starting from $x = \gamma(0)$ to $y = \gamma(\ell)$, 
we denote $\gamma^{+}(0) = \gamma'(0) = \gamma_x' = \gamma_x^+ = \uparrow_x^y$ the direction of $\gamma$ at $x$.
By $xy$, we denote some segment $xy = \gamma : [0, |xy|] \to X$ 
joining from $\gamma (0) =x$ to $\gamma(|xy|) = y$.
For a subset $A \subset X$, the closure of a set of all directions from $x$ to $A$ 
is denoted by $A_x'$, i.e., 
\[
A_x' := \{ \xi \in \Sigma_x \,|\,
\exists a_i \in A \text{ such that } \lim_{i \to \infty} |x a_i| = |x, A| 
\text{ and } \lim_{i \to \infty} \uparrow_x^{a_i} = \xi \}.
\]
When $x \in A$, we put $\Sigma_x(A) := A_x'$.
For $x, y \in X$, we denote as $y_x' := \{y\}_x'$.
Or sometimes we denote by $y_x'$ an element belong with $y_x'$. 
For $x \in X$ and $y, z \in X -\{x\}$, 
we denote by $\angle yxz$ the angle $\angle (xy, xz) = \angle (\uparrow_x^y, \uparrow_x^z)$ between some fixed segments $xy$, $xz$.

\begin{definition}\label{strainer} \upshape
A {\it $(k, \de)$-strainer} at $x \in M$ is a collection of points 
$\{p_{\alpha}^{\pm}\}_{\alpha = 1}^{k} 
= \{p_{\alpha}^+, p_{\alpha}^- \,|\, \alpha = 1, \dots, k \}$
satisfying the following.
\begin{align}
\wangle p_{\alpha}^+ x p_{\beta}^+ &> \pi / 2 - \de \\
\wangle p_{\alpha}^+ x p_{\beta}^- &> \pi / 2 - \de \\
\wangle p_{\alpha}^- x p_{\beta}^- &> \pi / 2 - \de \\
\wangle p_{\alpha}^+ x p_{\alpha}^- &> \pi - \de
\end{align}
for all $1 \leq \alpha \neq \beta \leq k$.

The {\it length} of a strainer $\{p_{\alpha}^{\pm}\}$ at $x$ is 
$\min_{1 \leq \alpha \leq k} \{ |p_{\alpha}^+, x|, |p_{\alpha}^-, x| \}$.
The {\it $(k, \de)$-strained radius} of $x$, 
denoted by $(k,\de)\text{-}\mathrm{str. rad}\, x$, is 
the supremum of lengths of $(k, \de)$-strainers at $x$.
A $(k, \de)$-strained radius $(k,\de)\text{-}\mathrm{str. rad}\, A$ of a subset $A \subset M$ is defined by
\[
(k,\de)\text{-}\mathrm{str. rad}\, A := \inf_{x \in A}\, (k,\de)\text{-}\mathrm{str. rad}\, x.
\]
If there is a $(k,\de)$-strainer at $x$, 
then $x$ is called $(k ,\de)${\it -strained}. 
Denotes by $R_{k, \de} (M)$ 
the set of all $(k, \de)$-strained points in $M$.
$R_{k, \de}(M)$ is an open subset.
Put $S_{k, \de} (M) := M - R_{k, \de} (M)$.
Any point in $S_{k, \de}(M)$ is called a $(k,\de)$-{\it singular point}.   
When we consider an $n$-dimensional Alexandrov space $M^n$
and $\de$ is sufficiently small with respect to $1/n$, 
we simply say $\de$-strained, $\de$-singular, etc. 
instead of $(n, \de)$-strained, $(n, \de)$-singular, etc., 
and we omit to write $R_{\de}(M)$, $S_{\de}(M)$ 
instead of $R_{n, \de}(M)$, $S_{n, \de}(M)$.
For an $n$-dimensional Alexandrov space $M^n$, 
put $R(M^n) := \bigcap_{\de > 0} R_{\de} (M^n)$ and 
$S(M^n) := \bigcup_{\de > 0} S_{\de} (M^n) = M^n - R(M^n)$.
\end{definition}

\begin{theorem}[\cite{BGP}, \cite{OS}]
For any $n$-dimensional Alexandrov space $M^n$, 
we have $\dim_{H} S(M)$ $\leq n-1$ and
$\dim_{H} S(M) - \partial M \leq n-2$.
\end{theorem}

Here, the boundary $\partial M$ of an Alexandrov space $M$ 
is defined inductively in the following manner.

\begin{definition}\upshape \label{def of boundary}
A one-dimensional Alexandrov space $M^1$ 
is a manifold, and the boundary of $M^1$ is 
the boundary of $M^1$ as a manifold.
Now let $M^n$ be an $n$-dimensional Alexandrov space with $n > 1$.
A point $p$ in $M^n$ is called a {\it boundary point} if
$\Sigma_p$ has a boundary point.
The set of all boundary points is denoted by 
$\partial M^n$, called the {\it boundary} of $M^n$.
Its complement is denoted by $\mathrm{int}\, M^n  = M^n - \partial M^n$, 
called the {\it interior} of $M^n$.
A point in $\mathrm{int}\, M^n$ is called an {\it interior point} of $M^n$. 
$\partial M^n$ is a closed subset in $M^n$ (\cite{BGP}, \cite{Per Alex II}).

A compact Alexandrov space without boundary is called 
a {\it closed} Alexandrov space, 
and a noncompact Alexandrov space without boundary is called 
an {\it open} Alexandrov space.
\end{definition}

\begin{definition} \upshape
For an $n$-dimensional Alexandrov space $M^n$, 
we say that $p \in M$ is a {\it topologically regular point} (or a {\it manifold-point}) if there is a neighborhood of $p$, which is homeomorphic to $\mathbb R^n$ or $\mathbb R^{n-1} \times [0, \infty)$.
$p$ is called a {\it topologically singular point} if $p$ is not a topologically regular point.
We denote by $S_{\mathrm{top}}(M)$ the set of all topologically singular points.
\end{definition}

\begin{definition}\upshape
For an Alexandrov space $M$, 
a point $p \in M$ is called an {\it essential singular point}
if $\rad \Sigma_p \leq \pi/2$.
A set of whole essential singular points in $M$ is denoted by $\ess(M)$.
We define the set of interior (resp. boundary) essential singular points 
$\ess(\mathrm{int}\,M)$ (resp. $\ess(\partial M)$) as follows:
\begin{align*}
\ess(\mathrm{int}\,M) 
&:= \ess(M) \cap \mathrm{int}\,M, \\ 
\ess(\partial M) 
&:= \ess(M) \cap \partial M.
\end{align*}
Remark that if $\dim M = 1$ then 
$\ess(\mathrm{int}\,M) = \emptyset$ and
$\ess(\partial M) = \partial M$.
\end{definition}

\begin{remark} \upshape
By Theorem \ref{radius sphere theorem} and Stability Theorem \ref{stability theorem},  
we can check the following.
\[
S_{\mathrm{top}}(M) \subset 
\ess(M) \subset 
S(M).
\]
\end{remark}

For small $\de \ll 1/n$, any $(n, \de)$-regular point  
in an $n$-dimensional Alexandrov space $M^n$
is an interior point.

\begin{theorem}[{\cite[Corollary 12.8]{BGP}}]
\label{regular interior}
An $(n-1, \de)$-regular interior point 
in an $n$-dimensional Alexandrov space 
is an $(n, \de')$-regular point.
Here, $\de' \to 0$ as $\de \to 0$.
\end{theorem}

The boundary of an Alexandrov space is 
determined by its topology:

\begin{theorem}[{\cite[Theorem 13.3(a)]{BGP}}, \cite{Per Alex II}] \label{boundary is top inv}
Let $M_1$ and $M_2$ be $n$-dimensional Alexandrov spaces 
with homeomorphism $\phi : M_1 \to M_2$.
Then $\phi(\partial M_1) = \partial M_2$.
\end{theorem}

\subsection{The Gromov-Hausdorff convergence}
For metric spaces $X$ and $Y$, and $\e > 0$, 
an $\e$-{\it approximation} $f$ from $X$ to $Y$
is a map $f: X \to Y$ such that 
\begin{itemize}
\item[(1)]
$|d(x,x') - d(f(x), f(x'))| \le \e$ for any $x, x' \in X$,
\item[(2)]
$Y = B(\mathrm{Image}\, (f), \e)$.
\end{itemize}
The {\it Gromov-Hausdorff distance} $d_{GH}(X,Y)$ 
between $X$ and $Y$ is defined by the infimum of 
those $\e >0$ that  
there exist $\e$-approximations 
from $X$ to $Y$ and from $Y$ to $X$.
We say that a sequence of metric spaces $X_i$, $i = 1, 2, \dots$ 
converges to a metric space $X$ as $i \to \infty$ if
$d_{GH}(X_i, X) \to 0$ as $i \to \infty$.

For two pointed metric spaces $(X,x)$, $(Y,y)$, 
a {\it pointed $\e$-approximation} $f$ from $(X,x)$ to $(Y,y)$ is 
a map $f : B_X(x, 1/\e) \to Y$ such that 
\begin{itemize}
\item[(1)] $f(x) = y$,
\item[(2)] $|d(x',x'') - d(f(x'), f(x''))| \le \e$ for $x', x'' \in B_X(x, 1/\e)$,
\item[(3)] $B_Y(y, 1/\e) \subset B(\mathrm{Image}\, (f), \e )$.
\end{itemize}
The {\it pointed Gromov-Hausdorff distance} $d_{GH}((X,x), (Y,y))$ between 
$(X,x)$ and $(Y,y)$ is defined by 
the infimum of those $\e > 0$ that 
there exist pointed $\e$-approximations
from $(X,x)$ to $(Y,y)$ and from $(Y,y)$ to $(X,x)$. 

For an $n$-dimensional Alexandrov space $X^n$,
the (Gromov-Hausdorff) {\it tangent cone $T_x X$ of $X$ at $x$} is defined by 
the pointed Gromov-Hausdorff limit of $(1/r_i X, x)$ 
for some sequence $(r_i)$ converging to zero.
Thus, $T_x X$ is an $n$-dimensional 
noncompact Alexandrov space with nonnegative curvature.
And, $T_x X$ is isometric to the metric cone $K(\Sigma_x)$ 
over the space of directions $\Sigma_x$. 

For a locally Lipschitz map $f : X \to M$ between Alexandrov spaces, 
and a curve $\gamma : [0,a] \to X$ 
starting at $p = \gamma(0)$ with direction $\gamma^{+}$ at $p$,
We say that $f$ has 
the {\it directional derivative $df(\gamma^{+})$ in the direction $\gamma^{+}$}
if there exists the limit 
\[
df(\gamma^+) := (f \circ \gamma)^{+} := \frac{d}{dt} f \circ \gamma (0+).
\]

A distance function 
on an Alexandrov space has the directional derivative in any direction.

For a local Lipschitz function $f$ on a metric space, 
the {\it absolute gradient} $|\nabla f|_p $ of $f$ at $p$ is defined by
\[
|\nabla f|_p := |\nabla f|(p) :=
\max \bigg\{ 
\limsup_{x \to p} \frac{f(x) - f(p)}{d(x,p)}, 
0 \bigg\}.
\]

\begin{definition} \upshape
$f$ is called {\it regular} at $p$ if $|\nabla f|_p > 0$.
Such a point $p$ is a {\it regular point} for $f$.
Otherwise, $f$ is called {\it critical} at $p$.
\end{definition}

Let $X$ be an Alexandrov space and $U$ be an open subset of $X$.
Let $f : U \to \mathbb R$ be a locally Lipschitz function. 
For $\lambda \in \mathbb R$, $f$ is said to be {\it $\lambda$-concave} if for every segment $\gamma : [0, \ell] \to U$, the function
\[
f \circ \gamma (t) - \frac{\lambda}{2} t^2
\]
is concave in $t$.
A $0$-concave function is said to be concave.
$f$ is said to be {\it semiconcave} if for every $x \in U$ there are an open neighborhood $V$ of $x$ in $U$ and a constant $\lambda \in \mathbb R$ such that $f|_V$ is $\lambda$-concave.

For a semiconcave function $f$ on a finite dimensional Alexandrov space, the gradient vector $\nabla f$ of $f$ is defined in the tangent cone:
\begin{definition}[\cite{PP QG}] \upshape 
Let $X$ be a finite dimensional Alexandrov space.
Let $f : U \to \mathbb{R}$ be a semiconcave function defined on an open neighborhood $U$ of $p$.
A vector $v \in T_p X$ is called the {\it gradient} of $f$ at $p$ if the following holds.
\begin{itemize}
\item[(i)]
For any $w \in T_p X$, we have $d_p f (w) \leq \langle v ,w \rangle$.
\item[(ii)]
$d_p f (v) = |v|^2$.
\end{itemize}
The gradient of $f$ at $p$ is shortly denoted by $\nabla_p f$.
\end{definition}
Remark that $\nabla_p f$ is uniquely determined as the following manner: 
If $|\nabla f|_p = 0$ then $\nabla_p f = o_p$, and otherwise,
\[
\nabla_p f = d_p f (\xi_{\mathrm{max}} ) \xi_{\mathrm{max}} , 
\] 
where $\xi_{\mathrm{max}} \in \Sigma_p$ is the uniquely determined unit vector such that $d_p f(\xi_{\mathrm{max}}) = \max_{\xi \in \Sigma_p} d_p f(\xi)$.

We can show that the absolute gradient $|\nabla f|(p)$ of $f$ is equals to the norm $|\nabla_p f|$ of gradient vector $\nabla_p f$ in $T_p X$.

\subsection{Ultraconvergence}
We will recall the notion of ultrafilters and ultralimits.
For more details, we refer to \cite{BH}.
A (non-principle) \textit{ultrafilter} $\omega$ on the set of natural numbers $\mathbb{N}$ is a finitely additive measure on the power set $2^{\mathbb{N}}$ of $\mathbb{N}$ that has values 0 or 1 and contains no atoms. 
For each sequence $\{y_i\}=\{y_i\}_{i \in \mathbb{N}}$ in a compact Hausdorff space $Y$, an \textit{ultralimit} $\lim_{\omega} y_i = y \in Y$ of this sequence is uniquely determined by the requirement $\omega (\{i \in \mathbb{N} \mid y_i \in U \}) = 1$ for all neighborhood $U$ of $y$.
If $f:Y \to Z$ is a continuous map between topological spaces, then $\lim_{\omega} f(y_i) = f(\lim_{\omega}y_i)$.

For a sequence $\{(X_i,x_i)\}$ of pointed metric spaces, consider the set of all sequence $\{y_i\}$ of points $y_i \in X_i$ with $\lim_{\omega} |x_iy_i| < \infty$. 
And provide the pseudometric $|\{y_i\}\{z_i\}| = \lim_{\omega} |y_iz_i|$ on the set.
The \textit{ultralimit} $(X,x) = \lim_{\omega}(X_i,x_i)$ of $\{(X_i,x_i)\}$ is defined to be the metric space arising from this pseudometric, and the equivalence class of a sequence $\{y_i\}$ is denoted by $(y_i)$.
The ultralimit of a constant sequence $\{(X,x)\}$ of a metric space $(X,x)$ is called the \textit{ultrapower} of $(X,x)$ and is denoted by $X^{\omega} = (X^{\omega},x)$.
The natural map $X \ni y \mapsto (y) = (y,y,y,\dots) \in X^{\omega} $ is an isometric embedding. 

We review a relation between the ultraconvergence and the usual convergence.
A sequence $(\e_i)$ of positive numbers is said to be a {\it scale} if $\lim_{i \to \infty} \e_i = 0$.
\begin{lemma} \label{liminf}
For a real number $A$ and a function $h : \mathbb{R}_+ \to \mathbb{R}$, the following are equivalent:
\begin{itemize}
\item[$(i)$]
$\displaystyle{ \liminf_{t \searrow 0 } h(t) \geq A}. $
\item[$(ii)$]For any scale $(o) = (t_i)$, we have
$\displaystyle{ \lim_{\omega} h(t_i) \geq A}.$
\end{itemize}
\end{lemma}
\begin{proof}
($(i) \Rightarrow (ii)$). We assume $(i)$. Then, for any $\e > 0$, there is $t_0 > 0$ such that 
\[
\inf_{0 < t \leq t_0} h(t) > A - \e.
\]
Let us take any scale $(t_i)$. Then there is $i_0$ such that, for all $i \geq i_0$, we have
\[
h(t_i) \geq \inf_{0 < t \leq t_0} h(t).
\]
Therefore, taking an ultralimit, we have 
\[
\lim_\omega h(t_i) \geq A - \e.
\]
The above inequality holds for all $\e > 0$. Then we obtain $(ii)$.

$((ii) \Rightarrow (i))$. We assume $(ii)$. We take a sequence $(t_i)$ tending to $0$ such that 
\[
\lim_{i \to \infty} h(t_i) = \liminf_{t \searrow 0} h(t).
\]
Then, taking an ultralimit, we obtain $(i)$:
\[
A \leq \lim_\omega h(t_i) = \lim_{i \to \infty} h(t_i) 
= \liminf_{t \searrow 0} h(t).
\]
\end{proof}


Let $(X_i, x_i)$ and $(Y_i, y_i)$ be sequences of pointed metric spaces and let $f_i : (X_i, x_i) \to (Y_i, y_i)$ be a sequence of maps.
Then the {\it ultralimit} $f_\omega = \lim_\omega f_i$ of $\{f_i\}$ is defined by 
\[
\lim_{\omega} X_i \ni a_\omega = (a_i) \mapsto f_\omega (a_\omega) := (f_i (a_i)) \in \lim_\omega Y_i,
\]
if it is well-defined. 
For instance, if $f_i$ is a $L_i$-Lipschitz map with $L_\omega := \lim_\omega L_i < \infty$ then 
the ultralimit $f_\omega$ is well-defined and $L_\omega$-Lipschitz.
If $f_i : (X_i, x_i) \to (Y_i, y_i)$ is a pointed $\tau_i$-approximation with $\tau_\omega := \lim_\omega \tau_i < \infty$, then the ultralimit $f_\omega$ is well-defined and a $\tau_\omega$-approximation.
Remark that if $f_i : (X_i, x_i) \to (Y_i, y_i)$ and $g_i : (Y_i, y_i) \to (Z_i, z_i)$ have the ultralimits $f_\omega := \lim_\omega f_i$ and $g_\omega := \lim_\omega g_i$, then $\lim_\omega (g_i \circ f_i) = g_\omega \circ f_\omega$.
For $a_\omega = (a_i), a_\omega' = (a_i') \in \lim_\omega X_i$, we have $|f_\omega (a_\omega), f_\omega (a_\omega') | = \lim_\omega |f_i(a_i), f_i(a_i')|$.

For a pointed metric space $(X, x)$ and a scale $(o) = (\e_i)$, 
we define the {\it blow-up} $X_x^{(o)} = (X_x^{(o)}, o_x)$ of $(X, x)$ by 
\[
(X_x^{(o)}, o_x) := \lim_{\omega} (1 / \e_i X, x).
\]
For a map $f : (X, x) \to (Y, y)$ between pointed metric spaces, 
we consider a sequence $\{f_i\}$ of maps defined by 
\[
f_i = f : (1 / \e_i X, x) \to (1 / \e_i Y, y).
\]
The {\it blow-up} $f_x^{(o)} : X_x^{(o)} \to Y_y^{(o)}$ of $f$ is defined by $f_x^{(o)} := \lim_\omega f_i$ if it is well-defined.

Let $X$ be an Alexandrov space and $x \in X$,
and let $(o) = (\e_i)$ be a scale.
We consider the exponential map at $x$
\[
\exp_x : 
(\mathrm{dom} (\exp_x), o_x) 
\ni (\gamma, t) \mapsto 
\exp_x(\gamma, t) := \gamma(t) \in (X, x).
\]
Here, $\mathrm{dom} (\exp_x) \subset T_x X$ is the domain of $\exp_x$.
Since $\exp_x$ is locally Lipschitz, the blow-up of $\exp_x$ is well-defined and written by 
\[
\exp_x^{(o)} := (\exp_x)_{o_x}^{(o)} : (T_x X, o_x) \to (X_x^{(o)}, o_x).
\]
The domain of $\exp_x^{(o)}$ is the blow-up of $(\mathrm{dom} (\exp_x), o_x)$, which is identified as $(T_x X, o_x)$.

\begin{lemma}[\cite{Lyt}, \cite{BGP}] \label{exponential map}
\begin{itemize}
Let $(o) = (\e_i)$ be an arbitrary scale.
\item[(i)] Let $X$ be a (possibly infinite dimensional) Alexandrov space.
Then $\exp_x^{(o)}$ is an isometric embedding.
\item[(ii)]
If $X$ be a finite dimensional Alexandrov space, then 
$\exp_x^{(o)} : K(\Sigma_x) \to X_x^{(o)}$ is surjective, for any $x \in X$.
\end{itemize}
\end{lemma}
\begin{proof}
(i) By the definition of the angle between geodesics, 
for any $(\gamma, s)$ and $(\eta, t) \in \Sigma_x' \times [0,\infty)$, we have 
\[
\frac{|\gamma(s \e_i), \eta(t \e_i)|_X}{\e_i}\, \mathop{\longrightarrow}\limits^{i \to \infty}\, 
|s \gamma, t \eta|_{K (\Sigma_x)}.
\]

(ii) By \cite{BGP}, the Gromov-Hausdorff tangent cone $T_x X$ 
and the cone $K(\Sigma_x)$ over space of directions are isometric to each other.
More precisely, the scaled logarithmic map 
\[
\log_x = \exp_x^{-1} : \left( \frac{1}{\e_i} X, x \right) \to \left( \frac{1}{\e_i} T_x X, o_x \right)
\]
is $\tau_i$-approximation for some sequence $\{ \tau_i \}$ of positive numbers converging to zero.
And $\exp_x \circ \log_x = id$. 
Then we have, for each $(x_i) \in X_x^{(o)}$, 
\[
\exp_x^{(o)} (\log_x (x_i) ) = (\exp_x \circ \log_x (x_i)) = (x_i).
\]
Therefore, $\exp_x^{(o)}$ is surjective.
\end{proof}

\subsection{Preliminaries from the geometry of Alexandrov spaces}
In this subsection, we review 
the basic facts on the geometry and topology of Alexandrov spaces.
We refer to mainly \cite{BGP}, \cite{Per Alex II}.

\subsubsection{Local structure around an almost regular point}
Burago, Gromov and Perelman proved 
that a neighborhood of an almost regular point 
is almost isometric to an open subset of Euclidean space.
 
\begin{theorem}[\cite{BGP}, \cite{OS}]
\label{around regular point}
For $n \in \mathbb{N}$, there exists a positive number 
$\de_n >0$ satisfying the following:
Let $X$ be an $n$-dimensional Alexandrov space with curvature $\geq -1$.
For $0 < \de \leq \de_n$, if $x \in X$ is an $(n, \de)$-strained point with 
a strainer $\{p_{\alpha}\}_{\alpha = \pm 1, \dots , \pm n}$ 
of length $\ell$, then the two maps
\begin{align}
\varphi &:= (d(p_{\alpha}, \cdot) )_{\alpha = 1, \dots, n} \\
\tilde{\varphi} &:= 
\biggl( 
\frac{1}{\mathcal{H}^n (B(p_{\alpha}, r))} 
\int_{B(p_{\alpha}, \e )} d(y, \cdot ) d \mathcal{H}^n (y)
\biggr)_{\alpha = 1, \dots , n}
\end{align}
on $B(x, r)$ for small $r > 0$ are both 
$(\theta_n (\de) + \theta_n (r / \ell) )$-almost isometries, 
where $\e$ is so small with $\e \ll r / \ell$.
Here, $\theta_n(\de)$ is a positive function depending on $n$ and $\de$ such that 
$\lim_{\de \to 0} \theta_n(\de) = 0$.
\end{theorem}

\begin{lemma}[{\cite[Lemma 1.8]{Y convergence}}]
\label{lemma 1.8}
Let $M$ be an $n$-dimensional Alexandrov space 
and $\de$ be taken in Theorem \ref{around regular point}.
For any $(n, \de)$-strained point $p \in M$, 
there exists $r > 0$ satisfying the following:
For every $q \in B(p, r/2)$ and $\xi \in \Sigma_q$ 
there exists $x, y \in B(p, r)$ such that 
\begin{align}
|x q| ,|y q| \geq r/4, \\
|x'_q, \xi| \leq \theta(\de, r), \\
\wangle x q y \geq \pi - \theta(\de, r).
\end{align} 
\end{lemma}

\begin{lemma}[{\cite[Lemma 1.9]{Y convergence}}] \label{lemma 1.9}
Let $M$, $p$, $r$ and $\de$ be taken in Lemma \ref{lemma 1.8}.
For every $q \in M$ with $r / 10 \le |pq| \le r$ and for every $x \in M$ with $|px| \ll r$, we have 
\[
|\angle x p q - \wangle x p q | < \theta(\de, r, |p x| / r).
\]
\end{lemma}

\subsubsection{Splitting Theorem}
Splitting theorem is an important tool 
to study the structure of nonnegatively curved spaces.

\begin{theorem}[Splitting Theorem \cite{Milka}] \label{splitting theorem}
Let $X$ be an Alexandrov space of curvature $\geq 0$.
Suppose that there exists a line $\gamma : \mathbb{R} \to X$.
Then there exists an Alexandrov space $Y$ of curvature $\geq 0$ such that 
$X$ is isometric to the product $Y \times \mathbb{R}$.
\end{theorem}

\begin{theorem} \label{maximal diameter}
If an Alexandrov space $\Sigma$ of curvature $\geq 1$ 
has the maximal diameter $\pi$, 
then $\Sigma$ is isometric to the metric suspension $\Sigma (\Lambda)$ 
over some Alexandrov space $\Lambda$ of curvature $\geq 1$.
\end{theorem}

\begin{corollary} \label{maximal radius}
If an $n$-dimensional Alexandrov space $\Sigma$ of curvature $\geq 1$ has the maximal radius $\pi$, then $\Sigma$ is isometric to a unit $n$-sphere of constant curvature $=1$.
\end{corollary}

\begin{remark}[\cite{Mitsuishi}] \upshape
Splitting theorem and Corollary \ref{maximal radius} hold even for infinite dimensional Alexandrov spaces.
\end{remark}

\subsubsection{Convergence and Collapsing theory}

Yamaguchi proved the following two 
Theorems \ref{Lipschitz submersion theorem} and \ref{fibration theorem},
for Alexandrov spaces converging to an almost regular Alexandrov spaces,
which are counterparts of Fibration theorem \cite{Yam collapsing and pinching} in the Riemannian geometry.

\begin{definition}\upshape
A surjective map $f: X \to Y$ between Alexandrov spaces 
is called an {\it $\e$-almost Lipschitz submersion} 
if $f$ is an $\e$-approximation, and 
for any $x, y \in X$ setting $\theta := \angle_x (y_x', \Sigma_x \Pi_x )$, 
then we have 
\begin{equation*}
\biggl| \frac{|f(x) f(y)|}{|x y|} - \sin \theta \biggr| < \e
\end{equation*}
where $\Pi_x := f^{-1} (f(x))$. 

A surjective map $f: X \to Y$ is called an {\it $\e$-almost isometry} if 
for any $x$, $y \in X$ we have 
\[
\biggl| \frac{|f(x)f(y)|}{|xy|} - 1 \biggr| < \e.
\]
\end{definition}

\begin{theorem}[Lipschitz submersion theorem, \cite{Y convergence}]
\label{Lipschitz submersion theorem} 
For $n \in \mathbb{N}$ and $\eta > 0$, 
there exist $\de_n$, $\e_n(\eta) > 0$ satisfying the following.
Let $M^n$, $X^k$ be Alexandrov spaces with curvature $\geq -1$, 
$\dim M^n = n$, and $\dim X^k = k$. 
Suppose that $\de$-strain radius of $X > \eta$. 
Then if the Gromov-Hausdorff distance between $M$ and $X$ 
is less than $\e \leq \e_n(\eta)$, 
then there is a $\theta(\de, \e)$-almost Lipschitz submersion $f: M \to X$. 
Here, $\theta(\de, \e)$ denotes a positive constant depending on $n, \eta$
and $\de, \e$ and satisfying 
$\lim_{\de, \e \to 0} \theta(\de, \e) = 0$.
\end{theorem}

When $M$ is almost regular (and $X$ has nonempty boundary), 
Theorem \ref{Lipschitz submersion theorem} deforms as 
Theorem \ref{fibration theorem} below.
Let $X$ be a $k$-dimensional complete Alexandrov space with curvature $\geq -1$
having nonempty boundary. 
Let $X^{\ast}$ be another copy of $X$.
Take the double $\mathrm{dbl}\,(X) = X \cup X^{\ast}$ of $X$.
The double $\mathrm{dbl}\,(X)$ is also an Alexandrov space of curvature $\leq -1$.
A $(k,\de)$-strainer $\{(a_i,b_i)\}$ of $\mathrm{dbl}(X)$ at $p \in X$ 
is called {\it admissible} if $a_i, b_j \in X$ 
for $1 \leq i \leq k$, $1 \leq j \leq k-1$
($b_k$ may be in $X^{\ast}$ if $p \in \partial X$ for instance).
Put $R_{\de}^D(X)$ the set of all admissible $(k,\de)$-strained points in $X$.

Let $Y$ be a closed domain of $R_{\de}^{D} (X)$.
For a small $\nu > 0$, we put 
\[
Y_{\nu} := \{x \in Y \,|\, d(x, \partial X) \geq \nu \}.
\]
And we put 
\[
\partial_0 Y_{\nu} := Y_{\nu} \cap \{d_{\partial X} = \nu \},\,
\mathrm{int}_0 Y_{\nu} := Y_{\nu} - \partial_0 Y_{\nu}.
\]

The admissible $\de$-strained radius $\de^D\text{-}\mathrm{str. rad}\, x$ at $p \in X$ is the supremum of the length of all admissible $\de$-strainers at $p$.
The admissible $\de$-strained radius $\de^D\text{-}\mathrm{str. rad}\, (Y)$ of a subset $Y \subset X$ is 
\[
\de^D\text{-}\mathrm{str. rad}\, (Y) := \inf_{p \in Y} \de^D\text{-}\mathrm{str. rad}\, p.
\]

\begin{theorem}[Fibration Theorem ({\cite[Theorem 1.2]{Y 4-dim}})]
\label{fibration theorem}
Given $k$ and $\mu > 0$, there exist positive numbers 
$\de = \de_k$, $\e_k(\mu)$ and 
$\nu = \nu_k(\mu)$ satisfying the following:
Let $X^{k}$ be an Alexandrov space with curvature $\geq -1$ of dimension $k$.
Let $Y \subset R_{\de}^{D}(X)$ be closed domain such that 
$\de_{D}$-$\mathrm{str.rad}\,(Y) \geq \mu$.
Let $M^n$ be an Alexandrov spaces with curvature $\geq -1$ of dimension $n$.
Suppose that $R_{\de_n}(M^n) = M^n$ for some small $\de_n > 0$.
If $d_{GH}(M, X) < \e$ for some $\e \leq \e_k(\mu)$, 
then there exists a closed domain $N \subset M$ and a decomposition 
\[
N = N_{\mathrm{int}} \cup N_{\mathrm{cap}} 
\]
of $N$ into two closed domains glued along their boundaries and 
a Lipschitz map $f : N \to Y_{\nu}$ such that 
\begin{itemize}
\item[(1)] $N_{\mathrm{int}}$ is the closure of $f^{-1} (\mathrm{int}_0 \, Y_{\nu})$, 
and $N_{\mathrm{cap}} = f^{-1} (\partial_0 Y_{\nu})$;
\item[(2)] both the restrictions 
$f_{\mathrm{int}} := f|_{N_{\mathrm{int}}} : 
N_{\mathrm{int}} \to Y_{\nu}$ and 
$f_{\mathrm{cap}} := f|_{N_{\mathrm{cap}}} : 
N_{\mathrm{cap}} \to \partial_0 Y_{\nu}$ are 
\begin{itemize}
\item[(a)] locally trivial fiber bundles (see Definition \ref{topological submersion});
\item[(b)] $\theta(\de, \nu, \e/\nu)$-Lipschitz submersions.
\end{itemize}
\end{itemize}
\end{theorem}

\begin{remark}
If $\partial X = \emptyset$ then $N_{\mathrm{cap}} = \emptyset$
in the statement of Theorem \ref{fibration theorem}.
\end{remark}

The following theorem is a fundamental and important tool 
to study a local structure of collapsing Alexandrov spaces.

\begin{theorem}%
[Rescaling Argument \cite{Y essential}, \cite{SY}, \cite{Y 4-dim}]
\label{rescaling argument} 
Let $M_i$, $i = 1, 2, \dots$ be a sequence of Alexandrov spaces 
of dimension $n$ with curvature $\geq -1$
and $X$ be an Alexandrov space 
of dimension $k$ with curvature $\geq -1$ and $k < n$.
Let $p_i \in M_i$ and $p \in X$.
Assume that $(M_i, p_i)$ converges to $(X, p)$.
And $r > 0$ is small number depending on $p$.
Assume that the following:
\begin{assumption} \label{rescaling assumption}
For any $\tilde{p}_i$ with $d(p_i, \tilde{p}_i) \to 0$ and for any sufficiently large $i$,
$B(\tilde{p}_i, r)$ has a critical point for $\dist_{\tilde{p}_i}$
\end{assumption}
\vspace{-1.2em} 
\mbox{}\\ 
Then there exist a sequence $\de_i \to 0$ of positive numbers and 
$\hat{p}_i \in M_i$ such that
\begin{itemize}
\item $d(p_i , \hat{p}_i) \to 0$ as $i \to \infty$;
\item for any limit $Y$ of $(\frac{1}{\de_i}M_i, \hat{p}_i )$, 
we have $\dim Y \geq k + 1$.
\item $\dim S \leq \dim Y - \dim X$, where $S$ is a soul of $Y$.
\end{itemize}
\end{theorem}

\begin{remark} \upshape
If a sequence of $B(p_i, r)$ metric balls 
dose not satisfies Assumption \ref{rescaling assumption}, 
then by Stability Theorem \ref{stability theorem}, 
$B(\tilde{p_i}, r)$ (resp. $U(\tilde{p_i}, r)$) 
is homeomorphic to the closed cone $K_1(\Sigma_{\tilde{p}_i})$
(resp. the open cone $K(\Sigma_{\tilde{p}_i})$) over the space of directions $\Sigma_{\tilde{p_i}}$
for some $\tilde{p}_i \in M_i$ with $d(p_i, \tilde{p}_i)$ tending to zero.
\end{remark}

Fukaya and Yamaguchi proved the following.

\begin{theorem}[\cite{FY}, \cite{Y convergence}] 
\label{almost nilpotent}
For $n \in \mathbb N$, there exists $\e_n > 0$ satisfying the following.
Suppose that an $n$-dimensional Alexandrov space $M^n$
with curvature $\geq -1$ and $\diam M^n < \e_n$.
Then, the fundamental group $\pi_1(M^n)$ is almost nilpotent, i.e.
$\pi_1(M^n)$ has a nilpotent subgroup of finite index. 
\end{theorem}

\begin{remark}[\cite{Y convergence}] \upshape
In Fibration Theorems 
\ref{Lipschitz submersion theorem} and \ref{fibration theorem}, 
the fiber is connected and
has an almost nilpotent fundamental group. 
\end{remark}

\subsubsection{Perelman's Morse theory and Stability theorem}
In this section, we mainly refer to \cite{Per Alex II}.

\begin{definition}[\cite{Per Alex II}] \upshape
Let $f = (f_1, \dots f_m) : U \to \mathbb{R}^m$ be a map on an open subset $U$ of an Alexandrov space $X$ defined by $f_i = d(A_i, \cdot)$ for compact subsets $A_i \subset X$. 
The map $f$ is said to be {\it $(c, \e)$-regular at $p \in U$} if there is a point $w \in X$ such that: 
\begin{itemize}
\item[$(1)$]
$\angle ((A_i)_p', (A_j)_p') > \pi / 2 - \e$. 
\item[$(2)$]
$\angle (w_p', (A_i)_p') > \pi /2 + c$.
\end{itemize}
\end{definition}

\begin{theorem}[\cite{Per Alex II}] \label{Morse theory}
Let $X$ be an finite dimensional Alexandrov space, $U \subset X$ an open subset, and 
let $f$ be $(c, \e)$-regular at each point of $U$. If $\e$ is small compared with $c$, then we have:
\begin{itemize}
\item[$(1)$]
$f$ is a topological submersion (see Definition \ref{topological submersion}).
\item[$(2)$]
If $f$ is proper in addition, then the fibers of $f$ are $MCS$-spaces. 
Hence $f$ is a fiber bundle over its image.
\end{itemize} 
\end{theorem}
Here, a metrizable space $X$ is called an {\it $n$-dimensional MCS-space} if any point $p \in X$ has an open neighborhood $U$ and there exists an $(n-1)$-dimensional compact MCS-space $\Sigma$ such that $(U, p)$ is a pointed homeomorphic to the cone $(K(\Sigma), o)$, where $o$ is the apex of the cone. 
Here, we regarded the $(-1)$-dimensional MCS-space as the empty-set, and its cone as the single-point set.


Perelman proved the stability theorem: 

\begin{theorem}[Stability theorem \cite{Per Alex II} (cf. \cite{Kap stab})] 
\label{stability theorem}
Let $X^n$ be a compact $n$-dimensional Alexandrov space with curvature $\geq \kappa$.
Then there exists $\de > 0$ depending on $X$ such that if $Y^n$ is an $n$-dimensional Alexandrov space with curvature $\geq \kappa$ and $d_{G H}(X, Y) < \de$, then $Y$ is homeomorphic to $X$.

In addition, let $A \subset X$ be a compact subset, and let $A' \subset Y$ be a compact subset.
Then there exists $\de > 0$ depending $(X, A)$ satisfying the following.
Suppose that there is a $\de$-approximation $f : Y \to X$ such that $f(A') \subset A$ and
$f|_{A'}$ is a $\de$-approximation. 
If $t \in (0, \sup d_A)$ is a regular value of $d_A$, then $S(A, t)$ is homeomorphic to $S(A', t)$.
Here, we say that $t$ is a regular value if $d_A$ is regular on $S(A, t)$.
\end{theorem}

In particular, every points in a finite dimensional Alexandrov space have a cone neighborhood over their spaces of directions.

\begin{theorem}[\cite{Per Alex II}] 
\label{diameter suspension theorem}
If an $n$-dimensional Alexandrov space $\Sigma^n$ of curvature $\geq 1$ 
has diameter greater than $\pi / 2$,
then $\Sigma$ is homeomorphic to a suspension 
over an $(n-1)$-dimensional Alexandrov space of curvature $\geq 1$.
\end{theorem}

\begin{theorem}[\cite{Per Alex II}, \cite{Pet Appl}, \cite{GP}] 
\label{radius sphere theorem}
If an $n$-dimensional Alexandrov space $\Sigma^n$ of curvature $\geq 1$
has radius $> \pi / 2$ 
then $\Sigma$ is homeomorphic to an $n$-sphere.
\end{theorem}

\subsubsection{Preliminaries from Siebenmann's theory in \cite{Sie}}

\begin{definition} \label{topological submersion} \upshape
A continuous map $p : E \to X$ between topological spaces is called a {\it topological submersion} (or called a {\it locally trivial fiber bundle}) if for any $y \in E$ there are an open neighborhood $U$ of $y$ in the fiber $p^{-1} (p(y))$, 
an open neighborhood $N$ of $p(y)$ in $X$, 
and an open embedding $f : U \times N \to E$ such that $p \circ f$ is the projection $U \times N \to N$.
We call the embedding $f : U \times N \to E$ a {\it product chart} about $U$ for $p$, 
and the image $f(U \times N)$ a {\it product neighborhood} around $y$.

A surjective continuous map $p: E \to X$ of topological spaces is called a {\it topological fiber bundle} if there exists an open covering $\{ U_{\alpha} \}$ of $X$, and a family $\{F_{\alpha} \}$ of topological spaces, and a family $\{\varphi_{\alpha} : p^{-1} (U_{\alpha}) \to U_{\alpha} \times F_{\alpha}\}$ of homeomorphisms such that $\mathrm{proj}_{U_{\alpha}} \circ \varphi_{\alpha} = p|_{p^{-1}(U_{\alpha})}$ holds for each $\alpha$.
Here, $\mathrm{proj}_{U_{\alpha}}$ is the projection from $U_{\alpha} \times F_{\alpha}$ to $U_{\alpha}$.
\end{definition}

A finite dimensional topological space $Y$ is said to be a {\it WCS-set} \cite[\S 5]{Sie} if it satisfies the following $(1)$ and $(2)$:
\begin{itemize}
\item[(1)] $Y$ is stratified into topological manifolds, i.e., it has a stratification 
\[
Y \supset \cdots Y^{(n)} \supset Y^{(n-1)} \supset \cdots Y^{(-1)} = \emptyset,
\]
such that $Y^{(n)} - Y^{(n-1)}$ is a topological $n$-manifold without boundary.
\item[(2)] For each $x \in Y^{(n)} - Y^{(n-1)}$ there are a cone $C$ with a vertex $v$ and a homeomorphism $\rho : \mathbb R^n \times C \to Y$ onto an open neighborhood of $x$ in $Y$ such that $\rho^{-1} (Y^{(n)}) = \mathbb R^n \times \{v\}$.
\end{itemize}
From the definition, we can see that an MCS-space is a WCS-set.

\begin{theorem}[Union Lemma \cite{Sie}] \label{union lemma}
Let $p : E \to X$ be a topological submersion 
and $F = p^{-1}(x_0)$ the fiber over $x_0 \in X$.
We assume that $F$ is a WCS-space.
Let $A_1$ and $A_2$ be compact sets in $F$.
Let $\varphi_i : U_i \times N_i \to E$ be a product chart about $U_i$ for an open neighborhood $U_i$ of $A_i$ in $F$, and $i = 1, 2$.
Then there exists a product chart $\varphi : U \times N \to E$ about 
$U \supset A_1 \cup A_2$ in $F$ such that 
\[
\varphi = 
\left\{
\begin{aligned}
\varphi_1 \text{ near }& A_1 \times \{x_0\}, \\
\varphi_2 \text{ near }& (A_2 - U_1) \times \{x_0\}.
\end{aligned}
\right.
\]
\end{theorem}

\begin{theorem}[\cite{Sie}]
Let $p: E \to X$ be topological submersion.
We assume that $p$ is proper and all fibers of $p$ are WCS-spaces. 
Then $p$ is a topological fiber bundle over $p(E)$.
\end{theorem}

We provide the following lemma that will be used in Section \ref{proof of 2-dim boundary}.
\begin{lemma} \label{union lemma corollary}
Let $f : E \to [0,1]$ be a fiber bundle, and the fiber $F := f^{-1}(0)$ be a WCS-space.
Let $U \subset F$ be an open subset, and $A \subset U$ be a closed subset.
Suppose that $\varphi : U \times [0,1] \to E$ is a product chart about $U$ for $f$.
Then there exists a product chart $\chi : F \times [0,1] \to E$ such that 
\[
\chi = \varphi \text{ on } A \times [0,1].
\]
In particular, $E - \varphi (A \times [0,1])$ is homeomorphic to 
$(F - A) \times [0,1]$.
\end{lemma}
\begin{proof}
We may assume that $E = F \times [0,1]$ and $f$ is the projection onto $[0,1]$. 
Let $\varphi : U \times [0,1] \to F \times [0,1]$ be a product chart about $U$.
Using Union Lemma \ref{union lemma} and the compactness of $[0,1]$, 
we will construct an extension of $\varphi|_{A \times [0,1]}$ to a product chart defined on $F \times [0,1]$.

By Union Lemma \ref{union lemma}, 
for any $t \in [0,1]$, there exist an open neighborhood $N_t$ of $t$ in $[0,1]$, 
and a product chart 
\[
\psi^{(t)} : F \times N_t \to E
\]
such that 
\[
\psi^{(t)} |_{A \times N_t} = \varphi |_{A \times N_t}.
\]
By the Lebesgue number lemma, there is $n \in \mathbb N$ such that,
setting $I_k := [k /n, (k+1) / n]$, $\{I_k\}_{k = 0, 1, \dots, n-1}$ is a refinement of an open covering $\{N_t\}_{t \in [0,1]}$ of $[0,1]$.
Namely, for $k = 0, 1, \dots, n-1$, there is $t_k \in [0,1]$ such that $I_k \subset N_{t_k}$.
Let us set 
\[
\psi^k := \psi^{(t_k)}|_{F \times I_k}.
\]
For $t \in I_k$, let us define a homeomorphism $\psi^k_t : F \to F$ by the equality:
\[
\psi^k (x, t) = (\psi^k_t (x), t).
\]
Gluing these local product chart $\psi^k$, we construct the required product chart $\chi$ as follows.
We inductively define a homeomorphism $\chi^k_t : F \to F$ by 
\[
\begin{array}{ll}
\chi^0_t = \psi^0_t                             &\text{ for } t \in I_0, \\
\chi^k_t = \psi^k_t \circ (\psi^k_{k/n})^{-1} \circ \chi^{k-1}_{k /n} 
&\text{ for } t \in I_k, k \ge 1.
\end{array}
\]
For $k = 0, 1, \dots, n-1$ and $(x, t) \in F \times I_k$, we define
\[
\chi (x ,t) := (\chi^k_t (x), t)
\]
One can easily check that 
\[
\chi = \varphi \text{ on } A \times [0,1].
\]
Namely, $\chi : F \times [0,1] \to E$ satisfies the conclusion of the lemma.
\end{proof}

\subsection{Differentiable structures of Alexandrov spaces}

Otsu and Shioya \cite{OS} proved that any Alexandrov space has a differential structure and a Riemannian structure in a weak sense.

\begin{definition}[\cite{Per DC}] \upshape \label{DC^1-chart}
Let $U \subset M^n$ be an open subset of an Alexandrov space $M$.
A locally Lipschitz function $f : U \to \mathbb{R}$ is called 
a {\it DC-function} if for any $x \in U$ there exist two (semi-)concave functions 
$g$ and $h$ on some neighborhood $V$ of $x$ in $U$ such that 
$f = g -h$ on $V$. 
A locally Lipschitz map $f = (f_1, \dots, f_m) : U \to \mathbb{R}^m$ is 
called a {\it DC-map} if each $f_i$ is a $DC$-function.
\end{definition}

In \cite[\S 2.6]{KMS}, they formulated a general concept of structure 
on topological spaces.
\begin{definition}[\cite{KMS}] \upshape
For an integer $n \geq 0$, we consider the family 
\[
\mathcal{F} = \{\,
\mathcal{F}(U; A) \,|\,
U \subset \mathbb{R}^n \text{ is an open subset and } A \subset U \text{ a subset}\,
\}
\]
such that 
\begin{itemize}
\item[(i)] each $\mathcal{F}(U; A)$ is a class of maps from $U$ to $\mathbb{R}^n$, 
\item[(ii)] if $A \supset B$, then $\mathcal{F}(U; A) \subset \mathcal{F}(U; B)$,
\item[(iii)] if $f \in \mathcal{F}(U; A)$, $g \in \mathcal{F}(V; B)$, and $f(U) \subset V$, 
then 
\[
g \circ f \in \mathcal{F}(U; A \cap f^{-1}(B)).
\]
\end{itemize}
The following are examples of $\mathcal{F} = \{\mathcal{F}(U; A)\}$.
\vspace{0.5em}\\
(Class $C^1$) 
Let $C^1(U; A)$ be the class of maps from $U$ to $\mathbb{R}^n$ which are $C^1$ on 
$A$, i.e., they are differentiable on $A$ and their derivatives are continuous on $A$.
\vspace{.5em}\mbox{}\\
(Class $DC$)
Let $DC(U; A)$ be the class of maps from $U$ to $\mathbb{R}^n$ which are $DC$ on 
some open subset $O \subset \mathbb{R}^n$ with $A \subset O \subset U$.
\vspace{0.5em}

Let $X$ be a paracompact Hausdorff space, $Y \subset X$ a subset, 
and $\mathcal{F}$ as above.
We call a pair $(U, \varphi)$ a {\it local chart of} $X$ 
if $U$ is an open subset of $X$ and if $\varphi$ is a homeomorphism from $U$ to 
an open subset of $\mathbb{R}^n$.
A family $\mathcal{A} = \{(U,\varphi)\}$ of local charts of $X$ is called an $\mathcal{F}$-{\it atlas on} $Y \subset X$ if the following (i) and (ii) hold:
\begin{itemize}
\item[(i)]
$Y \subset \bigcup_{(U,\varphi) \in \mathcal{A}} U$.
\item[(ii)]
If two local charts $(U,\varphi)$, $(V, \psi) \in \mathcal{A}$ satisfy $U \cap V \neq \emptyset$, then
\[
\psi \circ \varphi^{-1} \in \mathcal{F}(\varphi(U \cap V); \varphi(U \cap V \cap Y)).
\]
\end{itemize}
Two $\mathcal{F}$-atlases $\mathcal{A}$ and $\mathcal{A}'$ on $Y \subset X$ are said to be 
{\it equivalent} if $\mathcal{A} \cup \mathcal{A}'$ is also an $\mathcal{F}$-atlas on $Y \subset X$.
We call each equivalent class of $\mathcal{F}$-atlases on $Y \subset X$ an $\mathcal{F}$-{\it structure} on $Y \subset X$.

Assume that $Y = X$.
Then, an $\mathcal{F}$-structure on $Y \subset X$ is simply called an $\mathcal{F}$-{\it structure on} $X$. 
If there is an $\mathcal{F}$-structure on $X$ then, $X$ is a topological manifold.
We call a space equipped with an $\mathcal{F}$-structure an $\mathcal{F}${\it -manifold}.
Notice that $\mathcal{F}$-manifolds for $\mathcal{F} = C^1$ are 
nothing more than 
$C^1$-differentiable manifolds in the usual sense.
\end{definition}

Let $M^n$ be an $n$-dimensional Alexandrov space. 
Fix a number $\de > 0$ with $\de \ll 1/n$.
By Theorem \ref{around regular point}, for any $x \in M - S_{\de}(M)$, we obtain a local chart $(U, \tilde{\varphi})$, $U = U(x, r)$.
The family $\mathcal{A}_0$ of all the $(U, \tilde{\varphi})$'s on $M$ induces:

\begin{theorem}[\cite{OS}]
There exists a $C^1$-structure on $M - S(M) \subset M$ containing $\mathcal{A}_0$.
\end{theorem}

\begin{theorem}[\cite{Per DC}]
There exists a $DC$-structure on $M - S_{\de}(M) \subset M$ containing $\mathcal{A}_0$.
\end{theorem}

Thus, $M - S_{\de}(M)$ is a $DC^1$-manifold with singular set $S(M)$ in the following sense.

\begin{definition}[{\cite[\S 5]{KMS}}] \upshape
A paracompact topological manifold $V$ with a subset $S \subset V$ is said to be a $DC^1$-{\it manifold with singular set} $S$ if $V$ possesses a $DC$-atlas $\mathcal{A}$ on $V$ which is also a $C^1$-atlas on $V - S \subset V$.
We say that each local chart compatible with the atlas $\mathcal{A}$ is a $DC^1$-{\it local chart}.

Let $V'$ be an another $DC^1$-manifold with singular set $S'$.
A map $f : V \to V'$ is called a $DC^1$-{\it map} if for any $DC^1$-local chart $(U', \varphi')$ of $V'$, $(f^{-1}(U'), \varphi' \circ f)$ is a $DC^1$-local chart of $V$.
A homeomorphism $f : V \to V'$ is called a $DC^1$-{\it homeomorphism} if $f$ and $f^{-1}$ are $DC^1$-maps.
\end{definition}

Using Otsu's method \cite{Otsu}, Kuwae, Machigashira and Shioya \cite{KMS} proved that an almost regular Alexandrov space has a smooth approximation by a Riemannian manifold.

\begin{theorem}[\cite{KMS}, cf. \cite{Otsu}] \label{smooth approximation}
For any $n \in \mathbb{N}$, there exists a positive number $\e_n > 0$ depending only on $n$ satisfying the following:
If $C$ is a compact subset in an $n$-dimensional Alexandrov space $M$ with curvature $\geq -1$, and it is $\e$-strained for $\e \leq \e_n$, then there exist an open neighborhood $U(C)$, a $C^{\infty}$-Riemannian $n$-manifold $N(C)$ with $C^{\infty}$-Riemannian metric $g_{N(C)}$, and a $\theta(\e)$-isometric $DC^{1}$-homeomorphism $f : U(C) \to N(C)$ such that $g_{N(C)} ( df(v), df(w) ) = \langle v, w \rangle + \theta(\e)$ for any $v, w \in \Sigma_x U(C)$ and $x \in U(C)$.
Here, $\langle \cdot, \cdot \rangle$ is the inner product of $T_x M$.
\end{theorem}

\begin{remark} \upshape
Otsu \cite{Otsu} proved this theorem for any Riemannian manifold $M$ with a lower sectional curvature bound and having small excess.
\end{remark}

We will review the proof of Theorem \ref{smooth approximation} in the proof of Theorem \ref{flow theorem} in Section \ref{proof of flow theorem}.
It is important and needed in our proof of Theorem \ref{flow theorem}.

\subsection{Generalized Seifert fiber spaces}
To describe results obtained in the present paper 
we define the notion of a generalized Seifert fiber space. 

\begin{definition}\upshape \label{generalized Seifert fiber space}
Let $M^3$ and $X^2$ be a three-dimensional and two-dimensional topological orbifolds possibly with boundaries, respectively. 
A continuous map $f : M \to X$ is called a {\it generalized Seifert fibration} if 
there exists a family $\{c_x\}_{x \in X}$ of subset of $M$ such that the following properties holds:
\begin{itemize}
\item The index set of $\{c_x\}$ is $X$. 
Each $x \in X$, $f^{-1}(x) = c_x$.
\item Each $c_x$ is homeomorphic to a circle or a bounded closed interval.
$c_x$ are disjoint and 
\[
\bigcup_{x \in X} c_x = M.
\]
\item For each $x \in X$, there exists a closed neighborhood $U_x$ of $x$ such that $U_x$ is homeomorphic to a disk, and putting $V_x := f^{-1} (U_x )$, $V_x$ satisfies the following.
\begin{itemize}
\item[(i)] If $c_x$ is topologically a circle, then $f|_{V_x}: V_x \to U_x$ is a Seifert fibered solid torus in the usual sense.

\item[(ii)] If $c_x$ is topologically a bounded closed interval, then there exist homeomorphisms $\tilde{\phi}_x :  V_x \to B(\pt)$ and $\phi_x:  U_x \to K_1(S_\pi^1)$, which preserve the structure of circle fibration with singular fiber. Namely the following diagram commutes.
\[
\begin{CD}
( V_x, c_x ) @> \tilde{\phi}_x >> ( B(\pt), p^{-1}(o) ) \\
@V f|_{V_x} V V                                 @ V V p V \\
( U_x, x ) @> \phi_x >> (K_1(S_\pi^1), o)
\end{CD}
\]
Here, $B(\pt) = S^1 \times D^2 / \mathbb{Z}_2$ is the topological orbifold defined after Theorem \ref{1-dim circle} and $p$ is a canonical projection. 
\end{itemize} 
\item If $\partial X$ has a compact component $C$, then there is a collar neighborhood $N$ of $C$ in $X$ such that $f|_{f^{-1}(N)}$ is a usual circle fiber bundle over $N$. 
\end{itemize}

We say that a three-dimensional topological orbifold $M$ is a {\it generalized Seifert fiber space over} $X$ if there exists a generalized Seifert fibration $f : M \to X$.
Each fiber $f^{-1}(x) = c_x$ of $f$ is often called an {\it orbit} of $M$. 
An orbit $c_x$ is called {\it singular} if $V_x$ is a usual Seifert solid torus of $(\mu, \nu)$-type with $\mu > 1$ or if $c_x$ is homeomorphic to an interval.

\end{definition}

\subsection{Soul Theorem from \cite{SY} with complete classification} 

In this subsection, we recall the soul theorem for open three-dimensional Alexandrov space of nonnegative curvature, obtained in \cite{SY}. 
And, we classify the geometry and topology of open three-dimensional Alexandrov spaces of nonnegative curvature having two-dimensional soul together with some new precise arguments.
The soul theorem is very important to determine the topology of a neighborhood around a singular point in a collapsing three-dimensional Alexandrov space.

\begin{definition} \upshape
Let $M^n$ be an $n$-dimensional noncompact Alexandrov space with nonnegative curvature.
For a ray $\gamma: [0,\infty) \to M$ in $M$, we define the {\it Busemann function} $b_{\gamma} : M \to \mathbb{R}$ {\it with respect to} $\gamma$ as follows:
\[
b_{\gamma} (x) := \lim_{t \to \infty} d(\gamma(t), x) - t
\]
for $x \in M$. 
Fix a point $p \in M$ and define the {\it Busemann function} $b : M \to \mathbb{R}$ {\it with respect to} $p$ by 
\[
b(x) := \inf_\gamma b_{\gamma} (x) 
\]
for $x \in M$. 
Here, $\gamma$ runs over all the rays emanating from $p$.
The Busemann functions $b_{\gamma}$ and $b$ are concave on $M$.

We denote by $C(0)$ the set of all points attaining the maximum value of $b$: 
\[
C(0) := b^{-1} (\max_{M} b).
\]
Since $b$ is concave, $C(0)$ is an Alexandrov space possibly with boundary of dimension less than $n$.
If $C(0)$ has no boundary, we call it a {\it soul} of $M$.
Inductively, if $C(k)$, $k \geq 0$, has the non-empty boundary, we define $C(k+1)$ the set of all points attaining the maximum value of the distance function $\dist_{\partial C(k)}$ from the boundary $\partial C(k)$: 
\[
C(k+1) := \dist_{\partial C(k)}^{-1} ( \max_{C(k)} \dist_{\partial C(k)} ).
\]
Since $\dist_{\partial C(k)}$ is concave on $C(k)$, $C(k+1)$ is also an Alexandrov space of dimension $< \dim C(k)$. 
Since $M$ has finite dimension, this construction stop, i.e., $\partial C(k) = \emptyset$ for some $k \geq 0$.
Then we call such $C(k)$ a {\it soul} of $M$.
\end{definition}

\begin{proposition}[{\cite{Per Alex II}, cf.\cite[\S 2]{Pet Semi}}] \label{Sharafutdinov retraction}
For any open Alexandrov space $M$ of nonnegative curvature and its soul $S$, there is a Sharafutdinov's retraction from $M$ to $S$.
In particular, $S$ is homotopic to $M$.
\end{proposition}

\subsubsection{Soul Theorem} \label{subsec:soul}

We recall that a noncompact Alexandrov space without boundary is called open.
In this section, we state Soul Theorem for open three-dimensional Alexandrov spaces of nonnegative curvature obtained in \cite{SY}.
And we define examples of open three-dimensional Alexandrov spaces of nonnegative curvature which are not topological manifolds, and study those topologies.

First, we shell prove a rigidity result for the case that a soul has codimension one.
This is a generalization of \cite[Theorem 9.8(2)]{SY}.

\begin{theorem} \label{codim=1} 
Let $M$ be an $n$-dimensional open Alexandrov space and $S$ be a soul of $M$.
Suppose that $\dim S = n -1$ and $S$ has a one-normal point.
Let $B = B(S,t)$ be a metric ball around $S$ of radius $t > 0$.
Then, the metric sphere $\hat S := \partial B$ equipped with the induced intrinsic metric is an Alexandrov space of nonnegative curvature.
And, $\hat S$ has an isometric involution $\sigma$ such that $\hat S / \sigma$ is isometric to $S$ and $M$ is isometric to $\hat S \times \mathbb R / (x, t) \sim (\sigma(x), -t)$.
\end{theorem}

\begin{proof}
Let us denote by
\[
\pi : M \to S
\]
a canonical projection.
Namely, for $x \in M$, we set $\pi(x) \in S$ to be the nearest point from $x$ in $S$.
We use rigidity facts on the $\pi$ referring \cite[\S 9]{SY} and \cite[\S 2]{Y 4-dim}, for proving the theorem.

\begin{assertion} \label{hat S is nonneg}
$\hat S$ satisfies the following convexity property:
For $x, y \in \hat S$ with $|xy| < 2t$, any geodesic $\gamma$ between $x$ and $y$ in $M$ is contained in $\hat S$.
In particular, $\hat S$ with the induced intrinsic metric is an Alexandrov space of nonnegative curvature.
\end{assertion}
\begin{proof}[Proof of Assertion \ref{hat S is nonneg}]
Since $|xy| < 2t$, $\gamma$ does not intersect $S$.
From the totally convexity of $B$, we have $\gamma \subset B$.
Let us consider a curve $\bar \gamma := \pi \circ \gamma$ on $S$.
Let $\sigma_s$ denote a unique ray emanating from $\bar \gamma(s)$ containing $\gamma(s)$.
By \cite[Proposition 2.1]{Y 4-dim}, 
\[
\Pi := \bigcup_{s \in [0, |xy|]} \sigma_s
\] 
is a flatly immersed surface in $M$.
Moving $\gamma$ along with $\Pi$, we obtain a curve $\hat \gamma$ contained in $\hat S$.
This is a lift of $\bar \gamma$ via $\pi : \hat S \to S$.
Therefore, we obtain $L(\hat \gamma)=L(\bar \gamma) \le L(\gamma)=|x y|$. 
Suppose that $\gamma$ is not contained in $\hat S$.
It follows from the construction of $\hat \gamma$ and \cite[Proposition 2.1]{Y 4-dim}, one can show that $L(\hat \gamma) < L(\gamma)$.
This is a contradiction.
Therefore, $\hat \gamma$ must coincide with $\gamma$.
\end{proof}
Now, we denote by $\hat d$ the induced intrinsic metric on $\hat S$.
Assertion \ref{hat S is nonneg} says that $(\hat S, \hat d)$ is an Alexandrov space of nonnegative curvature.
Let us denote by $\hat \pi : \hat S \to S$ the restriction of $\pi$ on $\hat S$.
Let $S_{\mathrm{two}}$ (resp. $S_{\mathrm{one}}$) denote the set of all two-normal (resp. one-normal) points in $S$.
And we set $\hat S_{\mathrm{two}} := \hat \pi^{-1}(S_{\mathrm{two}})$ and $\hat S_{\mathrm{one}} := \hat \pi^{-1} (S_{\mathrm{one}})$.
Then, $\hat \pi : \hat S_{\mathrm{two}} \to S_{\mathrm{two}}$ is a two-to-one map, and $\hat \pi : \hat S_{\mathrm{one}} \to S_{\mathrm{one}}$ is a one-to-one map.

Let us consider $S_\reg := S \cap M_\delta^\reg$ for a small $\delta > 0$, which is open dense in $S$.
Note that since any one-normal point is an essentially singular point \cite{SY}, $S_\reg$ is contained in $S_{\mathrm{two}}$.
By \cite{Pet parallel}, $S_\reg$ is convex, and hence, it is connected. 
We set $\hat S_\reg := \hat \pi^{-1}(S_\reg)$.
The restriction 
\[
\hat \pi : \hat S_\reg \to S_\reg
\]
is a double covering.
We define an involution $\sigma$ on $\hat S_\reg$ as the non-trivial deck transformation of $\hat \pi : \hat S_\reg \to S_\reg$.
By using \cite[Proposition 2.1]{Y 4-dim}, we conclude that $\sigma$ is a local isometry.
Hence, there is a continuous extension of $\sigma$ on the whole $\hat S$.
We denote it by the same notation $\sigma$.
Then, $\sigma$ on $(\hat S, \hat d)$ is also a local isometric involution.
We note that $\sigma$ on $\hat S_{\mathrm{one}}$ is defined as the identity.
From the construction, $\sigma$ is bijective.
Therefore, $\sigma$ is an isometry on $\hat S$ with respect to $\hat d$.
We now fix the metric $\hat d$ on $\hat S$.
By the construction, $\hat S / \sigma$ and $S$ are isometric to each other.

Let us consider the quotient space $N := \hat S \times \mathbb R / (x,s) \sim (\sigma(x), -s)$, which is an open Alexandrov space of nonnegative curvature.
We define $\varphi : N \to M$ as 
sending $[x,t] \in N$ to $x \in \hat S$.
By the construction, $\varphi$ is an isometry.
\end{proof}


\begin{example}[{\cite[p.39]{SY}}] \label{L(S;k)} \upshape
For a nonnegatively curved closed Alexandrov surface $S$ and $p_1, p_2 \dots, p_k \in S$ ($k \in \mathbb Z_{\ge 0}$), we denote by $L(S; k) = L(S; p_1, p_2 \dots, p_k)$ an open three-dimensional Alexandrov space of nonnegative curvature (if it exists) satisfying the following. 
\begin{enumerate}
\item $p_1$, $p_2, \dots, p_k$ are essential singular points in $S$, and $S$ is isometric to a soul of $L(S;k)$. 
Hereafter, $S$ is identified a soul of $L(S;k)$.
\item $\{p_1, \dots, p_k\}$ is the set of all topological singular points in $L(S;k)$.
\item There is a continuous surjection $\pi : L(S; k) \to S$ such that 
for $x \in S - \{p_1, \dots p_k\}$, $\pi^{-1}(x)$ is the union of two rays emanating from $x$ perpendicular to $S$; and 
for $x \in \{p_1, \dots p_k \}$, $\pi^{-1}(x)$ is the unique ray emanating from $x$ perpendicular to $S$.
\item The restriction $\pi : \pi^{-1}(S - \{p_1, \dots p_k\}) \to S - \{p_1, \dots p_k\}$ is a line bundle.
\end{enumerate}
\end{example}

\begin{proposition}[{\cite[Proposition 9.5]{SY}}, cf.{\cite[\S 17]{Y 4-dim}}]
\label{2-dim soul}
If $k \ge 1$, then any space $L(S;k)$ is one of $L(S^2; 2)$, $L(P^2; 2)$ and $L(S^2; 4)$.

\end{proposition}
\begin{remark} \upshape
There is an error in Proposition 9.5 (and Theorem 9.6) in \cite{SY}.
Actually, a space $L(S; 1)$ can not exist, and a space $L(S; 2)$ can have a soul homeomorphic to $P^2$. See \cite[\S 17]{Y 4-dim}.

\end{remark}

\begin{proof}[Proof of Proposition \ref{2-dim soul}]
Since $k \ge 1$, by Theorem \ref{Alex surface}, $S$ is homeomorphic to $S^2$ or $P^2$. 
Moreover, if $S \approx S^2$, then we have $k \le 4$;
and if $S \approx P^2$, then $k \le 2$.

We consider the case that $S \approx P^2$. 
Suppose that $k = 1$. Let $p \in S$ be a unique topological singular point in $L(S;1)$.
Let $\pi : L(S; 1) \to S$ be a surjection obtained in Example \ref{L(S;k)}.
For a neighborhood $B$ of $p$ in $S$ homeomorphic to $D^2$, 
the restriction 
\[
\pi : \pi^{-1}(B) \to B
\]
is fiber-wise isomorphic to $\pi_0 : D^2 \times \mathbb R / \mathbb Z_2 \to D^2 / \mathbb Z_2$ such that $p \in B$ corresponds to the origin of $D^2/ \mathbb Z_2$.
Here, $D^2 \times \mathbb R / \mathbb Z_2$ denotes the quotient space of $D^2 \times \mathbb R$ by an involution $(x, t) \mapsto (-x, -t)$, $D^2 / \mathbb Z_2$ denotes the quotient space of $D^2$ by an involution $x \mapsto -x$ which is homeomorphic to a disk, and $\pi_0$ is a canonical projection $\pi_0 : [x, t] \mapsto [x]$.
In particular, $\partial \pi^{-1}(B)$ is homeomorphic to a Mobius strip $S^1 \tilde \times \mathbb R$. 
On the other hand, $B' := S - \mathrm{int}\, B$ is homeomorphic to $\Mo$.
Then, the restriction $\pi : \pi^{-1} (B') \to B'$ is a line bundle over $\Mo $. 
In particular, it is trivial over $\partial B'$. 
Namely, we have $\partial \pi^{-1} (B') \approx S^1 \times \mathbb R$. 
This contradicts to $\partial \pi^{-1}(B) \approx S^1 \tilde \times \mathbb R$.
Therefore, we obtain that if $S \approx P^2$ then $k =2$.

By a gluing argument as above, if $S \approx S^2$, then $k = 2$ or $4$.
\end{proof}


Explicitly, we determine the topology of $L(S;k)$.
\begin{corollary} \label{good orbifold}
$L(S^2;2)$ is isometric to $\hat S^2 \times \mathbb R / (x,s) \sim (\sigma(x),-s)$, where $\hat S^2$ is a sphere of nonnegative curvature in the sense of Alexandrov with an isometric involution $\sigma$ such that $\hat S^2 / \sigma$ is isometric to the soul $S^2$ of $L(S^2;2)$.

$L(P^2;2)$ is isometric to $K^2 \times \mathbb R / (x,s) \sim (\sigma(x),-s)$, where $K^2$ is a flat Klein bottle with an isometric involution $\sigma$ such that $K^2 / \sigma$ is isometric to the soul $P^2$ of $L(P^2;2)$.

$L(S^2;4)$ is isometric to $T^2 \times \mathbb R / (x,s) \sim (\sigma(x),-s)$, where $T^2$ is a flat torus with an isometric involution $\sigma$ such that $T^2 / \sigma$ is isometric to the soul $S^2$ of $L(S^2;4)$.
\end{corollary}
\begin{proof}
To prove this, it suffices to determine the topology of a metric sphere around the soul of any $L(S;k)$.
For any $L(S;k)$, we denote by $B(S;k)$ a metric ball around $S$.
Let us denote by $\pi$ a canonical projection 
\[
\pi : B(S;k) \to S.
\]
Namely, for $x \in S$, $\pi(x)$ is the nearest point from $x$ in $S$.

We consider the case that $S \approx S^2$ and $k = 2$.
Let $p_1, p_2 \in S$ be the topological singular points of $L(S^2;2)$ in $S$.
We divide $S$ into $D_1$ and $D_2$ such that each $D_i$ is a disk neighborhood of $p_i$ and $D_1 \cap D_2$ is homeomorphic to a circle.
Then, for $i = 1,2$, there is a homeomorphism $\varphi_i : \pi^{-1}(D_i) \to D^2 \times [-1,1]/(x,s) \sim (-x,-s)$. 
The gluing part $\pi^{-1}(D_1 \cap D_2)$ of $\pi^{-1}(D_1)$ and $\pi^{-1}(D_2)$ is homeomorphic to a Mobuis band $\Mo$.
Since the space $D^2 \times [-1,1] / \!\! \sim$ is homeomorphic to $K_1(P^2)$, we obtain that $B(S^2;2) = \pi^{-1}(D_1) \cup \pi^{-1}(D_2)$ is homeomorphic to $K_1(P^2) \cup_{\Mo} K_1(P^2)$ (see Remark \ref{K_1 cup K_1} before).
Then, $\partial B(S^2;2)$ is homeomorphic to a gluing of two copies of $P^2 - \mathrm{int} (\Mo) \approx D^2$.
Therefore, $\partial B(S^2;2) \approx S^2$.

We consider the case that $S \approx P^2$ and $k=2$.
Let $p_1, p_2 \in S$ be the topological singular points of $L(P^2;2)$ in $S$.
We take a disk neighborhood $D$ of $\{ p_1, p_2\}$ in $S$.
Let us divide $D$ into $D_1$ and $D_2$ such that each $D_i$ is a disk neighborhood of $p_i$ and $D_1 \cap D_2$ is homeomorphic to an interval.
Then, $\pi^{-1}(D_1 \cap D_2)$ is homeomorphic to $D^2$.
Hence, $\pi^{-1}(D) = \pi^{-1}(D_1) \cup \pi^{-1}(D_2)$ is homeomorphic to $K_1(P^2) \cup_{D^2} K_1(P^2)$ (see Lemma \ref{K_1 cup_D^2 K_1}).
By Lemma \ref{K_1 cup_D^2 K_1}, $\partial \pi^{-1}(D)$ is homeomorphic to a Klein bottle. 
Since $\pi$ is a non-trivial $I$-bundle over $\partial D_i$ for $i=1,2$, it is a trivial $I$-bundle over $\partial D$.
Then, $\pi^{-1}(\partial D) \approx S^1 \times I$.
Let us set $A := \partial B(P^2;2) \cap \pi^{-1}(D)$.
Since $D$ has singular points $p_1$ and $p_2$ of the projection $\pi$, $A$ is connected, and hence $A$ is homeomorphic to $S^1 \times I$.

Let us set $D' := S - \mathrm{int}\, D$ which is homeomorphic to $\Mo$.
Then, $\pi^{-1}(D')$ is homeomorphic to a total space of an $I$-bundle over $\Mo$, 
which is $\Mo \times I$ or $\Mo \tilde \times I$.
Let us set $A'$ to be $\partial B(P^2;2) \cap \pi^{-1}(D')$.
Therefore, if $\pi^{-1}(D') \approx \Mo \times I$, then $A'$ is a disjoint union of two Mobius bands, and if $\pi^{-1}(D') \approx \Mo \tilde \times I$, then $A'$ is homeomorphic to $S^1 \times I$.
Then, $\partial B(P^2;2) = A \cup A'$ is homeomorphic to a Klein bottle if $\pi^{-1}(D') \approx \Mo \times I$, and is homeomorphic to $S^1 \times I \cup_\partial S^1 \times I$ which is a torus or a Klein bottle if $\pi^{-1}(D') \approx \Mo \tilde \times I$.
Suppose that $\partial B(P^2;2)$ is homeomorphic to $T^2$.
By Theorem \ref{codim=1}, there is an involution on $T^2$ having only two fixed points.
This is a contradiction (see \cite[Lemma 3]{Natsheh}).
Therefore, $\partial B(P^2;2) \approx K^2$.

We consider the case that $S \approx S^2$ and $k=4$.
Let $p_1, p_2, p_3, p_4 \in S$ be all topological singular points of $L(S^2;4)$.
Let $D$ and $D'$ be domains in $S$ homeomorphic to a disk such that $\mathrm{int}\, D$ (resp. $\mathrm{int}\, D'$) contains $p_1$ and $p_2$ (resp. $p_3$ and $p_4$), $D \cap D'$ is homeomorphic to a circle and $S = D \cup D'$.
Let us denote $\partial B(S^2;4) \cap \pi^{-1}(D)$ (resp. $\partial B(S^2;4) \cap \pi^{-1}(D')$) by $A$ (resp. $A'$).
By repeating an argument similar to the case that $L(S;k)=L(P^2;2)$, we obtain that $A$ and $A'$ are homeomorphic to $S^1 \times I$.
Then, $\partial B(S^2;4) = A \cup A'$ is homeomorphic to a torus or a Klein bottle.
Suppose that $\partial B(S^2;4)$ is homeomorphic to $K^2$.
By Theorem \ref{codim=1}, there is an involution on $K^2$ having only four fixed points.
This is a contradiction (see \cite[Lemma 2]{Natsheh}).
Therefore, $\partial B(S^2;4) \approx T^2$.
\end{proof}

\begin{remark} \upshape
Since involutions on closed surfaces are completely classified \cite{Natsheh}, the topology of each $L(S;k)$ is unique.
\end{remark}

For any space $L(S;k)$, we denote a metric ball around $S$ in $L(S;k)$ by $B(S;k)$.
The topology of any $B(S;k)$ as follows. 
\begin{corollary} \label{ball around 2-dim soul}
$B(S^2;2)$ is homeomorphic to $S^2 \times [-1,1] / \mathbb Z_2$; 
$B(P^2;2)$ is homeomorphic to $K^2 \times [-1,1] / \mathbb Z_2$; and 
$B(S^2;4)$ is homeomorphic to $T^2 \times [-1,1] / \mathbb Z_2$.
Here, all $\mathbb Z_2$-actions are corresponding to ones of Corollary \ref{good orbifold}.
\end{corollary}

\begin{theorem}[Soul theorem (Theorem 9.6 in \cite{SY})] 
\label{soul theorem}
Let $Y$ be a three-dimensional open Alexandrov space, and $S$ be an its soul.
Then we have the following. 
\begin{itemize}
\item[(1)] If $\dim S = 0$, then $Y$ is homeomorphic to $\mathbb{R}^3$, or the cone $K(P^2)$ over the projective plane $P^2$, or $M_{\pt}$ which is defined in Example \ref{M_pt}.
\item[(2)] If $\dim S = 1$, then $Y$ is isometric to a quotient $(\mathbb{R} \times N) / \Lambda$, where $N$ is an Alexandrov space with nonnegative curvature homeomorphic to $\mathbb{R}^2$ and $\Lambda$ is an infinite cyclic group. Here, the $\Lambda$-action is diagonal.
\item[(3)] If $\dim S = 2$, then $Y$ is isometric to one of
the normal bundle $N(S) = L(S; 0)$ over $S$, $L(S; 2)$ and $L(S; 4)$.

\end{itemize}
\end{theorem}


We will define examples of $L(S^2; 2)$, $L(P^2; 2)$ and $L(S^2;4)$ in Example \ref{L_i}.


\begin{example}[{\cite[Example 9.3]{SY}}] 
\label{Ex 9.3 in SY} \upshape
Let $\Gamma$ be a group of isometries generated by $\gamma$ and $\sigma$ on $\mathbb{R}^3$.
Here, $\gamma$ and $\sigma$ are defined by $\gamma(x,y,z) = -(x,y,z)$ and $\sigma(x,y,z) = (x+1,y,z)$.
Then we obtain an open nonnegatively curved Alexandrov space $\mathbb{R}^3/\Gamma$.
This space is isometric to $M_{\pt}$ in Example \ref{M_pt}. 
\end{example}

We denote by $B(\pt)$ a metric ball $B(p_0, R)$ around a soul $p_0$ of $M_\pt = \mathbb{R}^3/\Gamma$ for large $R > 0$.
Remark that $B(\mathrm{pt})$ is homeomorphic to $S^1 \times D^2 / (x,v) \sim (\bar x, -v)$.
We can check that $B(\pt)$ is one of $K_1(P^2) \cup_{D^2} K_1(P^2)$.
Here, $K_1(P^2) \cup_{D^2} K_1(P^2)$ denotes the gluing $K_1(P^2) \cup_\varphi K_1(P^2)$ of two copies $K_1(P^2)$ along domains $A_1$ and $A_2$ homeomorphic to $D^2$ contained in $\partial K_1(P^2) \approx P^2$ via a homeomorphism $\varphi : A_1 \to A_2$.
We show that the topology of $K_1(P^2) \cup_{D^2} K_1(P^2)$ does not depend on the choice of the gluing map. 
\begin{lemma} \label{K_1 cup_D^2 K_1}
For any domains $A_1$ and $A_2$ which are homeomorphic to $D^2$ contained in $\partial K_1(P^2)$ and any homeomorphism $\varphi : A_1 \to A_2$, there is a homeomorphism 
\[
\tilde \varphi : K_1(P^2) \cup_\varphi K_1(P^2) \to K_1(P^2) \cup_{id} K_1(P^2).
\]
Here, $id : A_0 \to A_0$ is the identity of a domain $A_0$ which is homemorphic to $D^2$ contained in $\partial K_1(P^2)$.
In particular, any such gluing is homeomorphic to $B(\pt)$.
\end{lemma}
\begin{proof}
Let $X_1$, $X_2$ and $Y_1 = Y_2$ be spaces homeomorphic to $K_1(P^2)$.
Let us take domains $A_1 \subset \partial X_1$, $A_2 \subset \partial X_2$ and $A_0 \subset \partial Y_1 = \partial Y_2$ which are homeomorphic to $D^2$.
Let us take any homeomorphism $\varphi : A_1 \to A_2$.

And let us fix a homeomorphism $\varphi_1 : A_1 \to A_0$.
Then there is a homeomorphism $\hat \varphi_1 : \partial X_1 \to \partial Y_1$ which is an extension of $\varphi_1$.
By using the cone structures of $X_1$ and $Y_1$, we obtain a homeomorphism $\tilde \varphi_1 : X_1 \to Y_1$ which is an extension of $\hat \varphi_1$.
Let us set $\varphi_2 := \varphi_1 \circ \varphi^{-1} : A_2 \to A_0$.
By an argument similar to the above, we obtain a homeomorphism $\tilde \varphi_2 : X_2 \to Y_2$ which is an extension of $\varphi_2$.
And, we define a map $\tilde \varphi : X_1 \cup_\varphi X_2 \to Y_1 \cup_{id_{A_0}} Y_2$ by
\[
\tilde \varphi (x) = 
\left\{
\begin{aligned}
\tilde \varphi_1(x) &\text{ if } x \in X_1 \\
\tilde \varphi_2(x) &\text{ if } x \in X_2
\end{aligned}
\right.
\]
This map is well-defined and a homeomorphism.
\end{proof}

\begin{remark}\label{K_1 cup K_1} \upshape
We define a space $K_1(P^2) \cup_{\Mo} K_1(P^2)$ in a way similar to $K_1(P^2) \cup_{D^2} K_1(P^2)$.
Let us consider domains $A_1$, $A_2 \subset \partial K_1(P^2) \approx P^2$ which are homeomorphic to a Mobius band $\Mo$, and take a homeomorphism $\varphi : A_1 \to A_2$. 
Then, we denote $K_1(P^2) \cup_{\varphi} K_1(P^2)$ by the gluing $K_1(P^2) \cup_{\Mo} K_1(P^2)$ for some gluing map $\varphi$. 
And by an argument similar to the proof of Lemma \ref{K_1 cup_D^2 K_1}, the topology of $K_1(P^2) \cup_{\Mo} K_1(P^2)$ does not depend on the choice of the gluing map.
We can show that any such gluing is homeomorphic to $S^2 \times [-1,1] / (v, t) \sim (\sigma(v), -t)$.
Here, $S^2$ is regarded as $\{v = (x,y,z) \in \mathbb R^3 \mid |v| = 1\}$ and $\sigma$ is an involution defined as $\sigma : (x,y,z) \mapsto (-x,-y,z)$.
Further, it is homeomorphic to $B(P^2;2)$ (see Corollary \ref{ball around 2-dim soul}).


$K_1(P^2) \cup_\partial K_1(P^2)$ denotes the gluing of two copies of $K_1(P^2)$ via a homeomorphism on $\partial K_1(P^2)$. 
This space has the same topology as $K_1(P^2) \cup_{id} K_1(P^2)$, where $id$ is the identity on $\partial K_1(P^2)$, which is homeomorphic to the suspension $\Sigma (P^2)$ over $P^2$.
The proof is done by using the cone structure as in the proof of Lemma \ref{K_1 cup_D^2 K_1}. 
\end{remark}

\begin{example}
\label{L_i} \upshape

We will define open Alexandrov spaces $L_2$ and $L_4$ as follows.
Later, we show that $L_k$ is isometric to an $L(S;k)$ for $k=2,4$.

Recall that $M_\pt$ is defined as
\[
M_\pt := S^1 \times \mathbb R^2 / 
(x, y) \mathop{\sim}\limits^{\alpha} (\bar x, -y)
\]
in Example \ref{M_pt}.
We consider a closed domain $M_\pt'$ of $M_\pt$ as 
\[
M_\pt' := S^1 \times [-\ell , \ell] \times \mathbb R / \alpha 
\]
for some $\ell > 0$. 
Then, $M_\pt'$ is a convex subset of $M_\pt$, and hence it is an Alexandrov space of nonnegative curvature with boundary $\partial M_\pt' \equiv S^1 \times \mathbb R$.

We denote by $L_4$ one of open Alexandrov spaces of nonnegative curvature defined as
\begin{align*}
L_4 (\varphi) = M_\pt' \cup_\varphi M_\pt'.
\end{align*}
for an isometry $\varphi$ on $\partial M_\pt'$.
Here, we use the following notation: 
For Alexandrov spaces $A$ and $A'$ whose boundaries are isometric to each other in the induced inner metric with an isometry $\varphi : \partial A \to \partial A'$, $A \cup_\varphi A'$ denotes the gluing of $A$ and $A'$ via $\varphi$.

We will show that $L_4$ is $L(S^2;4)$ (Lemma \ref{L_4}).

Let $U_{2,1}$ be the Alexandrov space defined by 
\[
U_{2,1} := S^1 \times \mathbb R^2 / (x, y) \mathop{\sim}\limits^\beta (-x, -y).
\]
Let us set 
\[
U_{2,1}' := S^1 \times [-\ell, \ell] \times \mathbb R /\beta \subset U_{2,1}
\] 
which is a convex subset of $U_{2,1}$ and hence it is an Alexandrov space of nonnegative curvature with boundary $\partial U_{2,1}' \equiv S^1 \times \mathbb R$.
Let us set $S(U_{2,1}') := S^1 \times [-\ell, \ell] \times \{0\} / \beta$. 
Note that $S(U_{2,1}')$ is isometric to a Mobius band $\Mo$ 
and $U_{2,1}'$ is isomorphic to an $\mathbb R$-bundle over $S(U_{2,1}')$.

We define open Alexandrov spaces $L_{2,1}$, $L_{2,2}$ and $L_{2,3}$ of nonnegative curvature as
\begin{align*}
L_{2,1} &:= L_{2,1}(\varphi) = M_\pt' \cup_\varphi U_{2,1}' , \\
L_{2,2} &:= L_{2,2}(\varphi) = M_\pt' \cup_\varphi D^2 \times \mathbb R, \text{ and} \\
L_{2,3} &:= L_{2,3}(\varphi) = M_\pt' \cup_\varphi \Mo \times \mathbb R. 
\end{align*}
Here, $\varphi$ denotes a gluing isometry between the corresponding boundaries.
And $D^2$ denotes a two-disk of nonnegative curvature, $\Mo$ is a flat Mobius band.

Let us define an Alexandrov space $A$ of nonnegative curvature 
\[
A := [-a, a] \times [-b ,b] \times \mathbb R / (v,s) \sim (-v,-s). 
\]
Here, $v \in [-a,a] \times [-b,b]$ and $s \in \mathbb R$.
The boundary $\partial A$ is isometric to $S^1 \times \mathbb R$.
We define an open Alexandrov space $L_{2,4}$ of nonnegative curvature as
\[
L_{2,4} = L_{2,4}(\varphi) = A \cup_\varphi A
\]
for some isometry $\varphi$ on $\partial A$.

We will prove that $L_{2,1}$ and $L_{2,3}$ are $L(P^2; 2)$ and $L_{2,2}$ and $L_{2,4}$ are $L(S^2;2)$ (Lemma \ref{L_2}).

From now on throughout this paper, we denote by $L_2$ one of $L_{2,1}$, $L_{2,2}$, $L_{2,3}$ and $L_{2,4}$.
\end{example}

\begin{lemma} \label{L_4}
$L_4$ is $L(S^2;4)$.
\end{lemma}
\begin{proof}
Recall that $L_4 = L_4(\varphi) = M_\pt' \cup_\varphi M_\pt'$.
We identify $\partial M_\pt'$ as $S^1 \times \mathbb R$ via an isometry 
$[\xi, \ell, s] \mapsto [\xi,s]$. 
The isometry $\varphi : \partial M_\pt' \to \partial M_\pt'$ is written as 
\[
\varphi [\xi, \ell, s] = [f(\xi), \ell, g(s)]
\]
for some isometries $f$ on $S^1$ and $g$ on $\mathbb R$.
Then, $g(s) = (\pm 1) \cdot s + g(0)$.

Let us define $E := [-\ell, \ell] \times \mathbb R / (s,t) \sim (-s,-t)$.
Obviously, there is a canonical projection $\pi : M_\pt' \to E$ defined by $[\xi, s, t] \mapsto [s,t]$.
Here, $\xi \in S^1$, $s \in [-\ell, \ell]$ and $t \in \mathbb R$.
The map $\pi$ is a line bundle over $E - \{[0,0]\}$.

For $a \in \mathbb R$, let us define $S_\pt'(a) \subset M_\pt'$ as 
\[
S_\pt' (a) := \left. S^1 \times \left\{\left(t, a t /\ell \right) \,|\, t \in [-\ell, \ell] \right\}  \right/\!\! \alpha.
\]
$S_\pt'(a)$ is homeomorphic to a disk.
Then, by using the fibration $\pi : M_\pt' \to E$, we obtain that $M_\pt'$ is homotopic to $S_\pt'(a)$ for any $a \in \mathbb R$.

By choosing $a$ with respect to $g(0)$, we obtain that $L_4$ is homotopic to the gluing $S_\pt' (a) \cup_\partial S_\pt'(-a)$ which is homeomorphic to $S^2$. 
Thus, a soul of $L_4$ is homeomorphic to a sphere.
Since $M_\pt'$ has only two topological singular points in its interior, $L_4$ has only four topological singular point.
Therefore, $L_4$ is $L(S^2; 4)$.
\end{proof}

\begin{lemma} \label{L_2}
$L_{2,1}$ and $L_{2,3}$ are $L(P^2;2)$, and $L_{2,2}$ and $L_{2,4}$ are $L(S^2;2)$.
\end{lemma}
\begin{proof}
We will use the notation same as in the proof of Lemma \ref{L_4}.

Let us consider $L_{2,1} = L_{2,1}(\varphi) = M_\pt' \cup_\varphi U_{2,1}$.
Since $U_{2,1}'$ is isomorphic to a line bundle over $S(U_{2,1}')$. 
Here, $S(U_{2,1}')$ is a subset of $U_{2,1}'$ homeomorphic to $\Mo$. 
By using the bundle structure of $U_{2,1}'$ and the fibration $\pi$, we obtain that $L_{2,1}$ is homotopic to the gluing $S_\pt' \cup_\partial S(U_{2,1}')$, which is homeomorphic to $P^2$.
It follows from $L_{2,1}$ has only two topological singular points that $L_{2,1}$ is $L(P^2;1)$.

Let us take $L_{2,2} = L_{2,2}(\varphi) = M_\pt' \cup_\varphi D^2 \times \mathbb R$.
By using the fibration $\pi$, we obtain that $L_{2,2}$ is homotopic to the gluing $S_\pt' (a) \cup_\partial D^2$ for some $a$, which is homeomorphic to $S^2$. 
And $L_{2,2}$ has only two topological singular points. 
This implies that $L_{2,2}$ is $L(S^2;2)$.

Let us take $L_{2,3} = L_{2,3}(\varphi) = M_\pt' \cup_\varphi \Mo \times \mathbb R$. 
By using $\pi$, we obtain that $L_{2,3}$ is homotopic to the gluing $S_\pt'(a) \cup_\partial \Mo$ for some $a$, which is homeomorphic to $P^2$. 
It follows from $L_{2,3}$ has only two topological singular points that $L_{2,3}$ is $L(P^2;2)$.

Let us take $L_{2,4} = L_{2,4}(\varphi) = A \cup_\varphi A$.
Recall that $A = [-a,a] \times [-b,b] \times \mathbb R / (x, y, s) \sim (-x, -y, -s)$. 
Let us consider a subset $S' := [-a,a] \times [-b,b] \times \{0\} /\! \sim$ of $A$, which is homeomorphic to a disk.
Let us set $E := [-b,b] \times \mathbb R / (y, s) \sim (-y, -s)$.
There is a canonical projection $\pi' : A \to E$ defined by $\pi' ([x, y, s]) = [y, s]$.
By using it, 
we obtain that $L_{2,4}$ is homotopic to $S' \cup_\partial S'$, which is homeomorphic to $S^2$.
It follows from $L_{2,4}$ has only two topological singular points that $L_{2,4}$ is $L(S^2;2)$.
\end{proof}



\subsection{Classification of Alexandrov surfaces from \cite{SY}} \label{sec:surface}
We recall a result for a classification of Alexandrov surfaces, by quoting \cite{SY}.

\begin{proposition}[The Gauss-Bonnet Theorem, {\cite[Proposition 14.1]{SY}}]
If $X$ is compact Alexandrov surface, then we have
\[
\omega (X) +\kappa (\partial X) = 2\pi \chi(X).
\]
\end{proposition}

\begin{proposition}[The Cohn-Vossen Theorem, {\cite[Proposition 14.2]{SY}}]
If $X$ is noncompact Alexandrov surface, then we have
\[
2 \pi \chi (X) - \pi \chi (\partial X) - \omega (X) - \kappa (\partial X) \geq 0.
\]
\end{proposition}

\begin{theorem}[{\cite[Corollary 14.4]{SY}}] \label{Alex surface}
Let $X$ be a nonnegatively curved Alexandrov surface. 
Then, the following holds.
\begin{itemize}
\item[(1)]
$X$ is homeomorphic to either $\mathbb{R}^2$, $\mathbb{R}_{\geq 0} \times \mathbb{R}$, 
$S^2$, $P^2$, $D^2$ or isometric to $[0,\ell] \times \mathbb{R}$, 
$[0, \ell] \times S^1(r)$, $\mathbb{R}_{\geq 0} \times S^1(r)$, 
$\mathbb{R} \times S^1(r)$, $\mathbb{R} \times S^1(r) / \mathbb{Z}_2$, 
a flat torus, or a flat Klein bottle for some $\ell, r > 0$. 
\item[(2)]
$\mathrm{int}\, X$ contains at most four essential singular points, 
and denoting by $n$ the number 
of essential singular points in $\mathrm{int}\, X$, 
we have the following for some $\ell, r > 0$.
\begin{itemize}
\item[(a)] If $n \geq 1$, $X$ is either homeomorphic to $\mathbb{R}^2$, 
$S^2$, $P^2$, $D^2$ or isometric to 
$\mathrm{dbl}\, ( \mathbb{R}_{\geq 0} \times \mathbb{R}_{\geq 0}) 
\cap \{(x,y) \,|\, y \leq h\}$.
\item[(b)] If $n \geq 2$, $X$ is ether homeomorphic to $S^2$, or isometric to 
$\mathrm{dbl}\, (\mathbb{R}_{\geq 0} \times [0,h])$, 
$\mathrm{dbl}\, (\mathbb{R}_{\geq 0} \times [0,h]) \cap \{(x,y) \,|\, x \leq \ell\}$, or 
$\mathrm{dbl}\, ([0, \ell] \times [0,h]) / \mathbb{Z}_2$.
\item[(c)] If $n \geq 3$, $X$ is homeomorphic to $S^2$.
\item[(d)] If $n=4$, $X$ is isometric to 
$A \cup_{\phi} B$, where $A$ and $B$ are isometric to 
$\mathrm{dbl}\,([0, \ell] \times [0,\infty)) \cap \{(x, y) \,|\, y \leq a\}$ and
$\mathrm{dbl}\,([0, \ell] \times [0,\infty)) \cap \{(x, y) \,|\, y \leq b\}$
for some $a, b > 0$, respectively; and $\phi : \partial A \to \partial B$ is some isometry.
\end{itemize}
\end{itemize}
\end{theorem}

\subsection{A fundamental observation}
In this subsection, we prove fundamental propositions on the sets of topologically singular points of Alexandrov spaces.

First, we remark the following proposition on the number of topologically singular points of a three-dimensional closed Alexandrov space.
Let us consider a $(2n + 1)$-dimensional manifold $X$ such that its boundary $\partial X$ is homeomorphic to the disjoint union $\bigsqcup_{i = 1}^{m} P^{2n}$ of the projective spaces. 
Then we see that $m$ is even.
Indeed, we consider the double $\mathrm{dbl} (X)$ and its Euler number:
\[
0 = \chi (\mathrm{dbl}(X)) = 2 \chi(X) - \chi(\partial X) = 2 \chi(X) - m.
\]

\begin{proposition}
Let $M$ be a three-dimensional closed Alexandrov space.
Then the number of topologically singular points of $M$ is even.
\end{proposition}
\begin{proof}
Since $M$ is compact, $S_{\mathrm{top}}(M)$ is a finite set.
By Theorem \ref{stability theorem}, there exists $r > 0$ such that 
for any $p \in S_{\mathrm{top}}(M)$ 
we have $(B(p, r), p) \approx (K_1(P^2), o)$.
Therefore, 
\[
M_0 := M - \bigcup_{p \in S_{\mathrm{top}}(M)} U(p, r)
\]
is a manifold with boundary 
$\partial M_0 \approx \bigsqcup_{p \in S_{\mathrm{top}}(M)} P^2$.
By the above argument, $\sharp S_{\mathrm{top}}(M)$ is even.
\end{proof}

We also prepare the following proposition. 
\begin{proposition} \label{regular collapsing}
Let $(M_i, p_i)$ be a sequence of $n$-dimensional pointed Alexandrov spaces of curvature $\geq -1$ converging to $(X, p)$.
If $\diam \Sigma_p > \pi / 2$, then $\Sigma_{p_i}$ is homeomorphic to a suspension over an Alexandrov space of curvature $\geq 1$, for large $i$.
\end{proposition}
\begin{proof}
Suppose that the conclusion fails.
Then we have some sequence $\{M_i^n\}$ such that
$(M_i, p_i)$ converges to $(X, p)$ and
each $\Sigma_{p_{i}}$ does not have topological suspension structure over 
any Alexandrov space of curvature $\geq 1$.
It follows from Theorem \ref{diameter suspension theorem}
that $\diam(\Sigma_{p_i}) \leq \pi / 2$.
The convergence of spaces of directions is 
lower semi-continuous:
\[
\liminf_{i \to \infty} \Sigma_{p_i} \geq \Sigma_p .
\]
Then we have $\diam(\Sigma_p) \leq \pi/2$.
This is a contradiction.
\end{proof}

\section{Smooth Approximations and Flow Arguments} \label{proof of flow theorem}
\subsection{Flow Theorem}
\mbox{}

A bijective map $f : X \to Y$ between metric spaces 
is called a bi-Lipschitz if 
both $f$ and $f^{-1}$ are Lipschitz.

\begin{definition} \upshape \label{definition of flow}
Let $M$ be a topological space.
A continuous map $\Phi : M \times \mathbb{R} \to M$ is called 
a {\it flow} if it satisfies 
\begin{align*}
\Phi(x, 0) &= x, \\
\Phi(x, s + t) &= \Phi(\Phi(x, s), t)
\end{align*}
for any $x \in M$ and $s, t \in \mathbb{R}$.
Remark that, for each $t \in \mathbb{R}$, the map 
\[
\Phi_t = \Phi(\cdot, t) : M \to M
\]
has the inverse map $\Phi_{-t}$.

Let $M$ be a metric space. 
If a flow $\Phi$ is a Lipschitz map from 
$M \times \mathbb{R}$ to $M$, then 
we call it a {\it Lipschitz flow}.
Remark that for any Lipschitz flow $\Phi$, 
$\Phi(\cdot, t)$ is bi-Lipschitz for each $t \in \mathbb{R}$.
\end{definition}

By using the proof of Theorem \ref{smooth approximation}, 
we obtain the following theorem.
This is a main tool for the proof of our results 
throughout the present paper.


\begin{theorem}[Flow Theorem] \label{flow theorem} 
For any $n \in \mathbb{N}$, there exists a positive number $\e_n$ depending only on $n$ satisfying the following:
Let $C$ be a compact subset and $S$ be a closed subset in an $n$-dimensional Alexandrov space $M$ with curvature $\geq -1$. 
Suppose that $C \cap S = \emptyset$ and $C$ is $\e$-strained
and $\dist_{S}$ is $(1-\de)$-regular on $C$ for $\de > 0$,
where $0 < \e \leq \e_n$ and $\de$ is smaller than some constant.
Then there exist a neighborhood $U(C)$ of $C$ and a Lipschitz flow $\Phi : M \times \mathbb{R} \to M$ satisfying the following.
\begin{itemize} 
\item[$(i)$]
For any $x \in U(C)$, putting $I_x := \{t \in \mathbb{R} \,|\, \Phi(x,t) \in U(C)\}$, $\Phi(x, t)$ is $5 \sqrt \de +\theta(\e)$-isometric embedding in $t \in I_x$.
\item[$(ii)$]$\Phi$ is leaving from $S$ which speed is almost one. Namely, 
\begin{align} \label{almost gradient}
\left. \frac{d}{d t} \right|_{t = 0+} \!\! \dist_S \circ \Phi(x, t) 
&> 1 - 5 \sqrt \de - \theta(\e) 
\end{align}
at any $x \in U(C)$.
\end{itemize}
\end{theorem}

\begin{proof}[Proof of Theorem \ref{flow theorem}] 
To prove this, we must remember 
the proof of Theorem \ref{smooth approximation}
in reference to \cite{KMS} and \cite{Otsu}.

For a while, $x$ denotes an arbitrary point in $C$.
We set 
\[
v(x) := \frac{\nabla \dist_S}{|\nabla \dist_S|} \in \Sigma_x M.
\]
Since $d_S$ is $(1 -\de)$-regular, we have
\begin{equation}
(\dist_S)'_x (v(x)) = - \cos \angle (S'_x, v(x)) > 1-\de. \label{1-de regular}
\end{equation}
We fix a point $q(x) \in S$ such that
\[
|x q(x)| = |x S|.
\]
Then, by \eqref{1-de regular}, we have
\[
\angle( q(x)' , v(x) ) \ge \angle (S', v(x)) > \pi - \de'.
\]
Here, $\de' := \pi - \cos^{-1} (-1 + \de)$.
Note that $\lim_{\de \to 0} \frac{\de'}{\sqrt \de} = \frac{1}{2 \sqrt 2}$.

We put $\ell := \min \{ \e\strrad (C), d(S, C)\}$.
We fix positive numbers $s$ and $t$ with $s \ll t \ll \ell$.
Take a maximal $0.2s$-net $\{ x_j \, | \, j = 1, \dots, N \}$ of $C$.
Fix any $j \in \{1, \dots, N\}$. 
We take $\e$-strainer $\{q_j^\alpha \,|\, \alpha = \pm 1, \dots, \pm n\}$ at $x_j$ of length $\geq \ell$. 
We may assume that $\{q_j^\alpha\}$ satisfies the following.
\begin{align}
q_j^{+1} &= q (x_j). 
\end{align}
Since $t \ll \ell$, $\{q_j^\alpha\}$ is also $\theta(\e)$-strainer at any $x \in B(x_j,10 t)$.
It follows from $s \ll t$ and \cite[Lemma 1.9]{Y convergence} that 
\begin{equation}
|\wangle q_j^\alpha x y - \angle q_j^\alpha x y| < \theta(\e) 
\end{equation}
for any $x \in B(x_j, s)$ and $y \in B(x, s)$.

We denote by $E_j$ the standard $n$-dimensional Euclidean space.
Define a map 
\[ 
f_j = (f_j^{\alpha})_{\alpha =1}^n : B(x_j, 10t) \to E_j
\] 
by 
\begin{align}
f_j^{\alpha} (y) &= 
\frac{1}{\mathcal{H}^n (B(q_j^{\alpha}, \e'))}
\int_{z \in B(q_j^{\alpha}, \e')}\!\!\!\!\!\!\!
d(y, z) - d(x_j, z)\, d \mathcal{H}^n(z).
\end{align}
where $\e' \ll \e$. 
This map is a 
$\theta(\e)$-almost isometric $DC^1$-homeomorphism,
which is actually a $DC^1$-coordinate system.

\begin{lemma}[{\cite[Lemma 5]{Otsu}}] \label{lemma 5 Otsu}
There is an isometry $F_j^k : E_k \to E_j$ satisfying the following: 
\begin{align}
|F_j^k \circ f_k (y) - f_j (y)| 
&< \theta(\e) s, \label{chain rule 1} \\ 
|d F_j^k \circ d f_k (\xi) - d f_j (\xi) | 
&< \theta(\e) \label{chain rule 2} 
\end{align}
for any $j$ and $k$, 
and $y \in B(x_j, s) \cap B(x_k, s)$ and $\xi \in \Sigma_y$.
\end{lemma}
Remark that each $f_j$ has the directional derivative $d f_j$. 
\begin{proof}[Proof of Lemma \ref{lemma 5 Otsu}]
We first recall how to define $F_j^k$'s.
The property \eqref{chain rule 1} is proved in the same way to the original proof of \cite[Lemma 5]{Otsu} in our situation. 
We only prove \eqref{chain rule 2}.

Fix any $j$ and $k$. 
For $\alpha = 1, \dots, n$, 
take $y_k^{\alpha} \in x_k q_k^{-\alpha}$ and 
$y_k^{-\alpha} \in x_k q_k^{\alpha}$ such that 
\[
|x_k y_k^{\alpha}|= |x_k y_k^{-\alpha}| = s.
\]
Then we have 
\begin{align*}
\langle f_k(y_k^\alpha), f_k(y_k^\beta) \rangle
&= 
s^2 \delta_{\alpha\beta} + \theta(\e, s/\ell).
\end{align*}
for all $\alpha, \beta = 1, \dots, n$.
Here, $\langle \cdot, \cdot \rangle$ is the standard inner product on $E_k$.
Since $s \ll \ell$, $\theta(\e, s/\ell) = \theta(\e)$.
Then, we have
\[
|f_k (y_k^\alpha) - s e_k^\alpha | < \theta(\e). 
\]
Here, $\{e_k^\alpha\}_{\alpha = 1}^n$ is an o.n.b on $E_k$.
In a similar way, we have 
\[
|f_k ( y_k^{-\alpha}) + s e_k^\alpha| < \theta(\e).
\]

We define vectors $v^\alpha$, $w^\alpha \in E_j$ ($\alpha = 1, \dots, n$) by 
\begin{align*}
v^\alpha := \frac{1}{2 s} \{ f_j (y_k^\alpha) - f_j (y_k^{-\alpha}) \}. 
\end{align*}
Then, we have 
\begin{align*}
\langle v^\alpha, v^\beta \rangle 
=& \delta_{\alpha\beta} + \theta(\e).
\end{align*}
Then, $\{v_\alpha\}$ is an almost orthonormal basis. 
By Schmidt's orthogonalization we obtain an orthonormal basis
$\{\tilde{e}_\alpha \}$ of $E_j$ such that 
\[
|\tilde{e}_\alpha - v_\alpha| < \theta(\e).
\]

We now define an isometry $F_j^k : E_k \to E_j$ by changing 
the orthonormal basis and the translation: 
\begin{equation*}
F_j^k(v) = 
f_j(x_k) + \sum_{\alpha = 1}^n 
\langle v, e_k^\alpha \rangle \tilde{e}_\alpha.
\end{equation*}
Then, we have 
\begin{align*}
F_j^k(f_k(x)) 
&= f_j(x) + s \vec{v}({\theta}(\e) )
\end{align*}
for all $x \in B(x_j, s)$.
Here, $\vec{v}(c)$ is a vector whose norm less than or equal to $|c|$. 


We prove (\ref{chain rule 2}). 
For any $y \in B(x_j, s) \cap B(x_k, s)$ and $\xi \in \Sigma_y$, by Lemma \ref{lemma 1.8}, 
there exists $z \in M$ such that 
\begin{equation}
|y z| = t \text{ and } \angle(\xi, \uparrow_y^z) = \theta(\e).
\end{equation}
Then, we have
\[
\wangle q_j^\alpha y z = \angle ((q_j^\alpha)', z') + \theta(s/t).
\]
Since $s \ll t$, we have $\theta(s/t) = \theta(\e)$.
Therefore, 
\begin{align}
d_y f_j^{\alpha} (\xi) 
&= 
\frac{1}{\mathcal{H}^n(B(q_j^{\alpha}, \e'))} 
\int_{w \in B(q_j^{\alpha}, \e')}\!\!\!\!\!\!\!\!\!
- \cos \angle (w'_y , \xi) \,
d \mathcal{H}^n (w) \\
&= 
- \cos \wangle q_j^{\alpha} y z + \theta(\e). \label{derivation}
\end{align}
On the other hands, 
\begin{align}
d F_j^k \circ d f_k (\xi) &= d F_j^k \left( \big( - \cos \wangle q_j^{\alpha} y z \big)_{\alpha=1, \dots, n} \right) + \vec{v}({\theta}(\e)) \\
&= \sum_\alpha - \cos \wangle q_j^{\alpha} y z \cdot \tilde{e}_\alpha + \vec{v}({\theta}(\e)).
\end{align}
Therefore, we have \eqref{chain rule 2}.
\end{proof}

Set $V_j := B(0, 0.4 s) \subset E_j$ for all $j$.

Next, we perturb $\{F_j^k\}$ to a family $\{\tilde{F}_j^k \}$ satisfying the following.
\begin{lemma}[{\cite[Lemma 6]{Otsu}}] \label{lemma 6 Otsu}
For any $j$ and $k$ with $d(x_j, x_k) < 0.9s$, 
there exists a $\theta(\e)$-almost isometric $C^\infty$ map $\tilde{F}_j^k: V_k \to E_j$ satisfying the following: 
\begin{align}\label{chain rule}
\tilde{F}_j^j &= \mathrm{id} \text{ on $E_j$ and } \\
\tilde{F}_j ^l (v) &= \tilde{F}_j^k \circ \tilde{F}_k^l (v) 
\end{align}
for any $j$ and $k$ with $d(x_j, x_k) < 0.9s$ and $v \in V_l \cap \tilde{F}_l^k(V_k) \cap \tilde{F}_l^j(V_j)$.

Moreover, we can obtain this perturbed 
$\{ \tilde{F}_j^k \}$ also satisfying $( \ref{chain rule 1} )$ and $(\ref{chain rule 2})$. That is, we have
\begin{align}
|\tilde{F}_j^k \circ f_k (y) - f_j (y)| 
&< \theta(\e) s, \label{chain rule 3} \\ 
|d \tilde{F}_j^k \circ d f_k (\xi) - d f_j (\xi) | 
&< \theta(\e) \label{chain rule 4} 
\end{align}
for any $j$ and $k$, 
and $y \in B(x_j, s) \cap B(x_k, s)$ and $\xi \in \Sigma_y$.
\end{lemma}
\begin{proof}
We only review the first step of construction of $\tilde{F}_j^k$'s by induction referring to the proof of \cite{Otsu}.

Let us first review how to construct $\tilde{F}_j^k$'s.
Let $\phi : [0, \infty) \to [0, \infty)$ be a $C^{\infty}$-function
such that 
\begin{align*}
\phi &= 1 \text{ on } [0, 1/2], \\
\phi &= 0 \text{ on } [1, \infty), \text{ and }\\ 
-4 &\leq \phi' \leq 0.
\end{align*}
Set \[
\psi_j (v) : = \phi (|v|/0.8s) 
\] 
for $v \in V_j$.

We set $\tilde{F}_j^1 = F_j^1$ and $\tilde{F}_1^j = (\tilde{F}_j^1)^{-1}$, and 
define $\tilde{F}_j^2 : U_2 \to \mathbb{R}^n$ for $j \geq 2$ by 
\begin{equation*}
\tilde{F}_j^2 (v) := 
\psi_1 \circ \tilde{F}_1^2 (v) \cdot
 \tilde{F}_j^1 \circ \tilde{F}_1^2 (v) +
 (1 -\psi_1 \circ \tilde{F}_1^2 (v) ) \cdot
{F}_j^2 (v)
\end{equation*}
for $v \in V_2$. 
By construction, $\tilde{F}_j^2$ is smooth and satisfies \eqref{chain rule 3} and \eqref{chain rule 4}.

For $v \in V_2$, we have
\begin{align*}
\phi_1 \circ \tilde{F}_1^2 (v) &= 1 + \theta(\e)|v|, \\
\|d (\phi_1 \circ \tilde{F}_1^2) \|_{C^1} &= \theta(\e).
\end{align*}
Therefore, we have, for any $v \in V_2$ and $w \in T_v E_2$ with $|w| = 1$, 
\begin{align*}
d \tilde{F}_j^2 (w) &= d (\tilde{F}_j^1 \circ \tilde{F}_1^2) (w) + \theta(\e) \\
&= d( F_j^1 \circ F_1^2) (w) + \theta(\e) \\
&= d F_j^2 (w) + \theta(\e).
\end{align*}
Thus, we have
\begin{equation*}
\| d \tilde{F}_j^2 - d F_j^2 \| < \theta(\e)
\end{equation*}
at any $v \in V_2$. 

Therefore, for a segment $c : [0, t_0] \to V_2$ between $v$ and $y$, we have 
\[
|\tilde{F}_j^2 (v) - \tilde{F}_j^2 (w)| = \left| \int_0^{t_0} d \tilde{F}_j^2 (c'(t) ) dt \right| \geq 0.9 |v -w|.
\]
Thus, $\tilde{F}_j^2$ is injective.
\end{proof}

By the chain rule (\ref{chain rule}), 
an equivalence relation $\sim$
on the disjoint union $\bigsqcup_{j} V_j$
is defined in the following natural way:
$V_j \ni y \sim y' \in V_k \iff 
\tilde{F}_j^k (y') = y$. 
Set $N := \bigsqcup V_j /\!\! \sim$.
We denote by $\pi$ the projection 
\[
\pi: \bigsqcup_{j =1}^N V_j \to N.
\]
We denote by $\tilde{V}_j := \pi(V_j)$ the subset of $N$ corresponding to $V_j$, 
and by $\pi_j$ the restriction of $\pi$
\[
\pi_j : V_j \to \tilde V_j.
\]
We define $\tilde{f}_j := \pi_j^{-1}$. 
Then $N$ is a $C^{\infty}$-manifold with atlas $\{(\tilde{V}_j, \tilde{f}_j) \}_j$, and $\tilde{F}_j^k : \tilde{f}_k(\tilde{V}_{k} \cap \tilde{V}_{j}) \to \tilde{f}_j(\tilde{V}_{k} \cap \tilde{V}_{j})$ is the associate transformation.


Define maps $f^{(j)} : B(x_j, s) \to E_j$ ($j = 1, \dots, N$) by 
\begin{align*}
f^{(1)} (x) &:= f_1 (x) , \\
f^{(2)} (x) &:= \psi_1 \circ f^{(1)} (x) \cdot 
\tilde{F}_2^1 \circ f^{(1)} (x) +
(1 - \psi_1 \circ f^{(1)} (x) ) f_2 (x), \\
&\cdots 
\end{align*}
Set $\hat{V}_j := f^{(j)-1}(V_j)$.
Then we have 
\[
f^{(j)} = \tilde{F}_j^k \circ f^{(k)}
\]
on $\hat{V}_j \cap \hat{V}_k$.
Indeed, for instance, $f^{(2)} = \tilde{F}_2^1 \circ f^{(1)}$ on $B_1 \cap B_2$.
For general case, we refer to \cite[pp 1272-1273]{Otsu}.
Set $U := \bigcup_j \hat{V}_j$. 
A homeomorphism $f : U \to N$ is defined 
to be the inductive limit of $\pi \circ f^{(j)}$.

By \cite[Lemma 8]{Otsu}, we obtain the following properties of $f^{(j)}$.
\begin{align}
|f_j(x) - f^{(j)}(x)| &< \theta(\e) s, \label{error 1} \\
|d f_j(\xi) - d f^{(j)}(\xi)| &< \theta(\e) \label{error 2}
\end{align}
for all $x \in B(x_j, 0.4s)$ and $\xi \in \Sigma_x$.

Let $\{ \chi_j \}_j$ be a smooth partition of unity such that
$\mathrm{supp}\, (\chi_j) \subset \tilde{V}_j$. 
The desired Riemannian metric $g_N$ on $N$ is defined by 
\begin{equation}\label{Riemannian metric}
(g_N)_{x} (v, w) := 
\sum_{j} \chi_j (x) 
\langle d \tilde{f}_j (v), d \tilde{f}_j (w) \rangle
\end{equation}
for any $x \in N$ and $v ,w \in T_x N$.

\vspace{1em} 
Up to here, we reviewed the construction 
of a smooth approximation $f : U \to N$ by \cite{KMS} (and \cite{Otsu}).
Next, we construct the desired flow.

We first remark that 
\begin{lemma} \label{flow theorem lemma 0}
For each $j$, $f^{(j) -1} : V_j \to \hat{V}_j$ is differentiable.
And hence, $f$ and $f^{-1}$ are also differentiable.
\end{lemma}
\begin{proof}
Since $f^{(j)}$ is differentiable, for any scale $(o)$, the following diagram commutes.
\[
\begin{CD}
\left(\hat{V}_j\right)_x^{(o)} @> \left(f^{(j)}\right)_x^{(o)} >> \left( V_j \right)_{y}^{(o)} \\
@V \hat{\rho}^{(o)} V V                                                                                                 @V V \rho^{(o)} V \\
T_x M                                                           @>> d_x f^{(j)} >                                       T_y E_j
\end{CD}
\]
where $y := f^{(j)}(x)$ and $\hat{\rho}^{(o)}$ and $\rho^{(o)}$ are canonical isometries. We will omit the symbol $(o)$ to write 
$\hat{\rho} := \hat{\rho}^{(o)}$ and $\rho := \rho^{(o)}$

Since $f^{(j)}$ is $\theta(\e)$-isometric, $\left(f^{(j)}\right)_x^{(o)}$and $\left(f^{(j) -1}\right)_y^{(o)}$ are so. We define a map $A : T_y E_j \to T_x M$ by
\[
A := \hat{\rho} \circ \left(f^{(j) -1}\right)_y^{(o)} \circ \rho^{-1}.
\]
Then we have 
\begin{align*}
A \circ d_x f^{(j)} &= \mathrm{id}_{T_x M}, \\
d_x f^{(j)} \circ A &= \mathrm{id}_{T_y E_j}.
\end{align*}
Namely, $A = \left( d_x f^{(j)} \right)^{-1}$ is determined independently choice of $(o)$.
By its construction, $A = d_y (f^{(j)-1})$ is well-defined.
Thus $f^{(j)-1}$ is differentiable.

$f$ is the composition of differentiable map $f^{(j)}$ and smooth map $\pi_j$, and hence $f$ and $f^{-1}$ are also differentiable.
\end{proof}

Set $y_j := y_j^{+1}$. 
Remark that $y_j$ can taken satisfying the following.
\begin{equation}
\left\{
\begin{aligned}
&|x_j y_j| = t, \\
&\angle S x y_j \ge \wangle S x y_j > \pi - \de' - \theta(\e, s /t ) 
\end{aligned}
\right.
\end{equation}
for all $x \in B(x_j, s)$.


Now, let us forget the construction of $f_j$ above, 
we will use the following notation: 
\begin{equation}
f_j := f^{(j)}. \label{reconsideration}
\end{equation}

We set 
\begin{align*}
Y_j(x) &:= \big\uparrow_x^{y_j} \in \Sigma_x M, \\
Z_j(x) &:= \frac{f_j (y_j^{+1}) - f_j (x)}{|f_j (y_j^{+1}) - f_j (x)|} \in E_j, 
\end{align*}
for all $x \in B(x_j, s)$.

We recall that $f_j$ is $\theta(\e)$-isometry on $B(x_j, t)$.
It follows from \eqref{derivation} that we have 
\begin{equation}\label{error}
d f_j ( Y_j(x) ) = Z_j(x) + \vec{v}({\theta}(\e))
\end{equation}
for any $x \in B(x_j, 0.4 s)$.


Since
\[
\angle_x (S', Y_j) + \angle_x (S', Y_k) + \angle_x (Y_j, Y_k) \le 2 \pi, 
\]
we have 
\[
\angle (Y_j(x), Y_k(x)) < 2 \de' + \theta(\e)
\]
for all $x \in B(x_j, s) \cap B(x_k, s)$.
Then, we have 
\begin{align*}
d(Y_j, Y_k)^2 &< 2 (1 - \cos (2 \de' + \theta(\e))) \\
&\le 2 \de'^2 + \theta(\e).
\end{align*}
Therefore, we obtain 
\begin{align}
d f_j (Y_k(x)) &= d f_j (Y_j(x)) + \vec{v}(\sqrt{2} \de' + \theta(\e)) \\
&= Z_j(x) + \vec{v}(\sqrt{2} \de' + \theta(\e)). \label{error 5}
\end{align}

Note that $Z_j$ is smooth on $V_j \subset E_j$.
We define a smooth vector field $\tilde{W}_j$ on $\tilde{V}_j \subset N$ by 
\[
\tilde{W}_j (x) := d \tilde{f}_j^{-1} ( Z_j(x) ).
\]

We next prove the following.
\begin{lemma} \label{error lemma}
For any $x \in \tilde{V}_j \cap \tilde{V}_k$, we have
\[
|\tilde{W}_j(x) - \tilde{W}_k(x)|_N < 4 \sqrt{2} \de' + \theta(\e). 
\]
\end{lemma}
\begin{proof}
At first, we see
\begin{align*}
|\tilde{W}_j - \tilde{W}_k|_N^2 
&= \sum_{\ell = 1}^N \chi_\ell \left| d \tilde{f}_\ell \left( \tilde{W}_j - \tilde{W}_k \right) \right|_{E_\ell}^2 \\
&= \sum_{\ell = 1}^N \chi_\ell \left| d \tilde{F}_\ell^j (Z_j) - d \tilde{F}_\ell^k (Z_k)\right|_{E_\ell}^2 .
\end{align*}
By \eqref{error 5}, we have
\begin{align*}
d \tilde{F}_\ell^k (Z_k) &= d \tilde{F}_\ell^k (d f_k (Y_k)) + \vec{v}(\theta(\e)) \\
&= d f_\ell (Y_k) + \vec{v}(\theta(\e)) \\ 
&= Z_\ell + \vec{v}(2 \sqrt{2} \de' + \theta(\e)).
\end{align*}
Therefore, Lemma \ref{error lemma} is proved.
\end{proof}

We next define a smooth vector field $\tilde{W}$ on $N$ by 
\begin{align}
\tilde{W} (x) := \sum_{j=1}^N \chi_j (x) \, \tilde{W}_j (x).
\end{align}
By Lemma \ref{error lemma}, we have 
\begin{align*}
|\tilde{W} - \tilde{W}_j| 
&= \left| \sum_\ell \chi_\ell \cdot \tilde W_\ell - \tilde W_j \right| \\
&\le \sum_\ell \chi_\ell | \tilde W_\ell - \tilde W_j | < 4\sqrt{2}\de' + \theta(\e)
\end{align*}
on $\tilde{V}_j \subset N$.

We consider an integral flow $\tilde{\Phi}$ of $\tilde{W}$. Namely, 
\[
\frac{d \tilde{\Phi} }{dt} (x, t) = \tilde{W} (\tilde{\Phi} (x,t) ).
\]
We now define a flow $\Phi$ on $U$ by 
\[
\Phi(x, t) := f^{-1} \big( \tilde{\Phi}(f(x),t) \big).
\]

\begin{lemma} \label{flow theorem lemma 2}
The conclusion $(ii)$ of Theorem \ref{flow theorem} holds:
\begin{align*}
\left. \frac{d}{d t} \right|_{t = 0+} \!\! \dist_S \circ \Phi(x, t) 
&> 1 - 5 \sqrt \de - \theta(\e), \\
(\tilde d_S)' (\tilde W(\tilde x)) 
&> 1 - 5 \sqrt \de - \theta (\e)
\end{align*}
for all $x \in U$.
\end{lemma}
\begin{proof}
By Lemma \ref{flow theorem lemma 0}, $f$ is differentiable.
Therefore, the flow curve 
\[
\Phi(x, \cdot) : I_x \to U
\]
is differentiable for any $x \in U$.
Then $x \in V_j$ for some $j$.
Then, we have
\begin{align*}
\left. \frac{d}{d t} \right|_{t = 0} \!\!\!\!
d_S \circ \Phi(x ,t) &= (d_S \circ f^{-1})' \left( \frac{d}{d t} f \circ \Phi(x, t) \right) \\
& = (d_S \circ f^{-1})' \left( \tilde W(f(x)) \right) 
= (d_S \circ f_j^{-1})' \circ d \tilde f_j ( \tilde W ) \\
& = (d_S \circ f_j^{-1})' \circ d \tilde f_j ( \tilde W_j + 4 \sqrt{2} \vec{v}(\de' + \theta(\e)) ) \\
& = (d_S \circ f_j^{-1})' (Z_j + 4 \sqrt 2 \vec{v}(\de' + \theta(\e))) \\
& = (d_S \circ f_j^{-1})' (d f_j (Y_j) + 4 \sqrt 2 \vec{v}(\de' + \theta(\e))) \\
& > d_S' (Y_j) - 4 \sqrt 2 \de' - \theta(\e) \\
& > 1 - \de - 4 \sqrt 2 \de' - \theta(\e) \\
& > 1 - 5 \sqrt \de - \theta(\e). 
\end{align*}
\end{proof} 

\begin{lemma} \label{flow theorem lemma 1}
The conclusion $(i)$ of Theorem \ref{flow theorem} holds:
For any $x \in U$, $\Phi(x, t)$ is $5 \sqrt \de + \theta(\e)$-isometric embedding in $t \in I_x$.
Here, $I_x := \{t \in \mathbb{R} \,|\, \Phi(x,t) \in U\}$.
\end{lemma}
\begin{proof}
By the construction of $\tilde W$, we have $|\tilde W| \le 1 + \theta(\e)$.
Indeed, for any $t, t' \in I_x$ with $t < t'$, we obtain
\begin{align*}
d \left( \tilde{\Phi}(f(x),t'), \tilde{\Phi}(f(x),t) \right) 
&\leq \int_t^{t'} \left| \tilde{W} \left( \tilde{\Phi}(f(x), s) \right) \right| d s \\
&\leq |1 + \theta(\e)| (t' -t).
\end{align*}
Then we have
\begin{align*}
\frac{d\left(\Phi(x,t'), \Phi(x,t) \right)}{t' -t} 
& = 
\frac{ \left| \Phi(x,t'), \Phi(x,t) \right| }
{ \left| \tilde{\Phi}(f(x),t'), \tilde{\Phi}(f(x),t) \right| } 
\cdot
\frac{ \left| \tilde{\Phi}(f(x),t'), \tilde{\Phi}(f(x),t) \right| }
{t'-t} \\
& \leq 1 + \theta(\e)
\end{align*}

By Lemma \ref{flow theorem lemma 2}, for $t < t'$ in $I_x$, we obtain
\begin{align*}
d \left( \Phi(x,t'), \Phi(x,t) \right) &\ge d(S, \Phi(x, t')) - d(S, \Phi(x, t)) \\
&= \int_t^{t'} (d_S)' (W) d s \ge (1 - 5 \sqrt \de - \theta(\e)) (t' - t).
\end{align*}
This completes the proof of Lemma \ref{flow theorem lemma 1}.
\end{proof} 

Combining Lemmas \ref{flow theorem lemma 1} and \ref{flow theorem lemma 2}, we obtain the conclusions of Theorem \ref{flow theorem}.
\end{proof}

\begin{definition} \upshape \label{gradient-like flow}
Let $M$ be an Alexandrov space, $f : M \to \mathbb{R}$ be a Lipschitz function and $\Phi : M \times \mathbb R \to M$ be a Lipschitz flow. Let $M'$ be a subset of $M$.
We say that $\Phi$ is {\it gradient-like for $f$ on $M'$} if there exists a constant $c > 0$ such that 
for any $x \in M'$, we have
\[
\liminf_{t \to 0} \frac{f(\Phi(x, t)) - f(x)}{t} > c.
\]
We denote by 
\[
\Phi \pitchfork f \text{ on } M'
\] 
this situation.
\end{definition}

In this notation, we obtained in Theorem \ref{flow theorem}, a gradient-like flow $\Phi$ for $\dist_S$ on $U(C)$ with a constant $c = 1 - 5 \sqrt \de - \theta(\e)$.

\subsection{Flow and Fibration}
We will find out a nice relation between Fibration Theorems \ref{Lipschitz submersion theorem} and \ref{fibration theorem} and Flow Theorem \ref{flow theorem}.
We first recall an important property of Yamaguchi's fibration.

\begin{proposition}[cf. Lemma 4.6 in \cite{Y convergence} ] 
\label{almost linear}
Let $M$ and $X$ be Alexandrov spaces and 
$\pi : M \to X$ be a $\theta(\de, \e)$-Lipschitz submersion 
as in Theorem \ref{Lipschitz submersion theorem}.
Let $(o) = (\e_i)$ be an arbitrary scale. 
We denote by $H_x$ a set of horizontal directions to the fiber $\pi^{-1}(\pi(x))$ at $x$. 
Then for any $x \in M$, the restriction of the blow-up 
\[
\pi_x^{(o)} \circ \exp_x^{(o)}: H_x \to X_{\pi(x)}^{(o)}
\]
satisfies the following:
For any $Y$, $Z \in H_x$, we have 
\[
\left| \left| \pi_x^{(o)} \circ \exp_x^{(o)}(Y), \pi_x^{(o)} \circ \exp_x^{(o)}(Z) \right| - |Y, Z| \right| < \theta(\de, \e).
\]
\end{proposition}
Here, the set of horizontal directions is defined in \cite[\S 4]{Y convergence} as 
\[
H_x := \{ \xi \in y_x' \,|\, |x y| \geq \sigma \}
\]
for some small number $\sigma > 0$ with $\e \ll \sigma$.
\begin{proof}[Proof of Proposition \ref{almost linear}]
We will use the following notation: 
$\theta$ denotes a variable constant $\theta(\de, \e)$.
We set $\bar x = \pi(x)$ for any $x \in M$.

Let us take $Y \in H_x$.
By the definition of $H_x$, there is a point $y \in M$ such that 
\[
|x y| \geq \sigma, \,\, \angle (y', Y) < \theta.
\]
Then, by Lemma 4.6 in \cite{Y convergence}, 
for any $\bar Y \in \bar y' \subset \Sigma_{\bar x} X$, we have 
\begin{equation} \label{almost tangent 1}
\frac{|\pi(\gamma_Y (t)), \gamma_{\bar Y}(t)|}{t} < \theta
\end{equation}
for any small $t > 0$. 
Here, $\gamma_\xi$ denotes the geodesic from $\gamma_\xi(0)$ tangent to $\xi \in \Sigma_{\gamma(0)}$.
Let $(o) = (\e_i)$ be an arbitrary scale.
From \eqref{almost tangent 1}, we have
\begin{equation} \label{almost tangent 2}
|\pi_x^{(o)} \circ \exp_x^{(o)} (Y), \exp_{\bar x}^{(o)} (\bar Y)| 
= \lim_\omega \frac{|\pi(\gamma_Y (\e_i)), \gamma_{\bar Y}(\e_i)|}{\e_i} 
< \theta.
\end{equation}

We next take any $Z \in H_x$. 
Then there exists $z \in M$ such that 
\[
|x z| \geq \sigma, \,\, \angle (z', Z) < \theta.
\]
Then, for any $\bar Z \in \bar z' \subset \Sigma_{\bar x} X$, we have
\begin{equation} \label{almost tangent 3}
|\pi_x^{(o)} \circ \exp_x^{(o)} (Z), 
\exp_{\bar x}^{(o)} (\bar Z)| < \theta.
\end{equation}

On the other hand, by Lemma 4.7 in \cite{Y convergence}, we have 
\[
|\angle (Y, Z) - \angle (\bar Y, \bar Z) | < \theta.
\]
It follows together \eqref{almost tangent 1}, \eqref{almost tangent 2} and \eqref{almost tangent 3} that we obtain 
\[
||\pi_x^{(o)} \circ \exp_x^{(o)} (Y), \pi_x^{(o)} \circ \exp_x^{(o)} (Z)| 
- |Y, Z|| < \theta.
\] 
This completes the proof.
\end{proof}

\begin{theorem} \label{flow application}
For any $n \in \mathbb{N}$, there is a positive number $\epsilon = \epsilon_n$ satisfying the following:
Let $M^n$ be an $n$-dimensional Alexandrov space without boundary with curvature $\geq -1$ and $p$ be a point of $M^n$. 
Let $X^{n-1}$ be an $(n-1)$-dimensional nonnegatively curved Alexandrov space. 
Assume that $X$ is given by the Euclidean cone $K(\Sigma)$ over a closed Riemannian manifold $\Sigma$ of curvature $\geq 1$.
If $d_{G H}((M,p), (X, p_0)) < \epsilon$, where $p_0$ is the origin of the cone $X$, then there exists a small $r = r_p >0$ such that a metric sphere $\partial B(p, r)$ is homeomorphic to an $S^1$-fiber bundle over $\Sigma$.
\end{theorem}
\begin{proof}
$d_{GH}((M, p), (X, p_0)) < \e$ implies $d_{GH}(B_M(p, 1/\e), B_X(p_0, 1/\e)) < \e$.
Take a small number $r > 0$ such that $r \ll 1 \ll 1/\e$.
Since $\Sigma$ is a closed Riemannian manifold, 
$A(p_0; r/2, 2)$ is a Riemannian manifold 
$\approx \Sigma \times [r/2, 2]$ 
with boundary $\approx \Sigma \times \{r/2, 2\}$.

Let $C$ be an annulus $C := A(p; r/2, 2)$. 
Since $d_{GH}(C, A(p_0; r/2, 2)) < \e$, $C$ is $(n - 1, \e)$-strained.
Since $M$ has no boundary points, 
Theorem \ref{regular interior} 
implies that $C$ is $(n, \theta(\e))$-strained.
Therefore by Theorem \ref{fibration theorem}, 
there exists a $\theta(\e)$-Lipschitz submersion 
$\pi : M_1 \to A(p_0; r/2, 2)$ which is actually an $S^1$-fiber bundle.
Here, $M_1$ is some closed domain in $M$
near $C$ containing $A(p; r, 1)$. 

Set $S : = \pi^{-1} (\partial B(p_0, r))$. 
Let $\Phi = \Phi(x, t)$ be a gradient-like flow for $\dist_p$ obtained by Theorem \ref{flow theorem} on an annulus around $p$.

We are going to prove
\begin{lemma} \label{transversality}
The flow $\Phi$ is gradient-like for $\dist_{p_0} \circ \pi$. 
Namely, we obtain the following.
\begin{align}
\liminf_{t \to 0+} 
\frac{\dist_{p_0} \circ \pi \circ \Phi(x, t) - \dist_{p_0} \circ \pi (x)}{t} 
> 1 - \theta(\e) \label{transversality eq0}
\end{align}
for any $x \in M_1$.
\end{lemma} 
If it is proved then $S$ is homeomorphic to $\partial B(p, r)$ by a standard flow argument. 
\begin{proof}[Proof of Lemma \ref{transversality}] 
Let us set $\bar x := \pi(x)$ for any $x \in M_1$.
We set 
\[
V := \left. \frac{d}{d t} \right|_{t = 0+} \!\! \Phi(x, t) \in T_x M.
\]
By Theorem \ref{flow theorem} (ii), 
we have 
\begin{equation*}
V \doteqdot \nabla \dist_p
\end{equation*}
and $|V| \doteqdot 1$.
Here, $A \doteqdot A'$ means that $d(A, A') < \theta(\e)$.

We set $\xi := V/|V|$
and recall that $\xi \in H_x$.
It follows together \eqref{almost tangent 2} that there exists $q \in M$ with $|xq| \geq \sigma$ such that 
any $\bar \xi \in \bar{q}' \subset \Sigma_{\bar x} X$ satisfies 
\[
\pi_x^{(o)} \circ \exp_x^{(o)} (\xi) \doteqdot \exp_{\bar x}^{(o)} (\bar \xi)
\]
for each scale $(o)$.

Let us take $\eta \in p'_x \subset \Sigma_x M$. 
Then we have 
\begin{equation*} 
\angle (\xi, \eta) > \pi - \theta(\e).
\end{equation*}
Since $\eta \in H_x$, there exists $\bar \eta \in \Sigma_{\bar x} X$ such that 
\begin{align*}
\pi_x^{(o)} \circ \exp_x^{(o)} (\eta) 
&\doteqdot \exp_{\bar x}^{(o)} (\bar \eta). 
\end{align*}
By Proposition \ref{almost linear}, we obtain
\begin{equation} \label{almost inverse eq0}
\angle (\bar \xi, \bar \eta) > \pi - \theta(\e).
\end{equation}

On the other hand, from Lemma 4.3 in \cite{Y convergence}, $\pi$ is $\theta(\e)$-close to an $\e$-approximation from $(M, p)$ to $(X, p_0)$. 
This implies 
\begin{equation*} 
\wangle \bar q \bar x p_0 > \pi - \theta(\e).
\end{equation*}
We take an arbitrary direction $\bar \zeta \in p_0' \subset \Sigma_{\bar x} X$.
Then, we have 
\begin{equation} \label{almost inverse eq2}
\angle (\bar \zeta, \bar \xi) > \pi -\theta(\e).
\end{equation}
By \eqref{almost inverse eq0} and \eqref{almost inverse eq2}, we have 
\begin{align*}
\bar \xi &\doteqdot \nabla \dist_{p_0}.
\end{align*}

Summarizing the above arguments, we obtain 
\begin{align*}
\lim_\omega \frac{d(p_0, \pi \circ \Phi(x, \e_i)) - d(p_0, \bar x)}{\e_i} 
&= (\dist_{p_0})_{\bar x}' \circ (\exp_{\bar x}^{(o)})^{-1} \circ \pi_x^{(o)} \circ \exp_x^{(o)} (V) \\
&\doteqdot (\dist_{p_0})_{\bar x}' (\bar \xi) \\
&\doteqdot (\dist_{p_0})_{\bar x}' (\nabla \dist_{p_0}) \\
&= 1.
\end{align*}
It follows from Lemma \ref{liminf} that we obtain \eqref{transversality eq0}.
\end{proof}

As mentioned above, by Lemma \ref{transversality}, 
we have $\partial B(p, r) \approx S$. 
This completes the proof of Theorem \ref{flow application}
\end{proof}

\begin{remark} \upshape
Kapovitch proved a statement similar to Theorem \ref{flow application}
for collapsing Riemannian manifolds $M$ (\cite[Theorem 7.1]{Kap restriction}).
\end{remark}

Perelman and Petrunin proved the existence and uniqueness of a gradient flow of any semiconcave function, especially of any distance function (\cite{Pet QG}, \cite{PP QG}).
Note that the gradient ``flow'' is not a flow in the sense of Definition \ref{definition of flow}, because the gradient flow is defined on $M \times [0,\infty)$. 

\begin{remark} \upshape 
One might ask why we do not use the gradient flow of a distance function to prove Theorem \ref{flow application}. 
The reason is the gradient flow may not be injective.

For instance, we consider the cone $X = K(S^1_{\theta})$ over a circle $S^1_\theta$ with length $\theta < 2 \pi$.
$X$ is expressed by the quotient of a set 
\[
X_0 = \{ r e^{i t} \in \mathbb{C} \,|\,
r \geq 0, t \in [0, \theta] \} 
\]
by a relation $r \sim r e^{i \theta}$ for $r \geq 0$.
By $[r e^{i t}] \in X$ denotes the equivalent class of $r e^{i t} \in X_0$.
We fix $r > 0$ and take $p := r e^{i \theta /2}$.
Let $a > 0$ be a sufficiently small number such that $S_a \cap \partial X_0 = \emptyset$. 
Here, we denote by $S_a$ the circle centered at $p$ with radius $a$ in $\mathbb C$.
We take $b$ with $a < b < r$ such that $S_b \cap \partial X_0 \neq \emptyset$ and take $x_1, x_2$ with $x_1 \neq x_2$ in $S_b \cap \partial X_0$ near $p$.
Then $[x_1] = [x_2]$ in $X$.
We put points $y_i \in px_i \cap S_a$ in $X_0$ and set geodesics $\gamma_i := [y_i][x_i]$ in $X$ for $i =1,2$.
In particular, $\gamma_i$ $(i =1,2)$ are the gradient curves for $d_{[p]}$ in $X$.
This case says that $d_{[p]}$-flow does not injectively send from $[S_a] :=\{ [z] \in X \,|\, x \in S_a \}$ to $[S_b]$.
\end{remark}

\subsection{Flow Arguments} 

\begin{theorem} \label{flow theorem 1.5} 
For a positive integer $n$, there is a positive constant $\e_n$ satisfying the following:
Let $M^n$ be an $n$-dimensional Alexandrov space with curvature $\ge -1$.
Let $A_1, A_2, \dots, A_m \subset M$ be closed subsets and $C \subset M$ be an $(n,\e)$-strained compact subset with $A_i \cap C = \emptyset$ for all $i = 1, 2, \dots, m$ and for $\e \le \e_n$.
Suppose the following.
For each $x \in C$ and $1 \le i \le m$, 
there is a point $w = w(x) \in M$ such that 
\begin{equation}
\angle_x (A_i' , w') > \pi - \de.
\end{equation}
Here, $c\, (< \pi / 2)$ is a positive constant bigger than some constant.
Then there exist an open neighborhood $U$ of $C$ and a Lipschitz flow $\Phi$ on $M$ such that
\begin{equation}
\left. \frac{d}{d t} \right|_{t = 0+} d_{A_i} (\Phi(x, t)) > 
1 - 5 \de - \theta(\e)
\end{equation}
for all $i = 1, \dots m$ and $x \in U$.
\end{theorem}

\begin{proof}
We can show the following:
There exists a precompact open neighborhood $U$ of $C$ such that each $x \in U$ is $(n ,\theta(\e))$-strained, and there is a point $w = w(x) \in M$ such that 
\begin{align}
|x w(x)| &> \ell \\
\wangle A_i y w(x) &> \pi - \de - \theta(\e)
\end{align}
for all $y \in B(x, r)$ and $i = 1, \dots, m$. 
Here, 
$r$ and $\ell$ are positive numbers with $r \ll \ell$.

Since $U$ is $(n, \theta(\e))$-strained, there is a smooth approximation $f : U \to N$ which is a $\theta(\e)$-isometry for some Riemannian manifold $N$.
By an argument similar to the proof of Theorem \ref{flow theorem}, we can construct a smooth vector field $\tilde{W}$ and its integral flow $\tilde \Phi$ on $N$ such that 
\begin{align*}
\left. \frac{d}{d t}\right|_{t = 0+} \dist_{A_i} \circ f^{-1} (\tilde \Phi (x, t)) 
&= \dist_{A_i}' \circ d f^{-1} (\tilde W) 
\\
&> 1 - 5 \de - \theta(\e).
\end{align*}
Then, the pull-back flow $\Phi_t := \tilde \Phi_t \circ f$ satisfies the conclusion of Theorem \ref{flow theorem 1.5}.
\end{proof}

\begin{corollary}\label{flow theorem 1.75} 
Let $M^n$, $A_1$, $A_2$, $\dots A_m$ and $C$ be as same as in the assumption of Theorem \ref{flow theorem 1.5} which satisfy the following.
All $d_{A_i}$ is $(1 - \de)$-regular at $x$.
$m \le n$ and 
\begin{equation}
|\angle_x (A_i', A_j') - \pi /2| < \mu
\end{equation}
for any $x \in C$ and $1 \le i \neq j \le m$.

If $\nu := \de + \mu$ is smaller than some constant depending on $m$, 
then there are a Lipschitz flow $\Phi$ and a neighborhood $U$ of $C$ 
satisfying the following.
\begin{equation}
\left. \frac{d}{d t} \right|_{t = 0+} d_{A_i} (\Phi (x, t)) > 1 - 5 \sqrt {\de} - \theta(\e) - \theta_m(\nu) 
\end{equation}
for any $x \in U$ and $i = 1, \dots, m$.
\end{corollary}
\begin{proof}
Let us consider a smooth approximation 
\[
f : U \to N
\]
for some neighborhood $U$ of $C$ and a Riemannian manifold $N$.
By Lemmas \ref{flow theorem lemma 2} and \ref{flow theorem lemma 1}, we obtain smooth vector fields $\tilde W_i$ on $N$ such that 
\begin{align}
|\tilde W_i| &\le 1 + \theta(\e), \\
(d_{A_i})' (W_i) &> 1 - 5 \sqrt \de - \theta(\e)
\end{align}
on $N$ for all $i = 1, \cdots, m$.
Here, $W_i := d f^{-1} (\tilde W_i)$.

Let us define $\varphi_m(\nu) \in (\pi /2, \pi)$ by
\[
\cos \varphi_m(\nu) = \frac{1 - (m -1) \cos (\pi /2 - \nu)}{\sqrt m \sqrt {1 + (m-1) \cos(\pi /2 - \nu)}}.
\]
Note that $\cos \varphi_m(\nu) \to 1 / \sqrt{m}$ as $\nu \to 0$.

Let us consider a vector field 
\[
\tilde W := (\tilde W_1 + \tilde W_2 + \cdots + \tilde W_m) / |\tilde W_1 + \tilde W_2 + \cdots + \tilde W_m|.
\]
Since $|\angle (A_i', A_j') - \pi / 2| < \mu$, we have 
\[
|\angle (W_i, W_j) | < 10 \de + \mu + \theta(\e).
\]
Putting $W := d f^{-1} (\tilde W)$, we obtain 
\[
\cos \angle (W_i, W) \ge \cos (\varphi_m(\nu) + \theta(\e))
\]
for $\nu = 10 \de + \mu$.
Then we have 
\begin{align*}
(d_{A_i})' (W) &\ge (d_{A_i})' (W_i) - |W, W_i| \\
&\ge 1 - 5 \sqrt \de - \cos(\varphi_m(\nu)) - \theta(\e).
\end{align*}
We consider the gradient flow $\Phi$ of the vector field $W$ on $U$, which is the desired flow.
\end{proof}


\section{The case that $\dim X = 2$ and $\partial X = \emptyset$}
\label{proof of 2-dim interior}
In this and next sections, we study the topologies of three-dimensional closed Alexandrov spaces which collapse to Alexandrov surfaces. 
First, we exhibit examples of three-dimensional 
Alexandrov spaces (which are closed or open) collapsing to 
Alexandrov surfaces. 

We denote a circle of length $\e$ by $S^1_\e$.
We often regard $S^1_\e$ as $\{x \in \mathbb C \,\mid\, \|x\| = \e / 2 \pi \}$. 
And $\bar x$ denotes the complex conjugate for $x \in \mathbb C$.

\begin{example} \upshape \label{collapse M_pt}
Recall that $M_\pt$ is obtained by the quotient space $M_\pt := S^1 \times \mathbb R^2 / (x, y) \sim (\bar x, -y)$.
$M_\pt$ have collapsing metrics $d_\e$ and $\rho_\e$ as follows.

Recall that a collapsing metric provided Example \ref{M_pt}.
The quotient $(M_\pt, d_\e) := S^1_\e \times \mathbb R^2 / (x, y) \sim (\bar x, -y)$ has a metric $d_\e$ of nonnegative curvature collapsing to $K(S^1_\pi) = \mathbb R^2 / y \sim -y$ as $\e \to 0$.

We consider an isometry defined by 
\[
K(S^1_\e) \ni [t ,v] \mapsto [t, -v] \in K(S^1_\e).
\]
Here, $t \ge 0$ and $v \in S^1_\e$. 
Note that $K(S^1_\e)$ collapses to $\mathbb R_+$ as $\e \to 0$.
We consider a metric $\rho_\e$ on $M_\pt$  of nonnegative curvature defined by taking the quotient of the direct product $S^1 \times K(S^1_\e)$:
\[
(M_\pt, \rho_\e) := S^1 \times K(S^1_\e) / (x, t, v) \sim (\bar x, t, -v).
\]
Then, $(M_\pt, \rho_\e)$ collapses to $[0, \pi] \times \mathbb R_+$ as $\e \to 0$. 
Here, $[0, \pi]$ is provided as $S^1 / x \sim \bar x$.
\end{example}

\begin{example} \label{collapse to corner} \upshape
Let $\Sigma (S^1_\e)$ be the spherical suspension of $S^1_\e$, which has curvature $\ge 1$.
Any point in $\Sigma (S^1_\e)$ is expressed as $[t, v]$ parametrized by $t \in [0, \pi]$ and $v \in S^1_\e$.
We consider an isometry 
\[
\alpha : \Sigma (S^1_\e) \ni [t, v] \mapsto [\pi - t, -v] \in \Sigma(S^1_\e).
\]
Then, we obtain a metric $d_\e$ on $P^2$ of curvature $\ge 1$ defined by taking the quotient $\Sigma (S^1_\e) / \langle \alpha \rangle$.
We set $P^2_\e := (P^2, d_\e)$.
Note that $P^2_\e$ collapses to $[0, \pi /2]$ as $\e \to 0$.
Then, $K(P^2_\e)$ collapses to $K([0, \pi /2]) \equiv \mathbb R_+ \times \mathbb R_+$ as $\e \to 0$.

Remark that $K(P^2_\e)$ is isometric to the quotient space $\mathbb R \times K(S^1_\e) / \langle \sigma \rangle$ defined as follows:
Let $\sigma$ be an involution defined by 
\[
\sigma (x, t v) \mapsto (-x, t (-v))
\]
for $x \in \mathbb R$, $t \ge 0$ and $v \in S^1_\e$.
We sometime use this expression in the paper.
\end{example}

\begin{example} \upshape
Let us consider the direct product $S^1 \times \Sigma(S^1_\e)$ and an isometry
\[
\beta : S^1 \times \Sigma (S^1_\e) \ni (x, t, v) \mapsto (\bar x, t , -v) \in S^1 \times \Sigma(S^1_\e).
\]
Then, the quotient space $N_\e := S^1 \times \Sigma (S^1_\e) / \langle \beta \rangle$ has nonnegative curvature.
And, $N_\e$ collapses to $[0, \pi] \times [0, \pi]$ as $\e \to 0$.
\end{example}

Let us start the proof of Theorem \ref{2-dim interior}.
\begin{proof}[Proof of Theorem \ref{2-dim interior}]
Fix a sufficiently small $\de > 0$.
Then there are only finitely many $(2, \de)$-singular points $x_1, \dots, x_k$ in $X$.
For sufficiently small $r > 0$, consider the set $X' := X - (U(x_1, r)
\cup \cdots \cup U(x_k, r))$.
By Theorem \ref{regular interior}, there exists a $(3,
\theta(i,\de))$-strained closed domain $M_i' \subset M_i$ which is
converging to $X'$.
From Theorem \ref{fibration theorem},  we may assume that there exists
a circle fiber bundle $\pi_i' : M_i' \to X'$ which is 
$\theta(i,\de)$-almost Lipschitz submersion.
Here, $\theta(i, \de)$ is a positive constant such that $\lim_{i \to
  \infty, \de \to 0} \theta(i, \de) = 0$.
We fix a large $i$, and use a notation $\theta(\de) = \theta(i, \de)$
for simplicity.

Fix any $(2, \de)$-singular point $p \in \{x_1, \dots. x_k \} \subset
X$, take a sequence $p_i \in M_i$ converging to $p$.
%
%
Since Flow Theorem implies that  $B(p_i, r)$ is not contractible,
applying the rescaling argument \ref{rescaling argument}, we have points
$\hat{p}_i \in B(p_i, r)$ with $d(p_i, \hat{p}_i) \to 0$ and a scaling
constants  $\de_i$ such that any limit space $(Y, y_0)$ of  
$\lim_{i \to  \infty} (\frac{1}{\de_i} B(\hat{p}_i, r), \hat{p}_i)$ is a
three-dimensional open Alexandrov space of nonnegative curvature.
We may assume that $p_i = \hat{p_i}$.
We denote by $S$ a soul of $Y$.
By Theorem \ref{rescaling argument}, we have $\dim S \leq 1$.

From Theorem \ref{flow application}, the boundary $\partial
B(p_i, r)$ is  homeomorphic to a torus $T^2$ or a Klein bottle
$K^2$.
It follows from Soul theorem \ref{soul theorem} and Stability theorem
\ref{stability theorem} that $B(p_i, r)$ is homeomorphic to
the orbifold $B(\pt)$ if $\dim S = 0$, or a solid torus $S^1 \times
D^2$ or a solid Klein bottle $S^1 \tilde{\times} D^2$ if $\dim S = 1$.

We first consider the case of $\dim S = 1$, namely $S$ is a circle.
In this case, we obtain the following conclusion.

\begin{lemma}\label{ball is not a solid Klein bottle}
If $\dim S = 1$,  then $B(p_i, r)$ is homeomorphic 
to $S^1 \times D^2$. 
\end{lemma}

\begin{proof}
Put $B_i := B(p_i, r)$, $B := B(p, r)$ and $\e_i := d_{GH}(B_i, B)$.
Suppose that $B_i$ is homeomorphic to a solid Klein bottle $S^1 \tilde{\times} D^2$.
Take $r_i \to 0$ with $\e_i / r_i \to 0$ such that
$\lim (\frac{1}{r_i} B_i, p_i) = (T_p X, o)$.
Let $\pi_i : \tilde{B}_i \to B_i$ be a universal covering and
$\tilde{p}_i \in \pi_i^{-1}(p_i)$.
Let $\Gamma_i \cong \mathbb{Z}$ be the deck transformation group of $\pi_i$.
Passing to a subsequence, 
we have a limit triple $(Z, z, G)$ of a sequence of
triples of pointed spaces and isometry groups $(\frac{1}{r_i}
\tilde{B}_i, \tilde{p}_i, \Gamma_i)$ in the equivariant pointed
Gromov-Hausdorff topology (cf. \cite{FY}).
$Z$ is an Alexandrov space of nonnegative curvature because of $r_i
\to 0$,
and $G$ is abelian.
Note that $Z / G = \lim (\frac{1}{r_i} B_i, p_i) = (T_p X, o)$.
Using the $G$-action, we find a line in $Z$ (\cite{CG}).
Then, by the splitting theorem, there is some nonnegatively curved
Alexandrov space $Z_0$ such that $Z$ is isometric to the product
$\mathbb{R} \times Z_0$.
We may assume that $Z_0$ is a cone, by taking a suitable rescaling $\{ r_i \}$.
We denote by $G_0$ the identity component of $G$.
By \cite[Lemma 3.10]{FY}, there is a subgroup $\Gamma_i^0$ of $\Gamma_i$ such that:
\begin{itemize}
 \item[(1)] $(\frac{1}{r_i} \tilde{B}_i , \tilde{p}_i, \Gamma_i^0)$ 
   converges to $(Z, z, G_0)$. 
 \item[(2)] $\Gamma_i / \Gamma_i^0 \cong G / G_0$ for large $i$.
\end{itemize}
Since $\dim T_p X = 2$ and $\dim Z =3$, we have $\dim G = 1$. 
This implies $G \cong \mathbb{R} \times H$ for some finite abelian
group $H$.
Since $T_pX = Z_0/H$, $H$ must be cyclic.
Here, $G$-action is component-wise: $G_0 \cong \mathbb{R}$ acts by
translation of the line $\mathbb{R}$ and $H$ acts on $Z_0$
independently. 
By Stability Theorem \ref{stability theorem}, $Z$ is simply-connected. 
Therefore, $Z_0$ is homeomorphic to $\mathbb{R}^2$.

Take a generator $\gamma_i$ of $\Gamma_i$. 
From our assumption,  $\gamma_i$ is orientation reversing on $\tilde{B}_i$. 
Consider $\Gamma_i' := \langle \gamma_i^2\rangle \cong \mathbb{Z}$.
Then $\Gamma_i'$ acts on $\tilde{B}_i$ preserving orientation. 
Taking a subsequence, we have a limit triple $(Z, z, G')$ of a sequence
$\{ (\frac{1}{r_i} \tilde{B}_i , \tilde{p}_i, \Gamma_i') \}$.
By an argument similar to the above, $G' \cong \mathbb{R} \times H'$ for
some finite cyclic  group $H'$.
Let $\lim_{i \to \infty} \gamma_i =\gamma_{\infty} \in G$, which
implies that $\gamma_i (x_i) \to \gamma_{\infty} (x_{\infty})$ under the 
Gromov-Hausdorff convergence  $\frac{1}{r_i} \tilde{B}_i \ni
x_i \to x_{\infty} \in Z$.
Then $\gamma_{\infty}$ is represented by $(0, \phi) \in \mathbb{R} \times H$. 
Then, for large $i$, we have
\begin{equation}
Z / G = ( Z / G' ) / (G / G') = (Z_0 / H') / \langle [\phi] \rangle = T_p X.
\end{equation}
Since $Z_0 / H'$ is the flat cone over a circle or an interval, 
and $[\phi] \in H / H'$ acts on $Z_0 / H'$ reversing orientation, 
$(Z_0 / H') / \langle \phi \rangle$ can not be $T_p X$.
This is a contradiction.
\end{proof}

By Lemma \ref{ball is not a solid Klein bottle}, $B(p_i, r/2)$ must be
homeomorphic to a solid torus.
From Flow Theorem \ref{flow theorem}, $(\pi_i')^{-1}(\partial B(p,r))$ and $\partial
B(p_i, r/2)$ bound a closed domain homeomorphic to $T^2\times [0,1]$,
and this provides a circle fiber structure on $\partial B(p_i,r/2)$.
By \cite[Lemma 4.4]{SY}, it extends to a topological
Seifert structure on $B(p_i, r/2)$
over $B(p, r/2)$ 
which is compatible to the circle bundle structure on $A(p_i; r/2, r)$.
\par
\medskip
In the case of $\dim S = 0$, $B_i$ is homeomorphic to $B(\pt)$.
We must prove that

\begin{lemma} \label{ball is B(pt)}
 If $\dim S = 0$, then $B_i$ has the structure of circle fibration
  with a singular arc fiber satisfying
 \begin{itemize}
  \item[(1)]   it is isomorphic to the standard fiber structure on
    $B(pt)=S^1\times D^2/\mathbb Z_2 ;$
  \item[(2)]  it is  compatible to the structure of circle fiber bundle
    $\pi_i'$ near the boundary.
 \end{itemize}
\end{lemma}

\begin{proof}

Recall that $B(pt) =  S^1\times D^2/\mathbb Z_2$,
where $\mathbb Z_2$-action on $S^1\times D^2$ is given by the 
involution $\hat \sigma$ defined by 
$\hat \sigma(x, y) = (\bar x, -y)$.
Let $p_{+} :=(1,0)$, $p_{-} :=(-1, 0)$ be the fixed points of 
$\hat \sigma$.
Putting $\hat U:=S^1\times D^2\setminus \{ p_{+}, p_{-}\}$, and 
$U:=\hat U/\mathbb Z_2$, let $\hat \pi:\hat U \to U$ be the projection map.
Fix a homeomorphism
$f_i:S^1\times D^2/\mathbb Z_2 \to B_i$, and set
$U_i := f_i(U)$. 
Take a $\mathbb Z_2$-covering $\hat \pi_i: \hat U_i \to U_i$ such that
there is a homeomorphism $\hat f_i:\hat U\to \hat U_i$ 
together with the following commutaitve diagram:
\begin{equation*}
 \begin{CD}
    \hat U @>{\hat f_i}>> \hat U_i \\
    @V{\hat \pi}V V
        @V V{\hat\pi_i}V \\
    U @>>{f_i}> U_i
 \end{CD}
\end{equation*}
Consider the length-metric on $\hat U_i$  induced from that of $U_i$ via
$\hat\pi_i$, and the length-metrics of $U$ and $\hat U$ for which 
both $f_i$ and $\hat f_i$ become isometries.
Note that $U_i=\hat U_i/\hat\sigma_i$, 
where $\hat\sigma_i:= \hat f_i\circ \hat\sigma \circ (\hat f_i)^{-1}$.
$\hat\sigma_i$ extends to an isometry on the completion $\hat B_i$  of 
$\hat U_i$. 
Let $\hat\pi_i:\hat B_i \to B_i$ also denote the 
the projection.
Then we have the following commutative diagram:
\begin{equation} \label{eq:comm-diagram}
 \begin{CD}
   \mathbb R \times D^2 @>{\tilde f_i}>> \tilde B_i \\
    @V{\pi}V V
        @V V{\pi_i}V \\
      S^1\times D^2 @>{\hat f_i}>> \hat B_i \\
    @V{\hat \pi}V V
        @V V{\hat\pi_i}V \\
    S^1\times D^2/\mathbb Z_2    @>>{f_i}> B_i,
 \end{CD}
\end{equation}
where $\pi_i:\tilde B_i \to \hat B_i$ is the universal covering, and
$\tilde f_i$ is an isometry covering $\hat f_i$. Here we consider the metric on 
$\mathbb R\times D^2$ induced by that of $S^1\times D^2$.
Let $\sigma, \lambda :\mathbb R \times D^2 \to \mathbb R \times D^2$
be defined as 
\[
     \sigma(x,y)= (-x, -y), \,\,\, \lambda(x, y)=(x+1, y).
\]
Since $\sigma$ covers $\hat\sigma$, $\sigma$ is an isometry.
Put 
\[
     \sigma_i:= \tilde f_i\circ \sigma \circ (\tilde f_i)^{-1}, \,\,\,
     \lambda_i:= \tilde f_i\circ \lambda \circ (\tilde f_i)^{-1}.
\]
From construction, the group
$\Lambda_i$ generated by $\lambda_i$ is the deck transformation 
group of $\pi_i:\tilde B_i \to \hat B_i$.
Let $\Lambda$ be the group  generated by $\lambda$.
Let $\Gamma_i$ (resp. $\Gamma$) be the group generated by $\sigma_i$
and $\Lambda_i$ (resp. by $\sigma$ and $\Lambda$).
Obviously we have an isomorphism $(\Gamma_i, \Lambda_i)\simeq 
(\Gamma, \Lambda)$.  Note that 
\begin{equation}
\sigma\lambda \sigma^{-1} = \lambda^{-1}.
      \label{eq:gamsig}
\end{equation}
Let us consider the limit of the action of $(\Gamma_i, \Lambda_i)$
on $\tilde B_i$.
We may assume that 
$(\tilde B_i,\tilde p_i,\Gamma_i, \Lambda_i)$ converges to 
$(Z, z_0,\Gamma_{\infty}, \Lambda_{\infty})$,
where $Z = \mathbb R \times L$, $\Lambda_{\infty}=\mathbb R \times H$,
$L$ is a flat cone and $H$ is a finite cyclic group acting 
on $L$. Let $\sigma_{\infty}\in \Gamma_{\infty}$ and
$\lambda_{\infty}\in\Lambda_{\infty}$  be the limits of 
$\sigma_i$ and $\lambda_i$ under the above convergence.
Note that $\sigma_{\infty}:\mathbb R \times L\to \mathbb R \times L$ can be 
expressed as 
$\sigma_{\infty}(x,y)=(-x, \sigma_{\infty}'(y))$,
where $\sigma_{\infty}'$ is a rotation of angle $\ell/2$
and $\ell$  is the length of the space of directions at 
the vertex  of the cone $L$. Note that 
$T_pX=(L/H)/\sigma_{\infty}'$.

As discussed above, from the action of $H$ on $L$, we can put 
a Seifert fibered torus structure on $\partial\hat B_i$.
Namely if 
$\lambda_{\infty}(re^{i\theta}) = re^{i(\theta+\nu\ell/\mu)}$,
then  $\partial\hat B_i$ has a Seifert fibered torus structure
of type $(\mu,\nu)$ that is $\hat\sigma_i$-invariant
(See \cite[Lemma 4.4]{SY}).
From \eqref{eq:gamsig}, we have 
$\sigma_{\infty}\lambda_{\infty}\sigma_{\infty}=\lambda_{\infty}^{-1}$.
This yields that $\lambda_{\infty}^2 = 1$.
Thus $(\mu,\nu)$ is equal to $(1,1)$ or $(2,1)$.

We shall show $(\mu,\nu)=(1,1)$ and extend the fiber structure
on $\partial \hat B_i$ to a $\hat\sigma_i$-invariant 
fiber structure on $\hat B_i$ which projects down to 
the generalized Seifert bundle structure on $B_i$.

Let $B$ and $\hat B$ be the $r$-balls in the 
cone $T_pX = (L/H)/\sigma_{\infty}'$ and 
$L/H$ around the vertices $o_p$ and $\hat o_p$
respectively. 
Consider the metric annuli
\[
       A:= A(o_p;r/4,r), \,\,  \hat A:= A(\hat o_p;r/4,r).
\]
Applying the equivariant fibration theorem (Theorem 18.4 in \cite{Y 4-dim}),
we have a $\mathbb Z_2$-equivariant $S^1$-fibration
$\hat g_i:\hat A_i \to \hat A$ for some closed domain 
$\hat A_i$ of $\hat B_i$, which gives rise to an
$S^1$-fibration
$g_i:A_i \to A$ for some closed domain 
$A_i$ of $B_i$. 

We denote by $B(\pi_i')$ and $B(g_i)$ the closed domain bounded 
by $(\pi_i')^{-1}(S(p, r))$ and  $(g_i)^{-1}(S(o_p, r/2))$
respectively, and set
\[
      A(\pi_i', g_i) := \overline{ B(\pi_i') \setminus B(g_i)}.
\]

By Flow Theorem \ref{flow theorem}, there is a Lipschitz flow
$\Phi:  \partial B(\pi_i')\times [0,1] \to A(\pi_i', g_i)$
such that $\Phi(x, 0) =x$. Let $\Phi_1:\partial B(\pi_i') \to
\partial B(g_i)$
be the homeomorpshism defined by $\Phi_1(x) = \Phi(x,1)$.
Obviously the $\pi_i'$-fibers of $(\pi_i')^{-1}(S(p, r))$ and the $(\Phi_1)^{-1}$-images
of $g_i$-fibers of  $(g_i)^{-1}(S(o_p, r/2))$  are isotopic each
other. Namely we have an isotopy $\varphi_t$ of $\partial B(\pi_i)$,
$0\le t\le 1$, such that $\varphi_0 = id$ and 
$\varphi_1$ sends every  $\pi_i'$-fiber to the  $(\Phi_1)^{-1}$-image
of a  $g_i$-fiber. Define $\Psi:A(\pi_i', g_i) \to A(\pi_i', g_i)$
by
\[
         \Psi(\Phi(x,t)) = \Phi(\varphi_t(x), t).
\]
This joins  the two fiber structures of $\pi_i'$ and $g_i$. Thus we
obtain  a circle fibration 
$\pi'': M_i'' \to X''$ gluing the fibrations 
$\pi_i': M_i' \to X'$ and $g_i$, 
where $X''= X-(U(x_1, r/4)\cup\cdots\cup U(x_k,r/4))$.

Let $V_{\mu,\nu}=S^1\times D^2$ denote the fibered solid torus 
of type $(\mu,\nu)$.

From now on, for simplicity, we denote $B(g_i)$ by $B_i$, and 
use the same notaiton as in \eqref{eq:comm-diagram}.
In particular, we have the $\mathbb Z_2$-equivariant homeomorphism
$\hat f_i:  V_{\mu,\nu} \to \hat B_i$.
Using $\hat f_i$, we have a fiber structure on $\partial V_{\mu,\nu}$
induced from the $\hat g_i$-fibers which is isotopic to 
the standard fiber structure of type $(\mu,\nu)$. 

\begin{assertion}
$(\mu,\nu) = (1,1)$ and there is a $\hat\sigma$-equivariant isotopy 
of $\partial V_{1,1}$ joining the two fiber structures on 
 $\partial V_{1,1}$.
\end{assertion}

\begin{proof}
First suppose $(\mu, \nu) = (1,1)$. 
On the torus $\partial V_{1,1}=S^1\times \partial D^2$, 
let $m=m(t)=(1, e^{it})$ and $\ell=\ell(t)=(e^{it}, 1)$ 
denote the meridian and the longitude.
Fix a meridian $m_i$ and a longitude $\ell_i$ of 
$\partial \hat B_i$ such
that each fiber of $\hat g_i$ transversally meets $m_i$.
Here we may assume that all the longitude of $\partial \hat B_i$
discussed below are $\hat g_i$-fibers.

Set $h_i := (\hat f_i)^{-1}$ for simplicity.
We now show  that $h_i(\ell_i)$ is $\hat \sigma$-equivariantly
ambient isotopic to $\ell$. 
Recall that $\hat \pi: \partial V_{1,1}=S^1\times\partial D^2\to K^2=S^1\times\partial D^2/\hat\sigma$
is the projection. 
Since $h_i(\ell_i)$ is homotopic to $\ell$, $\hat\pi(h_i(\ell_i))$ is homotopic to 
$\hat\pi(\ell)$, and hence is ambient isotopic to $\hat\pi(\ell)$.
Namely, there exists an isotopy $\varphi_t$, $0\le t\le 1$, of $K^2$
such that 
\[
   \varphi_0=id, \,\,\, \varphi_1(\hat\pi(h_i(\ell_i)))=\hat\pi(\ell).
\]
Let $\hat\varphi_t : \partial V_{1,1} \to \partial V_{1,1}$ be the lift of 
$\varphi_t$ such that $\hat\varphi_0=id$. Note that
$\hat\varphi_1(h_i(\ell_i))=\ell$.
Therefore we may assume that $h_i(\ell_i)=\ell$ from the beginning.

Next we claim that $h_i(m_i)$ is $\hat\sigma$-equivariantly
ambient isotopic to $m$ while keeping $\ell$ fixed.
Namely we show that there exists an isotopy $\hat\varphi_t$, $0 \le t \le 1$, of $\partial V_{1,1}$
such that 
\[
     \hat\varphi_0=id,\,\,\, \hat\varphi_1(h_i(m_i))=m, \,\,\,
       \hat\varphi|_{\ell}=1_{\ell}.
\]
To show this, we proceed in a way similar to the above. 
Since $h_i(m_i)$ is homotopic to $m$, $\hat\pi(h_i(m_i))$ is homotopic to 
$\hat\pi(m)$, and hence is ambient isotopic to $\hat\pi(m)$. 
Here the construction of isotopy is local (see \cite{Epstein}).  Hence 
approximating $m$ near the intersection point  $\ell\cap m$ via a
PL-arc for instance,
we can choose such an isotopy $\varphi_t$, $0\le t\le 1$, of $K^2$
that 
\[
   \varphi_0=id, \,\, \varphi_1(\hat\pi(h_i(m_i)))=\hat\pi(m),\,\,\,
     \varphi_t|_{\hat\pi(\ell)}= 1_{\hat\pi(\ell)}.
\]
Let 
$\hat\varphi_t:\partial V_{1,1} \to \partial V_{1,1}$ be the lift of 
$\varphi_t$ such that $\hat\varphi_0=id$.
Note that $\hat\varphi_1$ sends $h_i(m_i)$ to $m$ and $\hat\varphi_t$ is the 
required isotopy.
Therefore we may assume that  $h_i(m_i)=m$ from the beginning.

For a small $\epsilon>0$, let $\ell'=(e^{it}, e^{i\epsilon})$ and 
$\ell''=(e^{it}, e^{-i\epsilon})$ 
(resp. $m'=(e^{i\epsilon}, e^{it})$ and $m''=(e^{-i\epsilon}, e^{it})$)
be longitudes near $\ell$ (resp. meridians near $m$).  
Let $\ell_i'$ and $\ell_i''$ (resp. $m_i'$ and $m_i''$)
be longitudes (resp.  meridians) near $\ell_i$ (resp. near $m_i$)
such that $\ell_i'$, $\ell_i''$, $m_i'$ and $m_i''$
bound a regular neighborhood of $\ell_i\cup m_i$.
In a way similar to the above, taking a $\hat\sigma$-equivariant ambient isotopy,
we may assume that $h_i(\ell_i')=\ell'$, $h_i(\ell_i'')=\ell''$, 
$h_i(m_i')=m'$ and $h_i(m_i'')=m''$.

Let $D$ (resp. $D_i$) be the small domain bounded by
$\ell$, $\ell'$, $m$ and $m'$ (resp. $\ell_i$, $\ell_i'$, $m_i$ and
$m_i'$ ).
Identify $D= I_0\times [0,1]$, $D_i= I_i\times [0,1]$, where
$I_0\subset m$, $I_i\subset m_i$ be arcs, 
and define
$k_i:D\to D$ by $k_i(x, t)=h_i(\hat f_i(x), t)$.
From what we have discussed above, 
 $k_i|_{\partial D} = 1_{\partial D}$.
It is then standard to obtain an isotopy $\psi_t$  of $D$ 
which sends $k_i$ to $1_D$ keeping $\partial D$ fixed.
Extending $\psi_t$ $\hat\sigma$-equivariantly, we obtain a 
$\hat\sigma$-equivariant isotopy of $\partial V_{1,1}$
which sends the $h_i$-image of $I$-fibers of $D_i$
to $I$-fibers of $D$ keeping the outside $D$ fixed.
Applying this argument to the other  domains
bounded by those  longitudes $\ell'$, $\ell''$,  $\hat\sigma(\ell')$,
$\hat\sigma(\ell'')$ 
and meridians $m'$, $m''$ of $\partial V_{1,1}$,
we finally construct  a $\hat\sigma$-equivariant ambient isotopy
$\varphi_t$ of $\partial V_{1,1}$ 
sending  the  $h_i$-images of the  $\hat g_i$-fibers in $\partial B_i$
to the corresponding longitudes of $\partial V_{1,1}$.
\par
\medskip

Finally we show that the case $(\mu,\nu)=(2,1)$ never happens.
Let us fix a $g_i$-fiber, say $k_i$,  and a standard $(2,1)$-fiber,
say $k$, on the fibered torus $\partial V_{2,1}$ of type $(2,1)$.
Since $h_i(k_i)$ is homotopic to $k$ in $T^2$, in a way similar to the above discussuin, 
we have a $\hat\sigma$-equivariant ambient isotopy $\hat\varphi_t$ of 
$\partial V_{2,1}$ such that $\hat\varphi_0 = id$ and 
$\hat\varphi_1$ sends $h_i(k_i)$ to $k$.
In $S^1\times \partial D^2$, $k$ is described as 
$k(t)=(e^{2it}, e^{it})$, and hence
$\hat\sigma\circ k(t) = (e^{-2it}, e^{i(t+\pi)})$.
Therefore the images ${\rm Im}(\hat\sigma\circ k)$ , ${\rm Im}(k)$ of 
$\hat\sigma\circ k$ and $k$ respectively must meet at 
$\hat\sigma\circ k(-\pi) = k(2\pi)$.
On the other hand, 
\[
      \hat\sigma\circ k = \hat\sigma\circ \hat\varphi_1(h_i(k_i)) 
              = \hat\varphi_1\circ\hat\sigma(h_i(k_i)), \,\,\,
                      k= \hat\varphi_1(h_i(k_i)).
\]
It turns out that ${\rm Im}(\hat\sigma(h_i(k_i))) = {\rm  Im}(h_i(\hat\sigma_i(k_i)))$ 
meets ${\rm Im}(h_i(k_i))$.
This implies that ${\rm Im}(\hat\sigma_i(k_i))$ meets ${\rm Im}(k_i)$,
a contradiction to the fact that $g_i$ is a ${\mathbb Z}_2$-equivariant fibration.

This completes the proof of the assertion.
\end{proof}
\par
\medskip

Obviously the standard fiber structure on  $\partial V_{1,1}$
extends to a standard 
$\hat\sigma$-invariant fiber structure on  $V_{1,1}$.
Now it becomes  easy to extend the fiber structure 
defined by $\hat g_i$-fibers on $\hat \partial B_i$
to a $\hat\sigma_i$-equivariant fiber structure on 
$\hat B_i$ of type $(1, 1)$ via $h_i$,
which  projects down to   
a generalized Seifert bundle structure
on $B_i$ and on $M_i$ for large $i$ which is compatible to the fiber
structure of $\pi_i'$.
This completes the proof of  Lemma \ref{ball is B(pt)}.
\end{proof}

This completes the proof of Theorem \ref{2-dim interior}.
\end{proof}






\section{The case that $\dim X = 2$ and $\partial X \neq \emptyset$}
\label{proof of 2-dim boundary}
Let $\{M_i \,|\, i = 1, 2, \dots \}$ be a sequence of 
three-dimensional closed Alexandrov spaces with curvature $\geq -1$ having a uniform diameter bound.
Suppose that $M_i$ converges to an Alexandrov surface $X$ with non-empty boundary.

In this section, we provide decompositions of $X$ into $X' \cup X''$ and of $M_i$ into $M_i' \cup M_i''$ such that $M_i'$ fibers over $X'$ in the sense of a generalized Seifert fiber space and $M_i''$ is the closure of the complement of $M_i'$. 
We will prove that each component of $M_i''$ has the structure of a generalized solid torus or a generalized solid Klein bottle, and the circle fiber structure on its boundary is compatible to the circle fiber structure induced by the generalized Seifert fibration.

From now on, we denote by $C$ one of components of $\partial X$. 
Since a two-dimensional Alexandrov space is a manifold, $C$ is homeomorphic to a circle.
Let us fix a small positive number $\e$.
To construct the desired decompositions of $X$ and $M_i$, 
we define a notion of an $\e$-regular covering of $C$.

\begin{definition} \label{def of regular covering} \upshape
Let $\{B_\alpha, D_\alpha\}_{1 \le \alpha \le n}$ be a covering of $C$ by closed subsets in $X$. We say that $\{B_\alpha, D_\alpha\}_{1 \le \alpha \le n}$ is $\e$-{\it regular} if it satisfies the following. 
\begin{itemize}
\item[(1)]
$\bigcup_{1 \le \alpha \le n} B_\alpha \cup D_\alpha - C$ is $(2, \e)$-strained.
\item[(2)]
Each $B_\alpha$ is the closed metric ball $B_\alpha = B(p_\alpha, r_\alpha)$ centered at $p_\alpha$ with radius $r_\alpha > 0$ such that 
\begin{align*}
&|\nabla d_{p_\alpha}| > 1 -\e \text{ on } B(p_\alpha, 2 r_\alpha) - \{p_\alpha\}, \\
& B_\alpha \cap B_{\alpha'} = \emptyset \text{ for all } \alpha \neq \alpha'. 
\end{align*}
And, the sequence $p_1, p_2, \dots p_n$ is consecutive in $C$. 
\item[(3)]
$D_\alpha$ forms 
\[
D_\alpha := B(\gamma_\alpha, \de) - \mathrm{int} (B_\alpha \cup B_{\alpha + 1}), 
\]
where, $\gamma_\alpha := \widehat{p_\alpha p_{\alpha + 1}}$ with $p_{n + 1} := p_1$.
Here, $\de > 0$ is a small positive number with $\de \ll \min_\alpha r_\alpha$.
\item[(4)]
For any $x \in D_\alpha$, we have
\begin{align*}
& \wangle p_\alpha x p_{\alpha + 1} > \pi - \e.
\end{align*}
For $x \in D_\alpha - C$ and $y \in C$ with $|x y| = |x C|$, we have
\begin{align*}
& |\nabla d_C|(x) > 1 - \e, \\
& |\wangle p_\alpha x y - \pi /2| < \e, \text{ and } \\
& |\wangle p_{\alpha + 1} x y - \pi /2| < \e.
\end{align*}
\end{itemize}
\end{definition}

The existence of an $\e$-regular covering of $C$ will be proved in Section \ref{appendix}.
We fix an $\e$-regular covering 
\[
\{B_\alpha, D_\alpha \mid \alpha = 1, 2, \dots, n \}
\]
of $C$. 

We consider a closed neighborhood $X_C''$ of $C$ defined as 
\begin{equation} \label{X''}
X_C'' := \bigcup_{\alpha = 1}^{n} B_\alpha \cup D_\alpha.
\end{equation}
And we set 
\[
X'' := \bigcup X_C'', \text{ and } X' := \text{ the closure of } X - X''.
\]
This is our decomposition $X = X' \cup X''$.

Since $\mathrm{int}\, X'$ has all interior $(2, \e)$-singular points of $X$, by Theorem \ref{2-dim interior}, we obtain a generalized Seifert fibration 
\begin{equation} \label{fibration 2-dim boundary}
\pi_i' : M_i' \to X'
\end{equation}
for some closed domain $M_i' \subset M_i$.
Let us denote by $X^\reg$ 
the complement of a small neighborhood of the union of $\partial X$ and 
the set of all interior $(2, \e)$-singular points in $X$.
By Theorem \ref{fibration theorem}, we may assume that ${\pi_i'}$ is both a circle fibration and a $\theta(\e)$-Lipschitz submersion on $X^\reg$.
Recall that ${\pi_i'}^{-1}(X^\reg)$ is $(3, \theta(\e))$-regular, for large $i$.

We set $M_i'' := M_i - \mathrm{int}\, M_i'$.
We will determine the topology of $M_i''$ in the rest subsections.

\subsection{Decomposition of $M_i''$} \label{construction of M_i''} 
\mbox{}

Let us denote by $M_i^\reg$ a $(3, \theta(\e))$-regular closed domain of $M_i$ which contains ${\pi_i'}^{-1}(X^\reg)$. 
By Theorem \ref{smooth approximation}, we obtain a smooth approximation 
\begin{equation} \label{smooth approx 2-dim boundary}
f_i : U(M_i^\reg) \to N(M_i^\reg)
\end{equation}
for a neighborhood $U(M_i^\reg)$ of $M_i^\reg$ and some Riemannian manifold $N(M_i^\reg)$.

Let us take $p_{\alpha, i} \in M_i$ converging $p_\alpha \in C \subset \partial X$, and $\gamma_{\alpha, i}$ a simple arc joining $p_{\alpha, i}$ and $p_{\alpha + 1, i}$ converging to $\gamma_\alpha$.
By the definition of regular covering, 
we may assume that 
\[
A \left( \bigcup_{\alpha = 1}^N \gamma_{\alpha, i} ; \de / 100, 10 \max r_\alpha \right) 
\]
is $(3, \theta(\e))$-regular.

From now on, we fix any index $\alpha \in \{1, \dots, N\}$ and use the following notations: $p := p_\alpha$, $p' :=p_{\alpha + 1}$, $B := B_\alpha$, $B' := B_{\alpha + 1}$, $\gamma := \gamma_\alpha$ and $\gamma'' := \gamma_{\alpha - 1}$; and 
$p_i := p_{\alpha, i}$, $p_i' := p_{\alpha + 1, i}$, $\gamma_i := \gamma_{\alpha, i}$ 
and $\gamma_i'' := \gamma_{\alpha -1, i}$. 
To avoid a disordered notation, we assume that 
all $r_\alpha$ are equal to each other, 
and set $r := r_\alpha$.

Let $\de'$ be a small positive number with $\de' \ll \de$.
We will construct an isotopy of $B(p_i, r + \de')$ which deform the metric ball $B(p_i, r-\de')$ to some domain $B_i$  such that 
\begin{align}
B_i &\approx B(p_i, r); \label{B_i eq3} \\
\partial B_i - U(\gamma_i \cup \gamma_i'', 3 \de /2 ) 
&= 
{\pi_i'}^{-1} (\partial B(p, r) - U(\gamma \cup \gamma'', 3 \de /2)); \label{B_i eq1} \\
\partial B_i \cap B(\gamma_i \cup \gamma_i'', \de) 
&= 
\partial B(p_i, r - \de') \cap B(\gamma_i \cup \gamma_i'', \de). \label{B_i eq2}
\end{align}

\begin{figure}[h]
\includegraphics[width=0.7\textwidth]{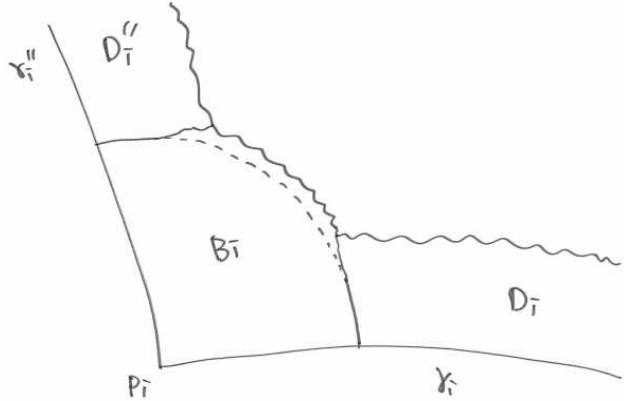}
\caption{: A domain near the corner}
\label{fig_corner}
\end{figure}

In Figure \ref{fig_corner}, 
the broken line denotes the metric sphere $S(p_i, r - \de')$ 
and the wavy line denotes the pull back of metric levels with respect to 
$\gamma''$, $\gamma$ and $p$ in $X$ by $\pi_i'$.

Suppose that we construct such an isotopy and obtain a domain $B_i = B_{\alpha, i}$ satisfying \eqref{B_i eq1} and \eqref{B_i eq2} for a moment. 
We consider the domain 
\begin{equation}
D_i = D_{\alpha, i} :=B(\gamma_{\alpha, i}, 3\de /2 ) \cup {\pi_i'}^{-1} (A(\gamma_\alpha; \de, 2\de)) - \mathrm{int}\, (B_{\alpha, i} \cup B_{\alpha + 1, i}).
\end{equation}
Then we obtain a decomposition of $M_i''$:
\begin{align}
M_{i, C}'' &:= \bigcup_{\alpha =1}^N B_{\alpha, i} \cup D_{\alpha, i}, \\
M_i'' &= \bigcup_{C \subset \partial X} M_{i, C}''.
\end{align}

Now we construct an isotopy which deforms $B(p_i, r -\de')$ to $B_i$ satisfying \eqref{B_i eq1} and \eqref{B_i eq2}.
From now on throughout this paper, we use the following notations.
For any set $A \subset M_i^\reg$, 
we set $\tilde A := f_i(A)$.
We denote by $U(A)$ a neighborhood of $A$ in $U(M_i^\reg)$ and $N(A)$ by the image of $U(A)$ by the approximation $f_i$.
Namely, $N(A) = \widetilde {U(A)}$.
For any point $x \in A$, we set $\tilde{x} := f_i(x) \in \tilde{A}$.
For any function $\phi : A \to \mathbb{R}$, we define $\tilde{\phi} : \tilde{A} \to \mathbb{R}$ by 
\begin{equation}
\tilde{\phi} := \phi \circ f_i^{-1}. \label{tilde map}
\end{equation}

Let $\tilde{V}$ be a gradient-like smooth vector fields for a Lipschitz function $\widetilde{\dist}_{p_i}$ on $N ( (B(p_i, r) \cup B(p_i', r) \cup B(\gamma_i, 2 \de)) \cap M_i^\reg )$ obtained by Lemma \ref{flow theorem lemma 2}.

We take a Lipschitz function $h$ defined on $B(p_i, r + \de')$ such that 
\begin{equation*} 
\begin{aligned}
&\tilde{h} \text{ is smooth}, \\
& 0 \leq h \leq 1, \\
& \mathrm{supp}\, (h) \subset B(p_i, r + \de') - U(\gamma_i \cup \gamma_i'', \de/2), \\
& h \equiv 1 \text{ on } B(p_i, r + \de') - U(\gamma_i \cup \gamma_i'' ,\de).
\end{aligned}
\end{equation*}
We consider a smooth vector field $\tilde{h} \cdot \tilde V$ and its integral flow $\tilde{\Phi}$. And we define the pull-back flow $\Phi_t := f_i^{-1} \circ \tilde{\Phi}_t \circ f_i$. 
Then by construction and Theorem \ref{flow application}, the flow $\Phi$ transversally intersects ${\pi_i'}^{-1} (\partial B(p, r) - U(\gamma \cup \gamma'', \de))$.
Then we can construct an isotopy by using the flow $\Phi$, which provide a closed neighborhood $B_i$ of $p_i$ satisfying \eqref{B_i eq3}, \eqref{B_i eq1} and \eqref{B_i eq2}.

\subsection{The topologies of the balls near corners} \label{topology of ball near boundary}

We first prove that $\partial B_i$ is homeomorphic to a closed 2-manifold.
\begin{lemma} \label{boundary is 2-manifold}
$\partial B_i$ $\approx \partial B(p_i, r)$ 
is a closed 2-manifold.
\end{lemma}
\begin{proof}
If $B_i$ does not satisfy Assumption \ref{rescaling assumption}, we have some sequence $\hat{p}_i$ with $|\hat{p}_i p_i| \to 0$ such that $\partial B(\hat{p}_i, r) \approx \Sigma_{\hat{p}_i}$, where we may assume that $\hat{p}_i = p_i$.
Since $M_i$ has no boundary, $\partial B(p_i, r)$ is homeomorphic to $S^2$ or $P^2$. 

If $B_i$ satisfies Assumption \ref{rescaling assumption}, there exist a sequence $\de_i \to 0$ and $\hat{p}_i$ with $|\hat{p}_i p_i| \to 0$ such that the limit $(Y, y_0)$ of $( \frac{1}{\de_i} B(\hat{p}_i, r), \hat{p}_i )$ has dimension three. 
Here, we may assume that $\hat{p}_i = p_i$.
Then, by Soul Theorem \ref{soul theorem} and Stability Theorem \ref{stability theorem}, $\partial B(p_i, r)$ is homeomorphic to $S^2$, $P^2$, $T^2$ or $K^2$. 
\end{proof}

From \eqref{B_i eq1} and the construction of $B_i$, we have 
\begin{equation}
\partial B_i - U(\gamma_i \cup \gamma_i'', \de) \approx S^1 \times I.
\end{equation}
Now, we put $F_i$ and $F_i''$ as follows.
\begin{equation}
F_i := \partial B_i \cap B(\gamma_i, \de) \text{ and }
F_i'' := \partial B_i \cap B(\gamma_i'', \de).
\end{equation} 
Then, by Lemma \ref{boundary is 2-manifold}, $F_i$ and $F_i''$ are 2-manifolds with boundaries homeomorphic to $S^1$.
By the generalized Margulis lemma \cite{FY}, $F_i$ has an almost nilpotent fundamental group. Hence $F_i$ is homeomorphic to $D^2$ or $\Mo$.

Therefore, we obtain the following assertion: 
\begin{lemma}\label{boundary is not T^2}
$\partial B_i$ is homeomorphic to $S^2$, $P^2$ or $K^2$.
\end{lemma}

We now determine the topology of $B_i$.

\begin{lemma} \label{ball near boundary}
$B_i$ is homeomorphic to $D^3$, $\Mo \times I$ or $K_1(P^2)$.

Moreover, if $\diam \Sigma_p > \pi / 2$ then $B_i$ 
is not homeomorphic to $K_1(P^2)$.
\end{lemma}
\begin{proof} 
We first consider the case that $\diam \Sigma_p > \pi /2$.
Then by Proposition \ref{regular collapsing}, $\Sigma_{p_i}$ is topologically a suspension over a one-dimensional Alexandrov space $\Lambda$ of curvature $\geq 1$. 
Since $\partial \Sigma_{p_i} = \emptyset$, $\Lambda$ is a circle. 
Hence $p_i$ is a topologically regular point.
Note that, in this situation, any $x \in B(p,r)$ has $\diam \Sigma_x > \pi / 2$. 
Therefore, $\mathrm{int}\, B_i$ is topologically a manifold, 
and $B_i$ is not homeomorphic to $K_1(P^2)$.

From now on we assume that $\diam \Sigma_p \leq \pi /2$.
If $B_i$ dose not satisfies Assumption \ref{rescaling assumption} 
then, there exists $\hat{p}_i$ such that $\lim|p_i \hat{p}_i| = 0$ and $B(\hat{p}_i, r) \approx K_1(\Sigma_{\hat p_i})$ which is homeomorphic to $D^3$ or $K_1(P^2)$,
where we may assume that $p_i = \hat{p}_i$.

Suppose that $B_i$ satisfies Assumption \ref{rescaling assumption}.
By Theorem \ref{rescaling argument}, there is a sequence $\de_i$ of positive numbers tending to zero and points $\hat{p}_i$ 
(where we may assume that $\hat{p}_i = p_i$) such that 
\begin{itemize}
\item any limit $( Y, y_0 )$ of $( \frac{1}{\de_i} B_i, p_i )$ as $i \to \infty$, is a three-dimensional open Alexandrov space of nonnegative curvature; 
\item denoting by $S$ a soul of $Y$, we obtain $\dim S \leq 1$.
\end{itemize}
Then, by Soul Theorem \ref{soul theorem}, $Y$ is homeomorphic to $\mathbb{R}^3$, $K(P^2)$ or 
$M_\pt$ 
if $\dim S = 0$, or an $\mathbb{R}^2$-bundle over $S^1$ if $\dim S =1$.
Therefore, $B_i$ is homeomorphic to $D^3$, $K_1(P^2)$ or $B(\pt)$ if $\dim S = 0$, or $S^1 \times D^2$ or $S^1 \tilde{\times} D^2 \approx \Mo \times I$ if $\dim S =1$.
By the boundary condition (Lemma \ref{boundary is not T^2}), $B_i$ is actually not homeomorphic to $S^1 \times D^2$.
It remains to show that 
\begin{equation}\label{ball is not B(pt)}
B_i \text{ is not homeomorphic to } B(\pt).
\end{equation}
We prove \eqref{ball is not B(pt)} by contradiction.
Suppose that there is a homeomorphism $f_i : B(\pt) \to B_i$.
We will use the notations in the proof of Lemma \ref{ball is B(pt)}.
Recall that $B(\pt)$ is obtained by the quotient space of $S^1 \times D^2$ by the involution $\hat \sigma$.
We consider the corresponding space $\hat{B}_i$ with an involution $\hat \sigma_i$ such that its quotient is $B_i$. 
By the argument of the proof of Lemma \ref{ball is B(pt)}, 
we obtain the following commutating diagram:
\[
\begin{CD}
\mathbb{R} \times D^2 @> \tilde{f}_i > > \tilde{B}_i \\
@V \pi V V                                                              @V V \pi_i V \\
S^1 \times D^2 @> \hat{f}_i > > \hat{B}_i \\
@V \hat{\pi} V V                                        @V V \hat{\pi}_i V \\
B( \pt )                @>> f_i >                       B_i 
\end{CD}
\]
Here, the horizontal arrows are homeomorphisms, 
$\pi$ and $\pi_i$ are the universal coverings, and 
$\hat{\pi}$ and $\hat{\pi}_i$ are the projections by involutions $\hat \sigma$ and $\hat \sigma_i$, respectively.
We may assume that $(\tilde{B}_i, \tilde{p}_i, \Gamma_i, \Lambda_i)$ converges to $(Z, z_0, \Gamma_\infty, \Lambda_\infty)$ with $Z = \mathbb R \times L$, $\Lambda_\infty = \mathbb R \times H$, $L$ is a flat cone over a circle and $H$ is a finite abelian group acting on $L$.
Note that all elements of $H$ are orientation preserving on $L$.
Recall that $\sigma_\infty$ is expressed as $\sigma_\infty(x, y) = (-x, \sigma_\infty'(y))$ 
and $\sigma_\infty$ is orientation preserving on $L$.
Therefore, $[\sigma_\infty']$ is orientation preserving on $L / H$.
We remark that $(L / H) / [\sigma_\infty'] = T_p X$. 
Then, $L / H$ has no boundary.
Indeed, to check this, we suppose that $L / H$ has non-empty boundary.
Then $L / H$ is the cone over an arc. 
Since $[\sigma_\infty']$ is non-trivial isometry on $L / H$, 
$[\sigma_\infty']$ is the reflection with respect to the center line.
Therefore, $[\sigma_\infty']$ does not preserve orientation.
This is a contradiction.

Thus, $L / H$ is the cone over a circle. 
It turns out that $\sigma_\infty'$ is a half rotation of $L$, and hence so is $[\sigma_\infty']$ for $L / H$. 
This implies $T_p X$ has no boundary, and we obtain a contradiction.
We conclude \eqref{ball is not B(pt)}, 
and complete the proof of Lemma \ref{ball near boundary}
\end{proof}

Next, We will divide $D_i$ into two pieces $D_i = H_i \cup K_i$ depending on the topology of $F_i$.
And we will determine the topology of $H_i$, $K_i$, and $D_i$.

\subsection{The case that $F_i$ is a disk} \label{F_i = D^2}
We consider the case that $F_i \approx D^2$.
Then, we divide $D_i$ into $H_i$ and $K_i$ as follows.
\begin{align*}
H_i &:= D_i - U(\gamma_i, \de),\\
K_i &:= D_i \cap B(\gamma_i, \de).
\end{align*}

\subsubsection{The topology of $K_i$}
We prove that
\begin{assertion} \label{K_i is D^3}
$K_i$ is homeomorphic to $D^3$.
\end{assertion}
$K_i$ is contained in a domain $L_i$ defined by 
\begin{equation} \label{definition of L_i}
L_i := A(p_i; r - \de', |p p'| -r/2) \cap B(\gamma_i, \de ).
\end{equation}
Since $(d_{p_i}, d_{\gamma_i})$ is $(c, \theta(\e))$-regular near $L_i \cap S(\gamma_i, \de)$, 
by Theorem \ref{Morse theory} and Lemma \ref{union lemma corollary}, $L_i$ is homeomorphic to $F_i \times [0,1] \approx D^3$.
On the other hand, we can take a closed domain $A_i \subset \mathrm{int}\,K_i$ such that $A_i \approx D^3$ and  
\begin{equation} \label{A_i is large}
K_i^0 := B(\gamma_i, \de / 2) - \left( U(p_i,2 r) \cup U(p'_i, 2 r)\right) \subset \mathrm{int}\, A_i.
\end{equation}
By Theorem \ref{Morse theory} and Lemma \ref{union lemma corollary}, $K_i \approx K_i^0$.
Remark that $F_i' := \partial B_i' \cap B(\gamma_i, \de)$ is homeomorphic to $D^2$. 
Indeed, if we assume that $F_i' \approx \Mo$, then $\partial K_i \approx P^2$. 
Then, by the embedding \eqref{A_i is large}, we have 
\[
P^2 \approx \partial K_i^0 \subset \mathrm{int}\, A_i \approx \mathbb{R}^3
\]
This is a contradiction.
Therefore, $F_i' \approx D^2$ and $\partial K_i^0 \approx \partial K_i \approx S^2$.
By Theorem \ref{Morse theory}, 
$\partial K_i^0$ is locally flatly embedded in $A_i \approx D^3$.
Therefore, by the generalized Schoenflies theorem, 
we conclude $K_i \approx K_i^0 \approx D^3$.

\subsubsection{The topology of $H_i$}
\begin{assertion}
$H_i$ is homeomorphic to $S^1 \times D^2$ and the circle fiber structure on $H_i$ induced by the standard one on $S^1 \times D^2$ is compatible to $\pi_i'$.
\end{assertion}
Let us define a domain $Q \subset X$ by 
\begin{equation}
Q := A(\gamma; \de - \de', 2 \de + \de') - (U(p, r- 2\de') \cup U(p', r -2 \de')). \label{definition of Q}
\end{equation} 
Note that $Q$ is homeomorphic to a two-disk without $(2, \e)$-singular points.
Then $Q_i := {\pi_i'}^{-1}(Q)$ is topologically a solid torus, and 
$H_i$ is contained in the interior of $Q_i$.

We will construct an isotopy $\varphi : Q_i \times [0, 1] \to Q_i$ 
satisfying
\begin{align}
&\varphi (\cdot, 0) = \mathrm{id}_{Q_i}, \label{isotopy_1} \\
&\varphi (Q_i, 1) = H_i, \label{isotopy_2}  \\
&\varphi : \partial Q_i \times [0,1] \to Q_i - \mathrm{int}\, H_i \text{ is bijective.} \label{isotopy_3}
\end{align}
If we obtain such a $\varphi$, then by \eqref{isotopy_2}, we conclude $H_i \approx Q_i \approx S^1 \times D^2$. 
And by \eqref{isotopy_3}, we can obtain the circle fiber structure of $H_i$ over $Q$ which is compatible to the generalized Seifert fibration ${\pi_i'}$.

Next we use the conventions as in \eqref{tilde map}.
\begin{lemma} \label{Q_i to H_i}
There is a smooth vector field $\tilde{X}$ on $N(Q_i - H_i)$ such that it is gradient-like 
\begin{itemize} 
\item for $\tilde {d}_{p_i}$ and 
$\widetilde {{d}_p \circ \pi_i'}$ on $N(B(p_i, r + \de') \cap Q_i - H_i)$;
\item for $\widetilde {d}_{p_i'}$ and 
$\widetilde {{d}_{p'} \circ \pi_i'}$ on $N(B(p_i', r + \de') \cap Q_i - H_i)$; 
\item for $\widetilde {d}_{\gamma_i}$ 
and $\widetilde {{d}_\gamma \circ \pi_i'}$ 
on $N(B(\gamma_i, \de + \de') \cap Q_i - H_i)$; and 
\item for $- \widetilde {d}_{\gamma_i}$ 
and $- \widetilde {{d}_\gamma \circ \pi_i'}$ 
on $N(Q_i - H_i - U(\gamma_i, 2 \de - \de'))$.
\end{itemize}
\end{lemma}
\begin{proof}
Let us take gradient-like smooth vector fields $\tilde{V}$, $\tilde{V}'$ and $\tilde{W}$ for 
$\widetilde d_{p_i}$, $\widetilde  d_{p_i'}$ 
and $\widetilde d_{\gamma_i}$ on $N(Q_i - H_i)$.
We prepare a decomposition of $Q_i - \mathrm{int}\, H_i$ as follows:
\begin{equation} \label{A_a}
Q_i - \mathrm{int}\, H_i = \bigcup_{\alpha = 1}^8 A_\alpha.
\end{equation}
See Figure \ref{fig_Q_i}.
Here, we define
\begin{align*}
A_1 &:= \left( Q_i - \mathrm{int}\, H_i \right) 
\cap \left( B(\gamma_i, \de) - U(\{p_i, p_i'\}, r + \de') \right), \\
A_2 &:= \left( Q_i - \mathrm{int}\, H_i \right) - \left( U(\gamma_i, 2 \de - \de')  \cup U(\{p_i, p_i'\}, r + \de') \right), \\
A_3^\ast &:= 
B(p_i, r + \de') \cap B(\gamma_i, \de + \de'), \\
A_4^\ast &:= 
B(p_i, r + \de') \cap A(\gamma_i; \de + \de', 2 \de - \de'), \\
A_5^\ast &:= 
B(p_i, r + \de') - U(\gamma_i, 2 \de - \de'). 
\end{align*}
Similarly, we put 
\begin{align*}
A_6^\ast &:= 
B(p_i', r + \de') \cap B(\gamma_i, \de + \de'), \\
A_7^\ast &:= 
B(p_i', r + \de') \cap A(\gamma_i; \de + \de', 2 \de - \de'), \\
A_8^\ast &:= 
B(p_i', r + \de') - U(\gamma_i, 2 \de - \de'). 
\end{align*}
And we define $A_3, A_4, \cdots, A_8$ by
\begin{align*}
A_\alpha &:= A_\alpha^\ast \cap Q_i - \mathrm{int}\, H_i \text{ for } \alpha = 3, 4, \dots, 8.
\end{align*}

\begin{figure}[h]
\includegraphics[width=200pt]{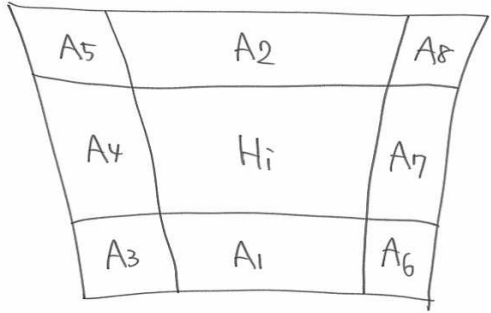}
\caption{: The decomposition of $Q_i$}
\label{fig_Q_i}
\end{figure}

We take smooth functions $h_\alpha$ ($\alpha = 1, \dots, 8$) on $N(Q_i - H_i)$ such that 
\begin{align*}
& 0 \leq h_\alpha \leq 1, \\
& h_\alpha \equiv 1 \text{ on } \tilde A_\alpha, \\
& \mathrm{supp}\, (h_\alpha) \subset B(\tilde A_\alpha, \de' / 100).
\end{align*}
We define a vector field $\tilde X$ as  
\[
\begin{aligned}
\tilde X &:= 
h_1 \tilde W - h_2 \tilde W \\
&\quad + h_3 ( \tilde V + \tilde W ) + h_4 \tilde V + h_5 (\tilde V - \tilde W) \\
&\quad + h_6 (\tilde V' + \tilde W ) + h_7 \tilde V' + h_8 (\tilde V' - \tilde W).
\end{aligned}
\]
Then, we can show that $\tilde X$ satisfies the conclusion of Lemma \ref{Q_i to H_i} as follows.
We will prove it only on $N(A_3)$.

We consider the integral flow $\tilde \Phi$ of $\tilde X$ 
and the pull-back $\Phi_t := f_i^{-1} \circ \tilde \Phi_t \circ f_i$.
It suffices to show that 
\begin{align}
\Phi &\pitchfork \dist_{p_i} \label{transversality eq1} \\
\Phi &\pitchfork \dist_{\gamma_i} \label{transversality eq2}\\
\Phi &\pitchfork \dist_p \circ \pi_i' \label{transversality eq3}\\
\Phi &\pitchfork \dist_\gamma \circ \pi_i' \label{transversality eq4}
\end{align}
on $N(A_3)$.
We can write 
\[
\tilde X = \alpha \tilde V + \beta \tilde W
\]
for smooth functions $\alpha, \beta \ge 0$ with $1 \le \alpha + \beta \le 3$, on $N(A_3)$.
By a direct calculus, we have 
\begin{align*}
|\tilde X| &\ge \sqrt 2 - \theta(\e), \\
\angle (\tilde X, \tilde V) < \gamma + \theta(\e), & \,\,
\angle (\tilde X, \tilde W) < \gamma + \theta(\e),
\end{align*}
on $N(A_3)$. 
Here, $\cos \gamma = 1 / \sqrt {10}$.


Let us set
\[
X := \left. \frac{d}{dt} \right|_{t = 0+} \Phi(x, t) \in T_x M_i.
\]
Then we have 
\[
X(x) = d f_i^{-1} (\tilde X (\tilde x)).
\]
And we set 
\[
V := d f_i^{-1} (\tilde V), \,\, W := d f_i^{-1} (\tilde W).
\]
Then we obtain 
\[
V \doteqdot \nabla d_{p_i}, \,\, W \doteqdot \nabla d_{\gamma_i}.
\]
Here, $A \doteqdot A'$ means $|A, A'| < \theta (\e)$.

Since $f_i$ is a $\theta(\e)$-almost isometry, we have
\[
|X| \doteqdot |\tilde X|, \,\,
|V, X| \doteqdot |\tilde V, \tilde X|, 
\,\, 
\angle (V, X) \doteqdot \angle (\tilde V, \tilde X). 
\]
Hence, we obtain
\begin{align*}
|X| &\ge \sqrt 2 - \theta(\e), \\
\angle (p_i', X) &\ge \angle (p_i', V) - \angle (V, X) > \pi - \gamma - \theta(\e). 
\end{align*}
Therefore, we have 
\[
(d_{p_i})' (X) = - |X| \cos \angle (p_i', X) \ge 1 / \sqrt 5 - \theta(\e).
\]
This implies \eqref{transversality eq1}.

For any fixed scale $(o)$, we set 
\[
d \pi_i' := (\exp_{\pi_i' (x)}^{(o)})^{-1} \circ (\pi_i')_x^{(o)} \circ \exp_x^{(o)}.
\] 
By Proposition \ref{almost linear}, we have
\begin{align*}
\angle (p', d \pi_i' (X)) < \gamma + \theta(\e), 
\,\,
|d \pi_i' (X)| > \sqrt 2 - \theta(\e). 
\end{align*}
Therefore, we obtain
\begin{align*}
(\dist_p \circ \pi_i')_x^{(o)} (\exp_x^{(o)}(X)) 
&= - |d \pi_i' (X) | \cos \angle_{\pi_i'(x)} (p', d \pi_i' (X) ) \\
&> 1 / \sqrt 5 - \theta(\e).
\end{align*}
Thus, we obtain \eqref{transversality eq3}.

In a similar way to above, we can prove \eqref{transversality eq2} and \eqref{transversality eq4}.
\end{proof}

By Lemma \ref{Q_i to H_i}, we obtain an isotopy $\varphi$ based on $\Phi$ satisfying from \eqref{isotopy_1} to \eqref{isotopy_3}.

Therefore, we conclude that if $F_i \approx D^2$ then $D_i \approx D^3$.

\subsection{The case that $F_i$ is a Mobius band} \label{F_i = Mo}
We consider the case that $F_i \approx \Mo$.
We prove that 
\begin{lemma} \label{L_i is regular}
$D_i$ is $(3, \theta(\e))$-strained.
\end{lemma}
\begin{proof}
We first define a domain $L_i$ similar to \eqref{definition of L_i}:
\begin{equation*}
L_i := A(p_i; r / 2, |p p'| -r/2) \cap B(\gamma_i, 3 \de ).
\end{equation*}
To prove Lemma \ref{L_i is regular}, it suffices to show that 
\begin{equation} \label{L_i eq1}
L_i \text{ is } (3, \theta(\e)) \text{-regular}.
\end{equation}

By Theorem \ref{Morse theory} and Lemma \ref{union lemma corollary}, we have $L_i \approx \Mo \times I$.
Let $\hat L_i$ be the orientable double cover which is homeomorphic to $(S^1 \times I) \times I$.
Since $\hat L_i$ is a covering space of $L_i$, $\hat L_i$ has the metric of Alexandrov space with $L_i \equiv \hat L_i / \langle \sigma \rangle$ for an isometric involution $\sigma$ on $\hat L_i$.

Since the projection $\hat L_i \to L_i$ is a local isometry, 
to prove \eqref{L_i eq1}, it suffices to show that
\begin{equation}\label{L_i eq2}
\hat L_i \text{ is } (3, \theta(\e)) \text{-regular}.
\end{equation}
$L_i$ converges to the following closed domain $L_\infty$ in $X$:
\[
L_\infty = A(p; r / 2, |pp'| - r /2) \cap B(\gamma, 3 \de).
\]
We may assume that $\hat L_i$ converges to some two-dimensional space $Y^2$.
Note that $L_\infty$ is $1$-strained, and hence $L_i$ and $\hat L_i$ is also $1$-strained. 
Therefore, 
\begin{equation} \label{Y is 1-strained}
\text{$Y$ is $1$-strained.}
\end{equation}
From the form of $L_\infty$, we have that $Y^2$ is a two-disk having no $\e$-singular points.
Indeed, if $Y$ has a boundary-point in the sense of Alexandrov space, 
then from an argument similar to the proof of Assertion \ref{K_i is D^3},
$\hat L_i$ contains a domain homeomorphic to $D^2 \times I$ or $\Mo \times I$. 
This is a contradiction, and hence $Y$ has no boundary.
By this and \eqref{Y is 1-strained}, $Y$ is $2$-strained. 
Therefore, $\hat L_i$ is $3$-strained, this is the assertion \eqref{L_i eq2}.
This implies \eqref{L_i eq1}, and completes the proof of Lemma \ref{L_i is regular}
\end{proof}

By Lemma \ref{L_i is regular} and Theorem \ref{flow theorem}, 
we have a Lipschitz flow $\Phi$ 
which is gradient-like for $\dist_{p_i}$ 
near $D_i$. 
We divide $D_i$ into $H_i$ and $K_i$ as follows.
\begin{align*}
K_i :=& \text{ the union of flow curves of $\Phi$} \\
                 & \text{ starting from $F_i$ in } B(\gamma_i, 2\de) - \mathrm{int}\, B_i'. \\
H_i :=&\,\, D_i - \mathrm{int}\, K_i.
\end{align*}
Note that the union of flow curves of $\Phi$ starting from $\partial F_i$ is contained in $A(\gamma_i; \de - \de'', \de + \de'')$ for some small $\de'' > 0$. 
By the construction, $K_i \approx \Mo \times I$.

We will prove that 
\begin{assertion} \label{H_i if F_i = Mo}
$H_i$ is homeomorphic to $S^1 \times D^2$ and 
the circle fiber structure on $H_i$ induced by one on $S^1 \times D^2$ is compatible to $\pi_i'$.
\end{assertion}
\begin{proof}
Let $Q_i$ be a closed neighborhood of $H_i$ obtained in a way similar to the construction of $Q_i$ in the subsection \ref{F_i = D^2}. 
We actually define 
\begin{align*}
Q &:= A(\gamma; \de - \de'', 2 \de + \de') - U(\{p, p'\}, r- 2\de'), \\
Q_i &:= {\pi_i'}^{-1}(Q). 
\end{align*}
We prepare a decomposition of $Q_i - \mathrm{int}\, H_i = \bigcup_{\alpha = 1}^8 A_\alpha$ in a way similar to \eqref{A_a} in Lemma \ref{Q_i to H_i}. 
Actually, we define $A_5$, $A_2$, $A_8$ as same as Lemma \ref{Q_i to H_i}, and other $A_\alpha$'s are defined by
\begin{align*}
A_1 &:= \left( Q_i - \mathrm{int}\, H_i \right) \cap \left( B(\gamma_i, \de + \de'') - U(\{p_i, p_i'\}, r + \de') \right), \\
A_3 &:= (Q_i - \mathrm{int}\, H_i) \cap B(p_i, r + \de') \cap B(\gamma_i, \de + \de''), \\
A_6 &:= (Q_i - \mathrm{int}\, H_i) \cap B(p_i', r + \de') \cap B(\gamma_i, \de + \de''), 
\\
A_4 &:= (Q_i - \mathrm{int}\, (H_i \cup A_3 \cup A_5)) \cap B(p_i, r + \de'), \\
A_7 &:= (Q_i - \mathrm{int}\, (H_i \cup A_6 \cup A_8)) \cap B(p_i', r + \de').
\end{align*}
Since $\nabla \dist_{p_i}$ and $\nabla \dist_{\gamma_i}$ are almost perpendicular to each other on $Q_i - \mathrm{int}\, H_i$, 
we can obtain a flow $\Phi$ which has nice transversality as Lemma \ref{Q_i to H_i}. 
And we can construct an isotopy from the identity to some homeomorphism which deforms $Q_i$ to $H_i$ inside $Q_i$.
Therefore, 
we obtain a circle fibration of $H_i$ over $Q$ which is compatible to the generalized Seifert fibration ${\pi_i'}$. 
This completes the proof of Assertion \ref{H_i if F_i = Mo}.
\end{proof}

Therefore, we conclude that if $F_i \approx \Mo$ then $D_i \approx \Mo \times I$.

\begin{proof}[Proof of Theorem \ref{2-dim boundary}] 
It remain to show that each component $M_{i, C}''$ of $M_i''$ has the structure of a generalized solid torus or generalized solid Klein bottles. 
It is clear from Sections \ref{F_i = D^2} and \ref{F_i = Mo}.
\end{proof}

\subsection{Proof of Corollary \ref{2-dim boundary corollary}}
To prove Corollary \ref{2-dim boundary corollary}, we show elementary lemmas. 
We define the mapping class group $\mathrm{MCG}(F)$ of a topological space $F$ 
to be the set of all isotopy classes of homeomorphisms of $F$.

\begin{lemma} \label{isotopy class}
Let $F$ be a topological space. 
For any element $\gamma$ of the mapping class group $\mathrm{MCG}(F)$, we fix a homeomorphism $\varphi_\gamma : F \to F$. 
such that $\varphi_\gamma \in \gamma$.
Let us set $B = F \times [0,1]$ and $\pi : B \to [0,1]$ a projection.
For any homeomorphisms $f_i : F \to \pi^{-1}(i)$, for $i = 0, 1$, there exist $\gamma \in \mathrm{MCG}(F)$ and a homeomorphism $h : F \times [0,1] \to B$ respecting $\pi$ such that, for every $x \in F$, $h(x,0) = f_0(x)$ and $h(x,1) = f_1 \circ \varphi_\gamma (x)$. 
\end{lemma}
\begin{proof}
Let us set $F_t = \pi^{-1}(t) = F \times \{t\}$.
Let us define the translation $\chi_t : F_0 \to F_t$ by $\chi_t(x,0) = (x,t)$, and set
a homeomorphism $\tilde f_t = \chi_t \circ f_0 : F \to F_t$.
Note that $\tilde f_0 = f_0$.
Let us take an element $\gamma \in \mathrm{MCG}(F)$ represented by a homeomorphism $f_1^{-1} \circ \tilde f_1$ of $F$.
Then, there is a homeomorphism $g_t : F \to F$, for $0 \le t \le 1$, such that 
\[
g_0 = \mathrm{id} \text{ and } \tilde f_1 \circ g_1 = f_1 \circ \varphi_\gamma.
\]
Therefore, setting $h_t = \tilde f_t \circ g_t : F \to F_t$, we obtain 
\[
h_0 = f_0 \text{ and } h_1 = f_1 \circ \varphi_\gamma.
\]
Hence, defining $h : F \times [0,1] \to B$ by $h(x,t) = h_t(x)$, $h$ satisfies the desired condition.
\end{proof}


\begin{lemma} \label{gen sol lemma}
Let $Y$ be a generalized solid torus or a generalized solid Klein bottle. 
Let $\pi : Y \to S^1$ be a projection as \eqref{projection of gen solid}.
Then, there is a continuous surjection 
\[
\eta : Y \to [0,1]
\]
such that $\eta^{-1}(1) = \partial Y$ and, setting 
\[
\Phi = (\pi, \eta) : Y \to S^1 \times [0,1], 
\]
$\Phi$ is an $S^1$-bundle over $S^1 \times (0, 1]$.
Further, for every $x \in S^1$, $\Phi^{-1}(x,0)$ is a one point set or a circle, and the homeomorphic type of the fiber $\Phi^{-1}(x,0)$ changes if and only if that of $\pi^{-1}(x)$ changes.
\end{lemma}
\begin{proof}
Let us take ordered points $t_1, t_2, \dots, t_{2N-1}, t_{2 N} \in S^1$ changing the fiber of $\pi$.
Then, for a small $\varepsilon > 0$, setting $I_k = [t_k - \varepsilon, t_k + \varepsilon] \subset S^1$, $\pi^{-1} (I_k)$ is homeomorphic to $K_1(P^2)$.

We regard $K_1(P^2) = \bigcup_{t \in [-1,1]} D(t)$ as Definition \ref{generalized solid}. 
Let us define a continuous surjection $\theta : K_1(P^2) \to [0,1]$ by 
\[
\theta (x,y,z) = \left\{
\begin{aligned}
& z^2 & \text{if } t > 0 \\
& x^2 + y^2 & \text{if } t \le 0
\end{aligned}
\right.
\]
This is well-defined.
($\theta$ is like the square of the distance function from the center of each surface $D(t)$. If $t \le 0$, then the center means a point $D(t) \cap \{x^2+y^2 = 0\}$ of disk $D(t)$, and if $t > 0$, then the center means a centric circle $D(t) \cap \{z = 0\}$ of a Mobius band $D(t)$).
Let us fix a homeomorphism $\varphi_k : K_1(P^2) \to \pi^{-1}(I_k)$ respecting $\pi$. 
We define a continuous surjection
\[
\eta_k = \theta \circ \varphi_k^{-1} : \pi^{-1}(I_k) \to [0,1].
\]
Thus, a continuous surjection from the disjoint union of $\pi^{-1}(I_k)$'s to $[0,1]$, 
is defined, and satisfies the desired property.

It remain to show that the domain of $\eta_k$'s can extend to the whole $Y$, satisfying the desired property.
Let $J_k := [t_k + \varepsilon, t_{k+1} - \varepsilon] \subset S^1$ be the interval between $I_k$ and $I_{k+1}$.
Let us set $F_k = \pi^{-1}(t_k + \varepsilon)$ which is homeomorphic to $D^2$ or $\Mo$.
Let $G_k = \pi^{-1}(t_{k+1} - \varepsilon)$ which is homeomorphic to $F_k$.

Suppose that $F_k \approx D^2$. 
We recall that $D(-1) \subset \partial K_1(P^2)$ is defined as
\[
\{(x,y,z) \in \mathbb R^3 \,|\, x^2 + y^2 - z^2 = -1, x^2 + y^2 \le 1 \} / (x,y,z) \sim -(x,y,z).
\]
We identify this as $D^2 = \{(x,y) \,|\, x^2 + y^2 \le 1\}$ by a map 
\[
D(-1) \ni [x,y,z] \mapsto (x,y) \in D^2.
\]
Then, via $\varphi_k$, the map $\eta_k : F_k \to [0,1]$ can be identified as the map 
\[
\theta' : D^2 \to [0,1];\,\, (x,y) \mapsto x^2 + y^2.
\]
Namely, $\eta_k = \theta' \circ \varphi_k^{-1}$.
Similarly, $\eta_{k+1} = \theta' \circ \varphi_{k+1}^{-1}$.
Here, $\varphi_k$ and $\varphi_{k+1}$ are restricted on $D^2 \subset \partial K_1(P^2)$.
Let $r : D^2 \to D^2 ; (x,y) \mapsto (x, -y)$ be the reflection with respect to the $x$-axis.
We note that $\theta' \circ r = \theta'$ and $r$ represents a unique non-trivial element of the mapping class group $\mathrm{MCG}(D^2)\, (\cong \mathbb Z_2)$ of $D^2$.
By using Lemma \ref{isotopy class}, we obtain a homeomorphism 
\[
\varphi_k' : D^2 \times J_k \to \pi^{-1}(J_k),
\]
respecting projections $\pi$ and $D^2 \times J_k \to J_k$ such that $\varphi_k'= \varphi_k$ on $F_k$ and either
\[
\begin{aligned}
&\varphi_k'= \varphi_{k+1} \hspace{17pt} \text{ on } G_k, \text{ or} \\
&\varphi_k'= \varphi_{k+1} \circ r \text{ on } G_k.
\end{aligned}
\]
Hence, 
\[
\eta_k' = \theta' \circ (\varphi_k')^{-1} : \pi^{-1}(J_k) \to [0,1]
\] 
satisfies
\begin{align*}
\eta_k' = \eta_k \text{ on } F_k \text{ and } \eta_k' = \eta_{k+1} \text{ on } G_k.
\end{align*}
Therefore, if the fiber of $\pi$ on $J_k$ is a disk, then $\eta_k$ and $\eta_{k+1}$ extend to the map $\eta_k'$ on $\pi^{-1}(J_k)$, satisfying the desired property. 

Next, we assume that $F_k \approx \Mo$.
We recall that $D(1)$ is 
\[
\{(x,y,z) \,|\, x^2 + y^2 - x^2 = 1, |z| \le 1\} / (x,y,z) \sim -(x,y,z).
\]
Let us identify $D(1) \subset \partial K_1(P^2)$ as $\Mo$ defined by
\[
\Mo = S^1 \times [-1,1] / (x,s) \sim (-x,-s)
\]
via a map
\[
D(1) \ni [x,y,z] \mapsto \left[ \frac{(x,y)}{\sqrt{x^2 + y^2}}, z \right] \in \Mo.
\]
Then, $\eta_k$ is identified as a projection
\[
\theta'' : \Mo \ni [x,s] \mapsto s^2 \in [0,1],
\]
via $\varphi_k$. Namely, $\eta_k = \theta'' \circ \varphi_k^{-1}$ on $F_k$.
And we can see $\eta_{k+1} = \theta'' \circ \varphi_{k+1}^{-1}$ on $G_k$.
Let us fix a homeomorphism $r : \Mo \to \Mo$ defined by 
$r[x,s] = [\bar x,s]$, where $\bar x$ is the complex conjugate of $x$ in $S^1 \subset \mathbb C$.
Then, $r$ reverses the orientation of $\partial \Mo$.
Hence, $r$ represents a unique non-trivial element of the mapping class group $\mathrm{MCG}(\Mo) \cong \mathbb Z_2$. 
And, we note that $\theta'' \circ r = \theta''$.
By Lemma \ref{isotopy class}, there exists a homeomorphism 
\[
\varphi_k'' : \Mo \times J_k \to \pi^{-1}(J_k),
\]
respecting projections $\pi$ and $\Mo \times J_k \to J_k$, such that $\varphi_k'' = \varphi_k$ on $F_k$ and either 
\begin{align*}
&\varphi_k'' = \varphi_{k+1} \hspace{16.5pt} \text{ on } G_k, \text{ or} \\
&\varphi_k'' = \varphi_{k+1} \circ r \text{ on } G_k.
\end{align*}
Since $\theta'' = \theta'' \circ r$, we obtain a continuous surjection
\[
\eta_k'' = \theta'' \circ (\varphi_k'')^{-1} : \pi^{-1}(J_k) \to [0,1]
\]
satisfying 
\[
\eta_k'' = \eta_k \text{ on } F_k \text{ and }
\eta_k'' = \eta_{k+1} \text{ on } G_k.
\]


By summarizing above, we obtain a continuous surjection 
\[
\eta : Y \to [0,1]
\]
satisfying the desired condition.
\end{proof}

\begin{proof}[Proof of Corollary \ref{2-dim boundary corollary}]
We may assume that $X$ has only one boundary component $\partial X$.
By Theorem \ref{2-dim boundary}, there are decompositions
\[
M_i = M_i' \cup M_i'' \text{ and } X = X' \cup X''
\]
satisfying the following.
\begin{itemize}
\item[(1)] $X''$ is a collar neighborhood of $\partial X$. We fix a homeomorphism $\varphi : \partial X \times [0,1] \to X''$ such that $\varphi (\partial X \times \{0\}) = \partial X$ and $\varphi (\partial X \times \{1\}) = \partial X'$;
\item[(2)] $M_i'$ is a generalized Seifert fiber space over $X' \approx X$. 
We fix a fibration $f_i' : M_i' \to X'$ of it;
\item[(3)] $M_i''$ is a generalized solid torus or a generalized solid Klein bottle. We fix a projection $\pi_i : M_i'' \to \partial X \approx S^1$ as \eqref{projection of gen solid} in Definition \ref{generalized solid}.
\item[(4)] The maps $f_i'$, $\pi_i$ and $\varphi$ are compatible in the following sense. 
For any $x \in \partial X$, 
\begin{align*} 
\pi_i^{-1}(x) \cap \partial M_i'' = (f_i')^{-1} (\varphi (x,1)) 
\end{align*}
holds.
\end{itemize}
By Lemma \ref{gen sol lemma}, we obtain a continuous surjection 
\[
\eta_i : M_i'' \to [0,1]
\]
such that 
\begin{itemize}
\item[(5)] $\eta_i^{-1}(1) = \partial M_i''$\,; 
\item[(6)] Setting $g_i = (\pi_i, \eta_i) : M_i'' \to \partial X \times [0,1]$, the restriction of $g_i$ on $g_i^{-1}\left(\partial X \times (0, 1]\right)$ is an $S^1$-bundle;
\item[(7)] For every $x \in \partial X$, $g_i^{-1}(x,0)$ is one point set or a circle. 
And, the fiber of $g_i$ changes at $x \in \partial X$ if and only if the fiber of $\pi_i$ changes at $x$.
\end{itemize}
Then, the map 
\[
f_i'' = \varphi \circ g_i : M_i'' \to X'', 
\]
satisfies 
\[
f_i' = f_i'' \text{ on } M_i' \cap M_i''.
\]
Therefore, the gluing $f_i : M_i \to X$ of maps $f_i'$ and $f_i''$ defined by 
\[
f_i = \left\{
\begin{aligned}
& f_i' \text{ on } M_i' \\
& f_i'' \text{ on } M_i''
\end{aligned}
\right.
\]
is well-defined. 
The map $f_i$ satisfies the topological condition desired in Corollary \ref{2-dim boundary corollary}. 

From the proof of Theorem \ref{2-dim boundary} and the construction of $X''$, for any $\varepsilon > 0$ and large $i$, we can take $\pi_i : M_i'' \to \partial X$ as an $\varepsilon$-approximation and $\varphi : \partial X \times [0,1] \to X''$ satisfying
\[
\big| \left|\varphi(x,t), \varphi(x',t')\right| - |x,x'| \big| < \varepsilon
\]
for any $x, x' \in \partial X$ and $t, t' \in [0,1]$.
Then, one can show that $f_i$ is an approximation.
\end{proof}




\section{The case that $X$ is a circle}
\label{proof of 1-dim circle}
Let $\{M_i^3\}$ be a sequence of closed three-dimensional 
Alexandrov spaces with curvature $\geq -1$ and uniformly bounded diameter.
Suppose that $M_i$ converges to a circle $X$.
We will prove Theorem \ref{1-dim circle}.

\begin{proof}[Proof of Theorem \ref{1-dim circle}]
We first show 
\begin{lemma}\label{M_i is manifold}
For large $i$, $\Sigma_{x} \approx S^2$ for all $x \in M_i$.
In particular, $M_i$ is a topological manifold. 
\end{lemma}
\begin{proof}
Indeed, by Proposition \ref{regular collapsing}, we may assume that $\diam \Sigma_{x_i}$ is almost $\pi$ for each $x_i \in M_i$. 
It follows from Theorem \ref{diameter suspension theorem} and $\partial M_i = \emptyset$ that $\Sigma_{x_i}$ is homeomorphic to the suspension over a circle, which is 2-sphere. 
Therefore, by Theorem \ref{stability theorem}, $M_i$ is a topological manifold.
\end{proof}

By taking a rescaling, we may assume that $M_i$ converges to the unit circle $X = S^1 = \{e^{i \theta} \in \mathbb C \,|\, \theta \in [0, 2\pi] \}$. 
We take points $p^+ := 1$ and $p^- := -1 \in S^1$, and prepare points $p_i^+$ and $p_i^- \in M_i$ converging to $p^+$ and $p^-$, respectively.
Let us set $q^+ := \sqrt {-1}$ and $q^- := - \sqrt{ -1} \in S^1$, 
and take $q_i^+, q_i^- \in M_i$ such that $q_i^\pm \to q^\pm$.

Let us take $\de_i$ the diameter of a part of $\partial B(p_i, \pi /2)$ 
which is GH-close to $q^+ \in S^1$.
We consider metric balls
\[
B_i^+ := B(p_i^+, \ell_i - \de_i ) \text{ and } 
B_i^- := B(p_i^-, \ell_i - \de_i ).
\]
Here, $\ell_i = |p_i^+, p_i^-| / 2$.
By the construction, $B_i^+ \cap B_i^- = \emptyset$.
We prove the next
\begin{lemma}\label{B_i is product}
$B_i^\pm$ is homeomorphic to $F_i^\pm \times [0, 1]$. 
Here, $F_i^\pm$ is homeomorphic to $S^2$, $P^2$, $T^2$ or $K^2$.
\end{lemma}
\begin{proof}
We will prove this assertion only for $B_i^+$.
Let us set $B_i := B_i^+$ and $p_i := p_i^+$.

By Lemma \ref{M_i is manifold}, $M_i$ is a manifold. 
We will implicitly use this fact throughout the following argument.

Remark that 
\begin{equation} \label{disconnected boundary}
\partial B_i \text{ is disconnected.}
\end{equation}

If $B_i$ does not satisfy Assumption \ref{rescaling assumption}, then there exists a sequence $\hat{p}_i \in M_i$, where we may assume that $\hat p_i = p_i$, and $\partial B_i \approx \Sigma_{p_i} \approx S^2$. 
Hence $\partial B_i$ is connected. 
This is a contradiction.

Therefore, $B_i$ must satisfy Assumption \ref{rescaling assumption}. 
Then, by Theorem \ref{rescaling argument}, there exists $\e_i \to 0$ and points $\hat p_i \in M_i$, where we may assume that $\hat p_i = p_i$, such that a limit $(Y, y_0) := \lim_{i \to \infty} (\frac{1}{\e_i}B_i, p_i)$ exists and has dimension $\ge 2$. 
We remark that $Y$ has a line, because $\wangle q_i^+ p_i q_i^- \to \pi$.
It follows from Theorem \ref{splitting theorem}, 
$Y$ is isometric to $S \times \mathbb R$ for some nonnegatively cured Alexandrov space $S$ of dimension at least one.

If $\dim S = 2$ then by Theorem \ref{stability theorem}, 
$S$ has no boundary, and the topology of $B_i$ can be determined. 
By the remark \eqref{disconnected boundary}, $S$ is compact, and hence, either $S$ is homeomorphic to $S^2$ or $P^2$, or is isometric to a flat torus or a flat Klein bottle.
Again, by using Theorem \ref{stability theorem}, we conclude that 
$B_i \approx S \times I$.

If $\dim S = 1$ then by Theorems \ref{2-dim interior} and \ref{2-dim boundary}, the topology of $B_i$ can be determined. 
It follows from the remark \ref{disconnected boundary}, $S$ is compact.
Hence $S$ is isometric to a circle or an interval.
If $S$ is a circle, then $Y$ has no singular point.
Then we can use Theorem \ref{fibration theorem}, and therefore we conclude that $B_i$ is homeomorphic to $T^2 \times I$ or $K^2 \times I$.
If $S$ is an interval, then by Theorem \ref{2-dim boundary}, 
$B_i$ is homeomorphic to $S^2 \times I$, $P^2 \times I$ or $K^2 \times I$.

This completes the proof of Lemma \ref{B_i is product}.
\end{proof}

Recall that $q_i^\pm$ are points in $M_i$ converging to $q^\pm = \pm \sqrt {-1} \in S^1$.
Let us consider 
\[
D_i^\pm := B(q_i^\pm, \pi / 2) - \mathrm{int}\, (B_i^+ \cup B_i^-).
\]
Let us set 
\[
S_i^{\pm} := B_i^\pm \cap D_i^+.
\]
By Lemma \ref{B_i is product}, $S_i^\pm \approx F_i^\pm$.

\begin{lemma}\label{D_i is product}
There is a homeomorphism $\phi_i : F_i^+ \times [0,1] \to D_i^+$ such that $\phi_i (F_i^+ \times \{0\}) = S_i^+$ and 
$\phi_i (F_i^+ \times \{1\}) = S_i^-$.
\end{lemma}
\begin{proof}
Let $W_i$ be the component of $S(p_i, \ell_i)$ converging to $q = \sqrt{-1} \in S^1$. 
Recall that $\de_i = \diam W_i$.
Then $\de_i \to 0$.

Let us take $q_i \in W_i$ and consider any limit $Y$ of a rescaling sequence: 
\begin{equation} \label{rescale eq1}
(\frac{1}{\de_i} M_i, q_i) \to (Y, q_\infty).
\end{equation}
Let $\gamma_\infty^\pm$ be rays starting at $q_\infty$ which are limits of geodesics $q_i p_i^\pm$.
Since $\wangle p_i^+ q_i z_i^- \to \pi$, $\gamma_\infty := \gamma_\infty^+ \cup \gamma_\infty^-$ is a line in $Y$.

Let $W_\infty$ be the limit of $W_i$ under the convergence \eqref{rescale eq1}.
By the choice of $\de_i$, $\diam W_\infty = 1$.
We will prove that 
\begin{assertion} \label{W_infty is splitting factor}
$Y$ is isometric to $W_\infty \times \mathbb R$. 
In particular, $\dim Y \ge 2$.
\end{assertion}
\begin{proof}[Proof of Assertion \ref{W_infty is splitting factor}]
Let us consider functions 
\begin{align*}
f_i^\pm (\cdot) &:= \tilde d_i (p_i^\pm, \cdot ) - \tilde d_i (p_i^\pm, q_i) \\
b^\pm (\cdot) &:= \lim_{t \to \infty} d(\gamma_\infty^\pm (t), \cdot ) - t.
\end{align*}
Here, $\tilde d_i$ is the original metric of $M_i$ multiplied by $1 / \de_i$.
The functions $b^\pm$ are the Busemann functions of the rays $\gamma_\infty^\pm$.
Then, we can show that $f_i^\pm$ converges to $b^\pm$.
Therefore, we obtain $W_\infty = (b^+)^{-1} (0)$.
This completes the proof of Assertion \ref{W_infty is splitting factor}.
\end{proof}

By Assertion \ref{W_infty is splitting factor}, $\dim W_\infty = 1$ or $2$.
If $\dim W_\infty = 2$ then by Theorem \ref{stability theorem}, we have a homeomorphism 
\[
\phi_i : D_i^+ \approx W_\infty \times [-1, 1]
\]
with respect to functions $f_i^\pm$ and $b^\pm$.
Namely, 
\[
\phi_i ( (f_i^\pm)^{-1}(t) ) = (b^\pm)^{-1} (t) 
\]
whenever $t$ is near $\{-1, 1\}$.
In particular, 
\[
S_i^+ = (f_i^+)^{-1} (1) \approx (b^\pm)^{-1} (0) = W_\infty \approx (f_i^-)^{-1} (1) = S_i^-.
\]
And, in this case, $W_\infty \approx S^2$, $P^2$, $T^2$ or $K^2$.

If $\dim W_\infty = 1$, then $W_\infty$ is a circle or an interval.
If $W_\infty$ is a circle, then by Theorem \ref{fibration theorem} and some flow argument, there is a circle fiber bundle 
\[
\pi_i : D_i^+ \to W_\infty \times [-1, 1]
\]
such that $\pi_i^{-1} (W_\infty \times \{\pm 1\}) = S_i^\pm$.
In this case, $S_i^\pm \approx T^2$ or $K^2$.

If $W_\infty$ is an interval, then by using Theorem \ref{2-dim boundary} and some flow argument, we have a homeomorphism 
\[
\phi_i : D_i^+ \to S_i^+ \times [-1, 1]
\]
such that $\phi_i (S_i^\pm) = S_i^+ \times \{\pm 1\}$.
In this case, $S_i^\pm \approx S^2$, $P^2$ or $K^2$.

This completes the proof of Lemma \ref{D_i is product}.
\end{proof}

Let $F_i$ be a topological space homeomorphic to $F_i^\pm \approx S_i^\pm$.
By Lemmas \ref{B_i is product} and \ref{D_i is product}, we obtain homeomorphisms
\begin{align*}
\varphi_i^\pm : F_i \times [0, 1] &\to B_i^\pm, \\
\psi_i^\pm : F_i \times [0, 1] &\to D_i^\pm
\end{align*}
such that they send the boundaries to the boundaries.
Therefore, $M_i = B_i^+ \cup B_i^- \cup D_i^+ \cup D_i^-$ is $F_i$-bundle over $S^1$.
\end{proof}








\section{The case that $X$ is an interval}
\label{proof of 1-dim interval}
Let $\{M_i\}$ be a sequence of three-dimensional closed Alexandrov spaces of curvature $\ge -1$ with $\diam M_i \le D$.
Suppose that $M_i$ converges to an interval $I$.
Let $\partial I = \{p, p'\}$. 
And let $p_i, p_i' \in M_i$ converge to $p, p'$, respectively.
We divide $M_i$ into $M_i = B_i \cup D_i \cup B_i'$, where $B_i = B(p_i, r)$, $B_i' = B(p_i', r)$ for small $r > 0$, and $D_i := M_i - \mathrm{int}\,(B_i \cup B_i')$.

\begin{proof}[Proof of Theorem \ref{1-dim interval}]
In a way similar to the proof of Lemma \ref{D_i is product},
we can prove that there exists a homeomorphism 
$\phi_i : F_i \times I \to D_i$ such that 
$\phi_i (F_i \times 0) = \partial B_i$ and 
$\phi_i (F_i \times 1) = \partial B_i'$, where
$F_i$ is homeomorphic to one of $S^2$, $P^2$, $T^2$ and $K^2$.

Next, we will find the topologies of $B_i$ (and $B_i'$).
If $B_i$ does not satisfy Assumption \ref{rescaling assumption}, then $B_i$ is homeomorphic to $D^3$ or $K_1(P^2)$.
Hence, we may assume that there exist sequences $\de_i \to 0$ and $\hat p_i$ such that a limit $(Y, y_0) = \lim_{i \to \infty} \frac{1}{\de_i} (B_i, \hat p_i)$ exists, where we may assume that $\hat p_i = p_i$, 
and $Y$ is a noncompact nonnegatively curved Alexandrov space of $\dim Y \geq 2$.

If $\dim Y =3$ with a soul $S \subset Y$, then Theorem \ref{soul theorem} implies $B_i$ is homeomorphic to one of 
\begin{itemize}
\item
$D^3$, $K_1(P^2)$ and $B(\pt)$ if $\dim S = 0$, 
\item 
$S^1 \times D^2$ and $S^1 \tilde{\times} D^2$ if $\dim S =1$, and
\item
$B(N(S))$ and $B(S_2)$ and $B(S_4)$ if $\dim S = 2$.
\end{itemize}
Here, $N(S)$ is a nontrivial line bundle over a closed surface $S$ of
nonnegative curvature and $B(N(S))$ is a metric ball around $S$ in
$N(S)$, and $B(S_i)$ is a metric ball around $S_i$ in $L_i = L(S_i)$
for $i = 2, 4$ (see \ref{subsec:soul}).
$B(N(S))$ is homeomorphic to one of nontrivial twisted $I$-bundles
over a closed surface $S$ with connected boundary. 
We determine the topology of $B(N(S))$ as follows:
If $S\approx S^2$, $N(S)$ is isometric to $S\times \mathbb R$, which
is a contradiction.
If $S\approx P^2$, we have the line bundle $N(\hat S)$ induced by
the double covering $\pi:\hat S \to S$. 
Since $N(\hat S)=\hat S \times \mathbb R$, we find
$N(S)=\hat S\times\mathbb R/(x, t)\sim (\sigma(x), -t)$, where
$\sigma$ is the involution on $\hat S$ with $\hat S/\sigma = S$.
Thus $B(N(S))$ is a twisted $I$-bundle over $P^2$; which is
homeomorphic to $P^3 - \mathrm{int}\, D^3$.
If $S$ is homeomorphic to either $T^2$ or $K^2$, then $N(S)$ is a
complete flat three-manifold.  
By \cite[Theorem 3.5.1]{Wo} we obtain that $B(N(S))$ is a twisted $I$-bundle over $T^2$; which is
homeomorphic to $\Mo \times S^1$, 
an orientable $I$-bundle $K^2 \tilde{\times} I$ over $K^2$, 
and a nontrivial non-orientable $I$-bundle $K^2 \hat{\times} I$ over $K^2$.

\vspace{1em}
If $\dim Y = 2$ and $\partial Y = \emptyset$, then either $Y$ is homeomorphic to $\mathbb{R}^2$ or isometric to a flat cylinder or a flat Mobius strip. 

Suppose that $Y \approx \mathbb{R}^2$.
Let us denote by $m$ the number of essential singular points in $Y$. Then $m \leq 2$.
When $m \leq 1$, Theorem \ref{2-dim interior} together with Lemma \ref{ball is B(pt)} implies that $B_i \approx S^1 \times D^2$ or $B(\pt)$. 
If $m = 2$, then $Y$ is isometric to the envelope $\mathrm{dbl}\,(\mathbb{R}_+ \times [0,\ell])$ for some $\ell > 0$. 
Let $B$ be a closed ball around $\{0\} \times [0, \ell]$ in $Y$.
By Theorem \ref{2-dim interior}, $B_i$ is a generalized Seifert fiber space over $B$ and its boundary $\partial B_i$ is 
homeomorphic to $T^2$ or $K^2$.
We may assume that $B_i$ has actually two singular orbits over two singular points $(0,0)$ and $(0,\ell)$ in $Y$.
Here, a singular orbit 
is either a $(2, 1)$-type fiber corresponding to the core of $U_{2, 1}$ or the interval fiber of $M_{\pt}$ in this case.
The topology of $B_i$ is determined as follows: 
When two singular orbits are both $(2,1)$-type, 
$\mathrm{int}\, B_i$ is homeomorphic to $U_{2,1}' \cup_\partial U_{2,1}'$. 
Since $U_{2,1}'$ is an $\mathbb R$-bundle over $\Mo$, $\mathrm{int}\, B_i$ is an $\mathbb R$-bundle over $K^2$.
By the boundary condition, $B_i$ is homeomorphic to $K^2 \tilde \times I$ if $\partial B_i \approx T^2$ 
or $K^2 \hat \times I$ if $\partial B_i \approx K^2$.
When singular fibers of $B_i$ are $(2,1)$-type and an interval, 
$\mathrm{int}\, B_i$ is homeomorphic to $U_{2,1}' \cup_\partial M_\pt'$. 
Then $B_i$ is homeomorphic to one of $B(S_2) \subset L_{2,1}$ with $S_2 \approx P^2$.
When $B_i$ has two singular interval fibers, $\mathrm{int} B_i$ is homeomorphic to $M_\pt' \cup_\partial M_\pt'$ which is $L_4$.
Then $B_i$ is homeomorphic to $B(S_4)$.

If $Y$ is a flat cylinder, then $\partial B_i$ is not connected, 
and hence this case can not happen.

If $Y$ is isometric to a flat Mobius strip, then $B_i$ is an $S^1$-bundle over $\Mo$.
Therefore, we have $B_i \approx \Mo \times S^1$ or $K^2 \tilde{\times} I$.

\vspace{1em}
If $\dim Y = 2$ and $\partial Y \neq \emptyset$, 
then $Y$ is either isometric to a flat half cylinder $S^1(\ell) \times [0, \infty)$ or $[0, \ell] \times \mathbb R$, or homeomorphic to a upper half plane $\mathbb{R}^2_+ = \mathbb{R} \times [0, \infty)$.

If $Y$ is a flat half cylinder, then $\partial Y$ has no essential singular point.
Therefore, $B_i$ is a fiber bundle over $S^1$ with the fiber homeomorphic to $D^2$ or $\Mo$. 
In other words, this is a generalized solid torus of type $0$ or a generalized solid Klein bottle of type $0$.

If $Y \equiv [0, \ell] \times \mathbb R$, then $\partial B_i$ is not connected, and hence this case can not happen.

Suppose that $Y$ is homeomorphic to $\mathbb{R}^2_+$.
Let us set $m := \sharp \ess\,(\mathrm{int}\, Y)$ and $n := \sharp \ess\, (\partial Y)$.
Then $m \leq 1$ and $n \leq 2$.

If $m = 0$ and $n \le 1$, then by Lemma \ref{ball near boundary}, $B_i$ is homeomorphic to one of $D^3$, $\Mo \times I$ and $K_1(P^2)$.

If $m = 0$ and $n = 2$, then $Y$ is isometric to $\mathbb R_+ \times [0, \ell]$ for some $\ell > 0$.
Let $B := [0, c] \times [0, \ell]$ for some $c >0$.
By Corollary \ref{2-dim boundary corollary}, there is a continuous surjective map 
\[
\pi : B_i \to B.
\]
We may assume that $B_i$ has two topologically singular points converging to the corners $(0, 0)$ and $(0, \ell)$ of $Y$.
We divide $B$ into two domains 
\[
A_j = [0,c] \times \{y \in [0,\ell] \,|\, (-1)^j (y - \ell/2) \ge 0\} \subset B
\]
for $j = 1, 2$.
Since $B_i$ has two topologically singular points, $\pi^{-1} (A_j) \approx K_1(P^2)$.
Then, $B_i$ is homeomorphic to $K_1(P^2) \cup_{D^2} K_1(P^2)$ if $\pi^{-1}(A_1 \cap A_2) \approx D^2$ or $K_1(P^2) \cup_{\Mo} K_1(P^2)$ if $\pi^{-1}(A_1 \cap A_2) \approx \Mo$.
By Lemma \ref{K_1 cup_D^2 K_1} and Remark \ref{K_1 cup K_1}, 
$B_i$ is homeomorphic to $B(\pt)$ or $B(S_2) \subset L_{2,2}$ with $S_2 \approx S^2$. 



If $m = 1$, then $n = 0$ and $Y$ is isometric to a cut envelope $\mathbb R \times [0, h] / (x, y) \sim (-x, y)$ for some $h > 0$.
Let $B := Y \cap \{(x, y) \,|\, x \le r \}$ which is homeomorphic to $D^2$. 
By Theorem \ref{2-dim boundary}, there is a generalized Seifert fibration $\pi_i : W_i \to B$ such that $B_i$ is homeomorphic to a gluing of $W_i$ and $F_i \times [-r ,r]$ via a homeomorphism 
\[
\partial F_i \times [-r , r]  \supset \partial F_i \times \{x\} \mapsto \pi_i^{-1} (x) \subset \pi_i^{-1} (\{(x, h) \in B \mid x \in [-r, r] \})
\]
for all $x \in [-r, r]$.
Here, $F_i$ is $D^2$ or $\Mo$.
We may assume that $W_i$ contains a singular orbit over the singular point $(0,0) \in \mathrm{int} B$.
If the singular orbit is a circle, then $W_i$ is isomorphic to a Seifert solid torus $V_{2,1}$ of $(2,1)$-type. 
Remark that $W_i$ can be regarded as $I$-bundle over $\Mo$, which corresponds to the preimage of the Seifert fibration over $\{0\} \times [0, h] \subset B$. 
Then, $B_i$ is isomorphic 
to an $I$-bundle over $\Mo \cup_\partial F_i$. 
Therefore, it is $P^2 \tilde \times I \approx P^3 - \mathrm{int} D^3$ if $F_i \approx D^2$, or
$K^2 \hat \times I$ if $F_i \approx \Mo$.
If the singular orbit is an interval, then Theorem \ref{2-dim interior} shows that $W_i$ is isomorphic to $M_\pt'$.
Recall that $B_i$ is homeomorphic to the union $W_i \cup F_i \times I$. 
Therefore, $B_i$ is homeomorphic to $B(S_2) \subset L_{2,2}$ with $S_2 \approx S^2$ if $F_i \approx D^2$, 
or $B(S_2) \subset L_{2,3}$ with $S_2 \approx P^2$ if $F_i \approx \Mo$.


This completes the proof of Theorem \ref{1-dim interval}.
\end{proof}

\section{The case that $X$ is a single-point set}
\label{proof of 0-dim}

\begin{lemma}\label{closed 3-dim}
If $M$ is a three-dimensional nonnegatively curved closed Alexandrov space, 
then a finite covering of $M$ is $T^3$, $S^1 \times S^2$ or simply-connected.
\end{lemma}
\begin{proof}
We may assume that $|\pi_1(M)| = \infty$. 
Then a universal covering $\tilde{M}$ of $M$ has a line.
Thus, $\tilde{M}$ is isometric to the product $\mathbb{R}^k \times X_0$, 
where $1 \leq k \leq 3$ and $X_0$ is 
a $(3-k)$-dimensional nonnegatively curved closed Alexandrov space.
\begin{itemize}
\item
If $k = 3$ then $\tilde{M}$ is the Euclidean space. 
Then a finite covering of $M$ is $T^3$.
\item
If $k = 2$ then $X_0$ is a circle. 
Then $\tilde{M}$ is not simply-connected, this is a contradiction.
\item
If $k = 1$ then $X_0$ is homeomorphic to $S^2$. 
Then a finite covering of $M$ is homeomorphic to $S^1 \times S^2$.
\end{itemize}

\end{proof}

\begin{proof}[Proof of Corollary \ref{0-dim}]
Let $\{ M_i \}$ be a sequence of three-dimensional closed Alexandrov spaces of curvature $\geq -1$ with $\diam M_i \le D$, 
which converges to a point $\{ \ast\}$.
Let $\de_i := \diam M_i$.
Then the rescaled space $\frac{1}{\de_i} M_i$ is an Alexandrov space 
with curvature $\geq - \de_i^2$ having diameter one.
Then, the limit $Y$ of the rescaled sequence $\frac{1}{\de_i}M_i$ is 
a nonnegatively curved Alexandrov space of dimension $\geq 1$.
If $\dim Y = 1$ then 
$M_i$ is homeomorphic to a space in the conclusion of 
Theorems \ref{1-dim circle} and \ref{1-dim interval}.
If $\dim Y = 2$ and $\partial Y = \emptyset$ then 
$M_i$ is homeomorphic to 
a generalized Seifert fiber space having at most 4 singular fibers.
If $\dim Y = 2$ and $\partial Y \neq \emptyset$ then 
$M_i$ is homeomorphic to 
a space in the conclusion of Theorem \ref{2-dim boundary} with 
at most 4 topologically singular points.
If $\dim Y = 3$ then by Stability Theorem, 
$M_i$ is homeomorphic to $Y$.
In this case, the topology of $Y$ 
is already obtained in Lemma \ref{closed 3-dim}.
\end{proof}

\section{Appendix: $\e$-regular covering of the boundary of an Alexandrov surface} \label{appendix}
Let $X$ be an Alexandrov surface with non-empty compact boundary $\partial X$. 
Let us denote $C$ by a component of $\partial X$.
The purpose of this section is to prove Lemma \ref{lemma regular covering} which state the existence of an $\e$-regular covering of $C$, used in Section \ref{proof of 2-dim boundary}.

We will first prepare a division of $C$ by consecutive arcs $\gamma_1, \gamma_2, \dots, \gamma_n$ with $\partial \gamma_\alpha = \{p_\alpha, p_{\alpha + 1} \}$ and $p_{n + 1} = p_1$.
We next prove that this division makes the desired regular covering $\{B_\alpha, D_\alpha\}_{\alpha = 1, 2, \dots, n}$ of $C$. 



For $\e > 0$, we define 
\[
S_\e (\partial X) := \{p \in \partial X \,|\, L(\Sigma_p) \le \pi - \e \},
\]
where $L(\Sigma_p)$ is the length of $\Sigma_p$.
Note that $S_\e(\partial X)$ is a finite set.
And we set 
\[
R_\e(\partial X) := \partial X - S_\e(\partial X), 
\]
and 
\[
S_\e(C) := S_\e(\partial X) \cap C \text{ and } R_\e(C) := R_\e(\partial X) \cap C.
\]


We review fundamental properties.
\begin{lemma} \label{local regularity of boundary function}
For $\e > 0$ and $p \in R_\e(\partial X)$, there exists $\de > 0$ such that for every $x \in B(p, \de) - \partial X$, we have
\[
|\nabla d_{\partial X}| (x) > \cos \e.
\]
\end{lemma}
\begin{proof}
Suppose the contrary. 
Then, there are a sequence $\de_i \to 0$ and $x_i \in B(p, \de_i) - \partial X$ such that $|\nabla d_{\partial X}|(x_i) \le \cos \e$.
Taking a subsequence, we consider the limit $x_\infty \in B(o_p, 1) \subset T_p X$ of $x_i$ under the convergence $(\frac{1}{\de_i} X, p) \to (T_p X, o_p)$.

If $|\partial T_p X, x_\infty| > 0$, 
then 
\[
|\nabla d_{\partial T_p X}|(x_\infty) > - \cos \left( \pi - \frac{\e}{2} \right) = \cos \left( \frac{\e}{2} \right). 
\]
By the lower-semicontinuity of angles, 
\[
\liminf_{i \to \infty} |\nabla d_{\partial X}|(x_i) \ge |\nabla d_{\partial T_p X}|(x_\infty).
\]
This implies a contradiction.

When $|\partial T_p X, x_\infty| = 0$, 
we take $y_\infty \in B(o_p , 1) - U(\partial T_p X, 1 /2)$
such that 
\[
\frac{|\partial T_p X, y_\infty|}{|x_\infty, y_\infty|} 
= 
\cos \angle x_\infty y_\infty \partial T_p X > \cos \left( \frac{\e}{2} \right).
\]
We take a sequence $y_i \in B(p, \frac{3 \de_i}{2}) - U(\partial X, \frac{\de_i}{4})$ such that $y_i \to y_\infty$ under the convergence $(\frac{1}{\de_i} X, p) \to (T_p X, o_p)$.
Since the distance function $d_{\partial X}$ is $\lambda$-concave for some $\lambda$ on $\mathrm{int} X$, 
\begin{align*}
\frac{|\partial X, y_i| - |\partial X, x_i|}{|x_i, y_i|}
& \le \frac{\lambda}{2} |x_i y_i| + (d_{\partial X})'_{x_i} (\uparrow^{y_i}_{x_i}) \\
& \le \frac{\lambda}{2} |x_i y_i| + |\nabla d_{\partial X}|(x_i).
\end{align*}
Remark that $x_i y_i \subset \mathrm{int}\, X$ (Remark \ref{interior is convex}, later).
It is obvious that
\[
\frac{|\partial X, y_i| - |\partial X, x_i|}{|x_i y_i|} 
\to 
\frac{|\partial T_p X, y_\infty|}{|x_\infty, y_\infty|} \,\,( \text{as } i \to \infty).
\]
Therefore, we conclude 
\[
\cos (\e /2) \le \cos \e
\]
This is a contradiction.
Therefore, we have the conclusion of Lemma.
\end{proof}

\begin{remark} \label{interior is convex} \upshape
The interior of an Alexandrov space is strictly convex.
In fact, let $p, q \in \mathrm{int}\, M$. 
For every $x, y \in \mathrm{int}\, (pq)$ (the relative interior), $\Sigma_x \equiv \Sigma_y$ (\cite{Pet parallel}). 
If $x$ is near $p$ then $x \in \mathrm{int}\, M$, and hence $\partial \Sigma_x = \emptyset$. 
Then $\partial \Sigma_y = \emptyset$.
Therefore, $pq \subset \mathrm{int}\, M$.
\end{remark}

\begin{corollary} \label{global local regularity of boundary function}
For any $\e, s >0$, there is $\de_1 > 0$ such that 
\[
|\nabla d_{\partial X}| > \cos \e 
\]
on $B(\partial X, \de_1) - ( \partial X \cup U(S_\e(\partial X), s))$.
\end{corollary}
\begin{proof}
The proof is provided by Lemma \ref{local regularity of boundary function} and Lebesgue covering lemma.
\end{proof}

\begin{lemma}\label{strained boundary}
For any $\e > 0$, there is $\de_2 > 0$ such that 
\[
B(\partial X, \de_2) - \partial X
\]
is $(2, \e)$-strained.
\end{lemma}
\begin{proof}
For any $p \in \partial X$, there is $\de_p > 0$ such that 
\[
B(p, \de_p) - \{p\}
\]
has no $\e'$-critical point for $d_p$, where, $\e' \ll \e$.
Therefore, $B(p, \de_p) - \partial X$ is $(1, \e')$-strained, and hence, this is $(2, \e)$-strained.
Since $\partial X$ is compact, there is $\de > 0$ such that,
for any $p \in \partial X$, there exists $q \in \partial X$ with $B(p, \de) \subset B(q, \de_q)$.
Therefore, $B(\partial X, \de) - \partial X$ is $(2, \e)$-strained.
\end{proof}


From now on, we use the notation $\wangle(A; B, C)$ defined as follows. 
Let $A$, $B$ and $C$ be positive numbers satisfying a part of triangle inequality: $B + C \ge A$ and $A + C \ge B$.
If $A +B \ge C$, then taking a geodesic triangle $\triangle abc$ in the hyperbolic plane $\mathbb H^2$ with side lengths $|ab| = C$, $|bc| = A$ and $|ca| = B$, 
we set $\wangle(A; B,C) : = \angle b a c$.
Otherwise, $\wangle (A; B, C) := 0$.

Let us start to construct a division of $C \subset \partial X$ to construct an $\e$-regular covering.
Let us fix a small positive number $\e >0$.
\begin{lemma} \label{lemma angle 1}
For any $p \in \partial X$, there is $s > 0$ such that 
for any $q \in B(p, s) \cap \partial X - \{p\}$ and $x \in \widehat{p q} - (\{q\} \cup U(p, |pq| /2))$, we have 
\[
\wangle (|p q|; |p x|, L(\widehat{x q})) > \pi - \e.
\]
Here, $\widehat{p q}$ is an arc joining $p$ and $q$ in $\partial X$.
In particular, 
\[
\wangle p x q > \pi - \e.
\]
\end{lemma}
\begin{proof}
Suppose the contrary. 
Then, there are $p \in \partial X$, $s_i \to 0$, $q_i \in S(p, s_i) \cap \partial X$ and $x_i \in \widehat{p q_i} - (\{q_i\} \cup U(p, |p q_i| /2))$ such that 
\[
\wangle (|p q_i|; |p x_i|, L(\widehat{x_i q_i})) \le \pi - \e.
\]
Taking a subsequence, we may assume that $q_i$, $x_i$ converges to $q_\infty$, $x_\infty$, respectively, under the convergence $(\frac{1}{s_i} X, p) \to (T_p X, o_p)$.
Then, $q_\infty \in \partial T_p X$, $|o_p, q_\infty| = 1$ and $x_\infty \in o_p q_\infty$.

If $x_\infty \neq q_\infty$, then 
\[
\lim_{i \to \infty} \wangle (|p q_i|; |p x_i|, L(\widehat{x_i q_i})) = \wangle o_p x_\infty q_\infty = \pi.
\]
This is a contradiction.

Otherwise, $x_\infty = q_\infty$.
We take $r_\infty \in \partial T_p X$ such that 
\[
q_\infty \in o_p r_\infty, \, |o_p, r_\infty| > 3 /2.
\]
We choose $r_i \in \partial X$ such that $r_i \to r_\infty$ as $i \to \infty$ under the convergence $(\frac{1}{s_i} X, p) \to (T_p X, o_p)$.
Since $\widehat{x_i r_i}$ is a quasigeodesic containing $q_i$, by the comparison theorem for quasigeodesics \cite{PP QG}, we have 
\[
\wangle (|p q_i|; |p x_i|, L(\widehat{x_i q_i})) \ge 
\wangle (|p r_i|; |p x_i|, L(\widehat{x_i r_i})).
\] 
Since $L(\widehat{x_i r_i}) / s_i \to |x_\infty r_\infty|$ (\cite{PP QG}), we obtain
\[
\wangle (|p r_i|; |p x_i|, L(\widehat{x_i r_i})) \to \wangle o_p x_\infty r_\infty = \pi.
\]
This is a contradiction.
\end{proof}

\begin{lemma} \label{lemma angle 2}
For $p \in R_\e(\partial X)$, there is $s > 0$ such that for any $q \in B(p, s) \cap \partial X - \{p\}$ and $x \in \widehat{p q} - \{p, q\}$, we have
\[
\wangle (|p q|; |p x|, L(\widehat{x q})) > \pi -\e \text{ or } 
\wangle (|p q|; L(\widehat{p x}), |x q|)) > \pi -\e.
\]
In particular, $\wangle p x q > \pi -\e$.
\end{lemma}
\begin{proof}
Suppose the contrary. Then there are $p \in \partial X$, $s_i \to 0$, $q_i \in S(p, s_i) \cap \partial X$ and $x_i \in \widehat{p q_i} - \{p, q_i\}$ such that 
\begin{align}
& \wangle (|p q_i|; |p x_i|, L(\widehat{x_i q_i})) \le \pi -\e, \text{ and } \label{angle eq1} \\
& \wangle (|p q_i|; L(\widehat{p x_i}), |x_i q_i|)) \le \pi -\e. \label{angle eq2}
\end{align}
We may assume that $q_i$ and $x_i$ converge to $q_\infty$ and $x_\infty$, respectively, under the convergence $(\frac{1}{s_i} X, p) \to (T_p X, o_p)$.
Then, $q_\infty \in \partial T_p X$, $|o_p q_\infty| = 1$ and $x_\infty \in o_p q_\infty$.

If $q_\infty \neq o_p$, then by the same argument of the proof of Lemma \ref{lemma angle 1}, we have 
\[
\wangle (|p q_i|; |p x_i|, L(\widehat{x_i q_i})) \to \pi.
\]
This is a contradiction to \eqref{angle eq1}.

Otherwise, $q_\infty = o_p$. 
We take $r_\infty \in \partial T_p X \cap S(o_p, 1) - \{q_\infty\}$ and $r_i \in \partial X$ such that $r_i \to r_\infty$.
Since $\widehat{x_i r_i}$ is a quasigeodesic containing $p$, by the comparison theorem for quasigeodesics, we have 
\[
\wangle (|p q_i|; L(\widehat{p x_i}), |x_i q_i|)) \ge 
\wangle (|r_i q_i|; L(\widehat{r_i x_i}), |x_i q_i|)).
\]
Since $L(\widehat{r_i x_i}) / s_i \to |r_\infty o_p|$, we obtain 
\[
\wangle (|r_i q_i|; L(\widehat{r_i x_i}), |x_i q_i|)) \to \wangle q_\infty o_p r_\infty > \pi - \e.
\]
This is a contradiction to \eqref{angle eq2}.
\end{proof}

\begin{definition} \upshape
Let $\gamma = \widehat{pq}$ be an arc joining $p$ and $q$ in $\partial X$.
We say that $\gamma$ is {\it strictly $\e$-strained by $\partial \gamma = \{p ,q\}$} if 
\begin{align} \label{angle eq3}
\wangle p x q > \pi - \e \text{ for all } x \in \mathrm{int}\, \gamma,
\end{align}
and if, setting $\xi$ and $\eta$ the directions of quasigeodesics $\widehat{x p}$ and $\widehat{x q}$ at $x$, respectively, we have
\begin{align}\label{angle eq4}
\angle (\xi, \uparrow_x^p) < \e \text{ and } \angle (\eta, \uparrow_x^q) < \e.
\end{align}
\end{definition}

Remark that an arc $\widehat{p q}$ in Lemma \ref{lemma angle 2} is strictly $\e$-strained by $\{p, q\}$.
Indeed, we assume that $\wangle (|p q|; |p x|, L(\widehat{x q})) > \pi -\e$ for some $x \in \mathrm{int}\, \widehat{p q}$.
We obtain $\wangle p x q \ge \wangle (|p q|; |p x|, L(\widehat{x q})) > \pi -\e$.
Let $\xi$ and $\eta$ be the directions of $\widehat{x p}$ and $\widehat{x q}$ at $x$, respectively.
Since $\dim X = 2$, $\xi$ and $\eta$ attain the diameter of $\Sigma_x$, i.e.,
\[
\angle (\xi, \eta) = L(\Sigma_x).
\]
Hence, we have 
\begin{align*}
\angle (\xi, \eta) &= 
\angle (\xi, \uparrow_x^p) + \angle (\uparrow_x^p, \eta) \\
& = 
\angle (\xi, \uparrow_x^p) + \angle (\uparrow_x^p, \uparrow_x^q) + \angle (\uparrow_x^q, \eta) \\
&\ge \angle (\uparrow_x^p, \uparrow_x^q) 
\ge \wangle p x q > \pi - \e.
\end{align*}
Since $L(\Sigma_p) \le \pi$, we obtain
\[
\angle (\xi, \uparrow_x^p) + \angle (\uparrow_x^q, \eta) < \e.
\]
In particular, \eqref{angle eq4} holds.

Let us fix a component $C$ of $\partial X$.
By Lemma \ref{lemma angle 1} and $\sharp S_\e (C) < \infty$, there is $s > 0$ such that for every $p \in S_\e (C)$, taking $q^+, q^- \in S(p, s) \cap C$, we have 
\begin{align*}
\wangle (|p q^\pm|; |p x|, L(\widehat{x q^\pm})) > \pi -\e 
\end{align*}
for all $x \in \beta_p^\pm - (U(p, s /2) \cup \{q^\pm\})$, where $\beta_p^\pm := \widehat{p q^\pm}$.

Let us consider the set
\begin{equation} \label{division of C eq1}
C - U(S_\e(C), s) = C - \bigcup_{p \in S_\e (C)} \mathrm{int}\, (\beta_p^+ \cup \beta_p^-).
\end{equation}
This consists of finitely many arcs.
We prove that each component $K$ of it is divided into finitely many strictly $\e$-strained arcs.
\begin{lemma} \label{division of K}
Let $K$ be an arc in $R_\e(C)$ with $\partial K = \{p, q\}$.
There are consecutive arcs $\gamma_\alpha = \widehat{p_\alpha p_{\alpha + 1}}$, $\alpha =1, 2, \dots, n$ with $p_1 = p$ and $p_{n + 1} = q$ such that each $\gamma_\alpha$ is strictly $\e$-strained by $\{p_\alpha, p_{\alpha + 1}\}$.
\end{lemma}
\begin{proof}
By repeatedly using Lemma \ref{lemma angle 2}, we have a set $\Phi$ of consecutive arcs starting from $p$ contained in $K$, 
\[
\Phi = \{ \gamma_1, \gamma_2, \dots, \gamma_n \}
\]
such that each $\gamma_\alpha$ is strictly $\e$-strained by $\partial \gamma_\alpha$.
Here, a word ``consecutive arcs starting from $p$'' means that each $\gamma_\alpha$ forms $\gamma_\alpha = \widehat{p_\alpha p_{\alpha + 1}} \subset K$ and $p_1 = p$.

In what follows, $\Phi$ denotes any such finite sequence of arcs as above.
Let us set
\[
L(\Phi) := \sum_{\alpha =1}^n L(\gamma_\alpha). 
\]
We consider the value $\ell := \sup_\Phi L(\Phi)$.
Since $\gamma_\alpha$ are consecutive and contained in $K$, we have $\ell \le L(K)$.
To prove the lemma, we show that there exists $\Phi$ with $L(\Phi) = L(K)$.

If $\ell = L(K)$ then there is $\Phi = \{\gamma_\alpha \}_{1 \le \alpha \le n}$ such that $p_n$ is arbitrary close to $q$. 
If there is $\Phi$ with $L(\Phi) = L(K)$, then the proof is done. Otherwise, by using Lemma \ref{lemma angle 2} for $q$, we can take $\Phi$ such that $\gamma_{n+1} := \widehat{p_n q}$ is strictly $\e$-strained by $\partial \gamma_{n + 1}$. 
Then we obtain an extension 
\[
\tilde \Phi := \Phi \cup \{ \gamma_{n+1} \}
\]
of $\Phi$ with $L(\tilde \Phi) = L(K)$. This is a contradiction. 
Therefore, if $\ell = L(K)$ then there is $\Phi$ attaining $L(\Phi) = L(K)$.

We assume that $\ell < L(K)$. 
By a similar argument as above, we have $\Phi = \{\gamma_\alpha\}_{1 \le \alpha \le n}$ such that $L(\Phi) = \ell$.
Again, by a similar argument as above, we have an extension $\tilde \Phi$ of $\Phi$. 
Hence $L(\tilde \Phi) > \ell$. 
This is a contradiction.
\end{proof}

By Lemma \ref{division of K} and the decomposition \eqref{division of C eq1}, we obtain a division of $C$:
\begin{equation} \label{nice division}
C = 
\left( \bigcup_{p \in S_\e(C)} \beta_p^+ \cup \beta_p^- \right) \cup 
\left( \bigcup_{K}\, \bigcup_{i = 1}^{n_K} \gamma_\alpha^K \right),
\end{equation}
where, $\beta_p^\pm := \widehat{p q^\pm}$ and $K$ denotes any arc component of $C - U(S_\e(C), s)$.
For each $K$, $\gamma_\alpha^K$ ($1 \le \alpha \le n_K$) expresses a strictly $\e$-strained arc by $\partial \gamma_\alpha^K$, obtained in Lemma \ref{division of K}.


By using a division \eqref{nice division} of $C$, 
we prove that the existence of an $\e$-regular covering of $C$.

\begin{lemma} \label{lemma regular covering}
There is an $\e$-regular covering of $C$.
\end{lemma}
\begin{proof}
Let us fix a division of $C$ obtained as \eqref{nice division}.
Fixing a component $K$, we write $n = n_K$, $\gamma_\alpha = \gamma_\alpha^K$.
Each $\gamma_\alpha$ forms $\gamma_\alpha = \widehat{p_\alpha p_{\alpha + 1}}$.
We take a small positive number $r$ such that 
\begin{align}
& |\nabla d_{p_\alpha}| >1 -\e \text{ on } B(p_\alpha, 2 r) - \{p_\alpha \} \text{ for all } \alpha,  \label{reg cov eq1} \\
& B_\alpha \cap B_{\alpha'} = \emptyset \text{ for all } \alpha \neq \alpha', \label{reg cov eq2}
\end{align}
where $B_\alpha := B(p_\alpha, r)$.

By the condition \eqref{angle eq3}, there is a small positive number $\de$ with $\de \ll r$ such that, setting 
\begin{equation*} 
D_\alpha := B(\gamma_\alpha, \de) - \mathrm{int}\, (B_\alpha \cup B_{\alpha + 1}),
\end{equation*}
we have
\begin{equation*} 
\wangle p_\alpha x p_{\alpha + 1} > \pi - \e
\end{equation*}
for all $x \in D_\alpha$; 
further, by \eqref{angle eq3} and \eqref{angle eq4}, $\de$ can be chosen  that, for every $x \in D_\alpha$ and $y \in C$ with $|x C| = |x y|$, we have
\begin{equation*} 
| \angle p_\alpha x y - \pi /2 | < 2 \e \text{ and }
| \angle p_{\alpha + 1} x y - \pi /2 | < 2 \e. 
\end{equation*}
To use later, we set 
\[
\Phi_K := \{B_\alpha \}_{1 \le \alpha \le n} \cup \{D_\alpha \}_{1 \le \alpha \le n -1}.
\]

For $p \in S_\e(C)$, there are unique components $K^+$ and $K^-$ of $C - U(S_\e(C, s))$ with $\beta_p^\pm \cap K^\pm \neq \emptyset$. 
We take unique elements $q^\pm \in \beta_p^\pm \cap K^\pm$.
Recall that $s > 0$ is a small positive number satisfying the conclusion of Lemma \ref{lemma angle 1} for $p$, and $|\nabla d_p| > 1- \e$ on $B(p, s) - \{p\}$.
For $q^\pm \in K^\pm$, we provided numbers $r^\pm$ satisfying \eqref{reg cov eq1} and \eqref{reg cov eq2}, above.
Let us set 
\begin{equation*}
B_p := B(p, s /2) \text{ and } D_p^\pm := B(\beta_p^\pm, \de) - \mathrm{int}\, (B_p \cup B(q^\pm, r^\pm)).
\end{equation*}
If we retake $\de$ small enough, we have
\begin{equation*}
|\angle p x q^\pm - \pi /2 | < \e
\end{equation*}
for all $x \in D_p^\pm$.

Thus, we obtain an $\e$-regular covering 
\[
\{ B_p, D_p^\pm \}_{p \in S_\e(C)} \cup \bigcup_K \Phi_K
\]
of $C$.
\end{proof}



\begin{thebibliography}{99999999}
\bibitem[BBI]{BBI} D.~Burago, Yu.~Burago, and S.~Ivanov.
A course in metric geometry.
Graduate Studies in Mathematics, 33. 
American Mathematical Society, Providence, RI, 2001. xiv+415 pp.
ISBN: 0-8218-2129-6

\bibitem[BGP]{BGP} Yu.~Burago, M.~Gromov, and G.~Perel'man.
A. D. Aleksandrov spaces with curvatures bounded below, 
Uspekhi Mat. Nauk 47 (1992), no. 2(284), 3--51, 222, 
translation in Russian Math. Surveys 47 (1992), no. 2, 1--58

\bibitem[BH]{BH}M.~Bridson and A.~Haefliger, Metric spaces of non-positive curvature, 
vol. 319 of Grundlehren der Mathematischen Wissenschaften, Springer-Verlag, 1999.


\bibitem[CaGe]{CaGe} J.~Cao and J.~Ge.
A simple proof of Perelman's collapsing theorem for 3-manifolds.
J. Geom. Anal. 
Volume 21, Number 4, 807--869, DOI: 10.1007/s12220-010-9169-5

\bibitem[ChGr]{CG} J.~Cheeger and D.~Gromoll. 
On the structure of complete manifolds of nonnegative
curvature, Ann. of Math. 96 (1972) 413--443.

\bibitem[E]{Epstein} D.~B.~A.~Epstein. 
Curves on 2-manifolds and isotopies.
Acta Math. 115 (1966) 83--107.

\bibitem[FY]{FY} K.~Fukaya and T.~Yamaguchi.
The fundamental groups of almost nonnegatively curved manifolds, 
Ann. of Math. 136 (1992) 253--333.

\bibitem[GP]{GP} K.~Grove and P.~Petersen.
A radius sphere theorem, 
Invent. Math. 112 (1993), no. 3, 577--83. 

\bibitem[Kap Rest]{Kap restriction}V.~Kapovitch. 
Restrictions on collapsing with a lower sectional curvature bound. 
Math. Z. 249 (2005), no. 3, 519--539.

\bibitem[Kap Stab]{Kap stab} V.~Kapovitch. Perelman's stability theorem,
Surveys of Differential Geometry, Metric and Comparison Geometry, vol. XI, 
International press, (2007), 103--136.

\bibitem[KL]{KL} B.~Kleiner and J.~Lott. 
Locally collapsed 3-manifolds.
preprint.

\bibitem[KL orb]{KL orb} B.~Kleiner and J.~Lott. 
Geometrization of three-dimensional orbifolds via Ricci flow.
preprint.

\bibitem[KMS]{KMS} K.~Kuwae, Y.~Machigashira and T.~Shioya.
Sobolev spaces, Laplacian, and heat kernel on Alexandrov spaces, 
Math. Z. 238 (2001), no. 2, 269--316.

\bibitem[L]{Lyt}A.~Lytchak. 
Differentiation in metric spaces. 
Algebra i Analiz 16 (2004), no. 6, 128--161;
translation in St. Petersburg Math. J. 16 (2005), no. 6, 1017--1041 

\bibitem[Milka]{Milka} A.~D.~Milka. 
Metric structure of a certain class of spaces that contain straight lines. (Russian) 
Ukrain. Geometr. Sb. Vyp. 4 1967 43--48. 

\bibitem[M]{Mitsuishi} A.~Mitsuishi. 
A splitting theorem for infinite dimensional Alexandrov spaces with nonnegative curvature and its applications. 
Geom. Dedicata 144 (2010), 101--114.


\bibitem[MT]{MT}
J.~Morgan and G.~Tian. 
Completion of the proof of the geometrization conjecture.
preprint.

\bibitem[N]{Natsheh} 
M.~A.~Natsheh. 
On the involutions of closed surfaces, 
Arch. Math., Vol. 46, 179--183 (1986)

\bibitem[O]{Otsu} Y.~Otsu. On manifolds of small excess, 
Amer. J. Math. 115 (1993), no. 6, 1229--1280.

\bibitem[OS]{OS} Y.~Otsu and T.~Shioya.
The Riemannian structure of Alexandrov spaces, 
J. Differential Geom. 39 (1994), no. 3, 629--658

\bibitem[Per II]{Per Alex II} G.~Perelman. 
A. D. Alexandrov's spaces with curvatures bounded from below. II, Preprint.

\bibitem[Per Elem]{Per Morse} G.~Perelman. 
Elements of Morse theory on Alexandrov spaces, 
St. Petersburg Math. J. 5 (1994) 207--214.

\bibitem[Per DC]{Per DC} G.~Perelman.
DC Structure on Alexandrov Space. Preprint.


\bibitem[Per Ent]{Pr:entropy}  G.~ Perelman,
   {\it The entropy formula for the {R}icci flow and its geometric 
         applications},
   arXiv:math. DG / 0211159.

\bibitem[Per Surg]{Pr:surgery}  G.~ Perelman,
   {\it Ricci flow with surgery on three-manifolds},
   arXiv:math. DG / 0303109.



\bibitem[PP QG]{PP QG} G.~Perelman and A.~Petrunin.
Quasigeodesics and gradient curves in Alexandrov spaces, preprint.

\bibitem[Pet QG]{Pet QG} A.~Petrunin. 
Quasigeodesics in multidimensional Alexandrov spaces.
Thesis (Ph.D.)-University of Illinois at Urbana-Champaign. 1995. 92 pp.

\bibitem[Pet Appl]{Pet Appl} A.~Petrunin.
Applications of quasigeodesics and gradient curves. 
(English summary) Comparison geometry (Berkeley, CA, 1993--94), 203--219,
Math. Sci. Res. Inst. Publ., 30, Cambridge Univ. Press, Cambridge, 1997. 

\bibitem[Pet Para]{Pet parallel} A.~Petrunin. 
Parallel transportation for Alexandrov space with curvature bounded below. 
Geom. Funct. Anal. 8 (1998), no. 1, 123--148.

\bibitem[Pet Semi]{Pet Semi} A.~Petrunin. 
Semiconcave functions in Alexandrov's geometry. 
Surv. Differ. Geom., Vol. XI,  137--201, Int. Press, Somerville, MA, 2007.
        
\bibitem[Pl]{Plaut} C.~Plaut. 
Spaces of Wald-Berstovskii curvature bounded below, J. Geom. Anal. 6, 113--134 (1996)

\bibitem[SY00]{SY} T.~Shioya and T.~Yamaguchi. 
Collapsing three-manifolds under a lower curvature bound. 
J. Differential Geom. 56, 1--66 (2000)

\bibitem[SY05]{SY vol collapse}
T.~Shioya and T.~Yamaguchi. 
Volume collapsed three-manifolds with a lower curvature bound.
Math. Ann., 333 (1), (2005), 131--155.

\bibitem[Sie]{Sie}L.~C.~Seibenmann. 
Deformation of homeomorphisms on stratified sets. I, II. 
Comment. Math. Helv. 47 (1972), 123--136; ibid. 47 (1972), 137--163.

\bibitem[W]{Wo} J. A. Wolf, 
Spaces of constant curvature, 
Publish or Pelish, 1984.

\bibitem[Y91]{Yam collapsing and pinching} T.~Yamaguchi.
Collapsing and Pinching Under a Lower Curvature Bound. 
Ann. of Math. 133 (1991) 317--357.

\bibitem[Y conv]{Y convergence} T.~Yamaguchi.
A convergence theorem in the geometry of Alexandrov spaces. 
Actes de la Table Ronde de Geometrie Differentielle (Luminy, 1992), 
Semin. Congr., vol. 1, Soc. Math. France, Paris, 1996, pp. 601--642

\bibitem[Y 4-dim]{Y 4-dim} T.~Yamaguchi, 
Collapsing 4-manifolds under a lower curvature bound. 
arXiv:1205.0323v1 [math.DG]

\bibitem[Y ess]{Y essential} T.~Yamaguchi. Collapsing and essential covering. 
arXiv:1205.0441v1 [math.DG]
\end{thebibliography}
\end{document}